\documentclass[12pt,twoside,reqno]{amsart}

\usepackage{amssymb}
\usepackage[T1]{fontenc}
\usepackage[hidelinks]{hyperref}

\providecommand{\MR}[1]{MR #1}

\def\to{\rightarrow}

\def\Ra{{\Rightarrow}}
\def\Lra{{\Leftrightarrow}}

\def\ti{\tilde}
\def\pa{\partial}
\def\bs{\bigskip}\def\ss{\smallskip}\def\ms{\medskip}
\def\no{\noindent}

\def\const{\text{\rm const}}

\def\ti{\tilde}

\def\ker{\text{\rm ker}}

\def\dist{\text{\rm dist}}
\def\supp{\text{\rm supp}\,}
\def\to{\rightarrow}

\def\pa{\partial}

\def\sign{\text{\rm sign}}

\def\Sc{\text{\rm Schr}}

\def\const{\text{\rm const}}

\def\sign{\text{\rm sign}}

\def\ker{\text{\rm ker}}

\def\dist{\text{\rm dist}}
\def\supp{\text{\rm supp}\,}

\def\bs{\bigskip}
\def\ss{\smallskip}
\def\ms{\medskip}
\def\no{\noindent}

\def\R{{\mathbb R}}

\def\Z{{\mathbb{Z}}}
\def\C{{\mathbb{C}}}
\def\N{{\mathbb{N}}}

\def\RR{{\mathcal{R}}}

\def\DD{{\mathcal{D}}}

\def\BB{{\mathcal B}}
\def\MM{{\mathcal M}}
\def\NN{{\mathcal N}}

\def\FF{{\mathcal F}}
\def\SS{{\mathcal S}}
\def\HH{{\mathcal H}}
\def\GG{{{{\mathbf{G}}}}}
\def\EE{{\mathcal E}}
\def\WW{{\mathcal W}}

\def\e{\varepsilon}

\def\L{\Lambda}
\def\l{\lambda}
\def\G{\Gamma}
\def\lan{\lambda_n}

\newcommand{\charf}{\raisebox{\depth}{\(\chi\)}}
\def\nlhat{\overset{\ \curlywedge}}

\theoremstyle{plain}
\newtheorem{lemma}{Lemma}
\newtheorem{theorem}{Theorem}
\newtheorem{corollary}{Corollary}
\newtheorem{proposition}{Proposition}

\newtheorem{remark}{Remark}

\newtheorem{example}{Example}

\newtheorem*{lem}{Lemma}
\newtheorem*{thm}{Theorem}
\newtheorem*{cor}{Corollary}
\newtheorem*{prop}{Proposition}

\numberwithin{equation}{section}

\newcommand{\lecture}[2]{%
  \setcounter{section}{#1}%
  \setcounter{subsection}{0}%
  \setcounter{equation}{0}%
  \section*{\texorpdfstring{\textbf{Lecture #1: #2}}{Lecture #1: #2}}%
}
\newcommand{\lsection}[1]{%
  \subsection*{#1}%
}

\author{A.~Poltoratski}
\address{University of Wisconsin\\ Department of Mathematics\\ Van Vleck Hall\\
480 Lincoln Drive\\
Madison, WI  53706\\ USA }
\email{poltoratski@wisc.edu}
\thanks{The author was partially supported by
NSF Grant DMS-2244801.}

\title[Completeness, spectral and scattering problems]{Complex Analysis in completeness, spectral and scattering problems}

\begin{document}

\begin{abstract}
These lectures are aimed at junior researchers seeking to enter the field of
applications of complex analysis to harmonic analysis, Fourier analysis, spectral and
scattering theory for differential operators and related fields. Starting with classical
completeness problems posed by Bernstein, Beurling and Malliavin, Wiener, Kolmogorov and
Krein, the course progresses to the solutions of the Gap and Type Problems, inverse
spectral problems for Schr\"odinger operators and canonical Hamiltonian systems, and
pointwise convergence of the non-linear Fourier transform. Each lecture is supplemented
with exercises and references.
\end{abstract}
	
	\maketitle
	
	\tableofcontents

\newpage

\section*{\texorpdfstring{\textbf{Preface}}{Preface}}

\ms\no These notes originate from the lectures I gave at the internet seminar organized
by Brett Wick at Georgia Tech in the spring of 2013. Later in the same year I gave a
10-hour course on the same set of topics in the CBMS conference series at Clemson
University, organized by Constanze Liaw, Mishko Mitkovski and Brett Wick. In the next
decade I was happy to give a number of survey lectures in the same area, including a
one-semester course at Brown University, a 10-hour minicourse at the University of
Toulouse, a 5-hour minicourse at the University of Helsinki, two minicourses at the CRM,
Barcelona, and a shorter 3-hour minicourse at the Fields Institute; I have also prepared a
3-hour minicourse to be given in Milan later this year. Hence, I decided to include all
this material in the updated notes. I am grateful to the organizers of these events --
Sergei Treil (Brown), Stefanie Petermichl and Pascal Thomas (Toulouse), Eero Saksman
(Helsinki), Joaquim Ortega-Cerd\`a (CRM, Barcelona), Javad Mashreghi and Ilia Binder (the
Fields Institute), Carlo Bellavita, Mattia Calzi and Marco Peloso (Milan), among others --
for their kind invitations and hospitality.

\ms\no The text inherits the format of the original course: each lecture is intended to be
read in one sitting; proofs are included when they are short and instructive and replaced
with references when they are not; and every lecture ends with exercises, ranging from
elementary to nearly impossible, and with its own list of references.

\ms\no In the present version, I have made revisions and updates to include the results
which appeared since; in particular, the last six lectures present developments of the past
decade and bring the reader to the edge of current research.

\ms\no The unifying theme of the course is that all of its methods are based on complex
analysis. Even though such basic objects of complex analysis as entire functions, inner
functions, Hardy spaces, Dirichlet spaces, etc., may not be directly mentioned in every
lecture, they are the tool box, which the reader can find in the references given. A
standard graduate background in real and complex analysis is all that is needed to begin.

\ms\no The problems studied in these notes are easy to state. Can every continuous
function be approximated by polynomials in a weighted norm on the line? How many complex
exponentials does one need to span $L^2$ on an interval? Does the spectrum determine the
operator? Questions like these can be explained to an undergraduate, yet they occupied
many of the strongest analysts of the twentieth century, and some of them are still open
after nearly a hundred years.

\ms\no What binds these problems together is the Uncertainty Principle of
Harmonic Analysis -- the impossibility for a function and its Fourier transform to be
simultaneously small -- and what makes them approachable is the rigidity of analytic
functions: an entire function cannot vanish too often, and a function analytic in a
half-plane is determined by a fragment of its boundary data.

\ms\no Perhaps the greatest pleasure this field offers is watching a single idea travel
across mathematics. A family of exponentials becomes the zero set of an entire function
of exponential type; a Schr\"odinger operator becomes a meromorphic inner function; a
spectral measure becomes a chain of de Branges spaces; a scattering coefficient becomes
an outer function. A question posed in the language of one area is answered in the
language of another, and the translations are not tricks but permanent bridges, open for
two-way traffic. The reader will cross them many times in this course: the same densities
of sequences, the same Poisson-summable logarithms, the same Toeplitz kernels will
reappear in problems that at first sight have nothing in common.

\ms\no The field is far from finished. Several of its classical problems, presented in
the first half of the course, were solved only in the recent decades, by newly found
combinations of classical tools simple enough to be explained here in full; others remain
wide open. A young researcher could
hardly ask for a better combination: questions of proven depth, a tool box of classical
beauty, and open territory in every direction. If these lectures persuade a few readers
to enter this field, they will have served their purpose.

\newpage

\lecture{1}{Three classical problems}

\ms\no A set of vectors in a Banach or Hilbert space is called complete if finite linear combinations of its vectors are dense in the corresponding
space with respect to the standard topology generated by the norm. For example, any basis, or any set containing a basis, is a complete set.
One of the first advanced examples of a complete set we see in the standard analysis course is the set of monomials $1, x, x^2, x^3,...$ in the
space $C([a,b])$ of all continuous functions on a closed interval with the standard supremum norm. The Weierstrass theorem tells us that
the set is complete in the space.

\ms\no Completeness problems appear in many areas of analysis and its applications. For instance, a function on a subset of the real line may represent a wave and one may ask if it can be approximated, in a specified sense, by linear combinations of specially selected functions, often called harmonics.
In modern terms, defining the criterion of approximation amounts to defining the norm in the space. Verifying whether any function in the space
can be approximated by finite linear combinations of harmonics is equivalent to proving that harmonics are complete in the space.
Problems of this kind gave name to a large and important part of mathematics, Harmonic Analysis.

\ms\no The role of harmonics in the above set-up can be played by a number of different sets of functions, such as trigonometric functions, monomials, complex exponentials, or specials functions such as Bessel functions, Jacobi or Chebyshev polynomials or
Airy functions originating from Physics. The most common choices for the Banach space are $L^p$ spaces and spaces of continuous or smooth functions
with various norms.

\ms\no  Let us start by formulating some of the classical completeness problems that will be discussed in this course.

\newpage

\lsection{Bernstein's problem}

\ms\no  Let $W:\R\to [1,\infty)$ be a continuous function satisfying $x^n=o(W(x))$ for any $n\in\N$, as $x\to\pm\infty$.
Denote by $C_W$ the space of all continuous functions $f$ on $\R$ such that $f/W\to 0$ as $x\to\pm\infty$ with the norm
\begin{equation}||f||_W=\sup_\R\frac{|f|}W.\label{norm1}\end{equation}
The famous weighted approximation problem posed by Sergei Bernstein in 1924 \cite{BernsteinL1}
asks to describe the weights $W$ such that polynomials are dense in $C_W$.

\ms\no Bernstein's problem can be viewed as a natural consequence of the Weierstrass theorem. If instead of $C([a,b])$ one tries to consider $C(\R)$ with the same supremum norm, one immediately runs into an obstacle: polynomials do not belong to such a space. Bernstein's weighted sup norm turns out to be
the best way to remedy that situation. A slightly more general version of Bernstein's problem allows the weight $W$ to be semicontinuous and take
infinite values. This extension allows one to study polynomial approximation on arbitrary closed subsets of $\R$.

\ms\no Throughout the 20th century Bernstein's problem was investigated by many prominent analysts including N. Akhiezer, L. de Branges, L. Carleson, T. Hall, P. Koosis, B. Levin, P. Malliavin, S. Mandelbrojt, S. Mergelyan, H. Pollard and M. Riesz. This activity continues to this day with more recent significant contributions
by A. Bakan,  M. Benedicks, A. Borichev, P. Koosis, M. Sodin and P. Yuditski, among others.
Besides the natural beauty of the original question, such an extensive interest towards Bernstein's problem is generated by numerous links with adjacent fields, including its close relation with the moment problem.

\ms\no
Further information and references on the remarkable history of Bernstein's problem can be found
in two classical surveys by Akhiezer \cite{AkhiezerL1} and Mergelyan \cite{Mergelyan}, a more recent one by  Lubinsky \cite{LubinskyL1}, or
in the first volume of Koosis' book \cite{KoosisL1}.

\ms\no A simple if and only if solution in Bernstein's problem most likely does not exist. We plan to discuss some of the classical conditions along with recent progress in these lectures.

\newpage

\lsection{The Beurling-Malliavin problem}

\ms\no Let $\L$ be a subset of the complex plane. Denote by $E_\L$ the set of complex exponential functions
with 'frequencies' from $\L$:
$$\EE_\L=\{\exp(2\pi i\l t)|\ \l\in \L\}.$$

\ms\no  The most standard example of such a set of functions is
$$E_\Z=\{e^{2\pi int} \}_{n\in\Z},$$
which forms an orthonormal basis in $L^2([0,1])$. A natural extension of this important example is the following question.

\ms\no Consider the case when $\L=\{\lan\}$ is a general sequence of complex numbers. Under what conditions on the frequencies $\L$ will
the system of exponentials $\EE_{\L}$ be complete in $L^2(0,a)$?

\ms\no This natural question occupied analysts for many decades. The history of its solution can be started with a theorem by Paley
and Wiener (1935) that says that if a real sequence $\L$ has upper density greater than $a$, i.e.
$$\limsup_{x\to\infty}\frac{\#(\L\cap (0,x))}x>a,$$
then $E_\L$ is complete in $L^2(0,a)$.

\ms\no An immediate question is whether the statement of the theorem can be reversed. Examples by Levinson, Kahane and Koosis showed that
no converse statement can be formulated using the primitive upper density, utilized by Paley and Wiener.

\ms\no A solution to that problem was obtained by Beurling and Malliavin in a series of papers in the early 1960's. Instead of the upper density,
they defined the so-called effective density of a sequence that allows one to replace an inequality in the Paley-Wiener result with an equation (after some additional reductions, that will be discussed in the future lectures.)

\ms\no To obtain the formula for the 'completeness radius' of a sequence of exponentials Beurling and Malliavin proved three intermediate results that are now known as the first BM theorem, the little multiplier theorem and the big multiplier theorem. Each of these results has independent value and usage. Together these three theorems and the final result, the second BM theorem, form the so-called Beurling-Malliavin theory, that is considered to be one of the deepest parts of the 20th century Harmonic Analysis. Modern treatments of the BM theory
can be found in \cite{HJL1, KoosisL1} along with further references.

\newpage

\lsection{The Type Problem}

\ms\no Consider a family $\EE_a=\EE_{[0,a]}$ of exponential functions whose frequencies belong to the interval
from 0 to $a$.
If $\mu$ is a finite positive measure on $\R$ we denote by $T_\mu$ its exponential type that is defined as
\begin{equation} T_\mu=\inf\{\ a>0\ |\ \EE_a  \textrm{ is complete in }  L^2(\mu)\ \}\label{type1}
\end{equation}
if the set of such $a$ is non-empty and as infinity otherwise. The type problem asks to calculate $T_\mu$ in terms
of $\mu$.

\ms\no This question first appears in the work of Wiener, Kolmogorov and Krein  in the context of stationary Gaussian processes that play an important role in Probability Theory (see
\cite{Krein1L1, Krein2L1} or the book by Dym and McKean \cite{DML1}). If $\mu$ is a spectral measure of a stationary Gaussian process, completeness of $\EE_a$ in $L^2(\mu)$
is equivalent to the property that  the process at any time is determined by the data for the time
period from 0 to $a$. Hence the type of the measure is the minimal length of the period of observation necessary to predict the rest of the process. Since any even measure is a spectral measure of a stationary Gaussian process, and vice versa, this reformulation is practically equivalent.
Important connections with spectral theory of second order differential operators were studied by Gelfand and Levitan \cite{GL} and Krein \cite{Krein2L1, Krein3}.

\ms\no For more on the history and connections of the type problem see, for instance,  a note by Dym \cite{DymL1} or a paper by Borichev and Sodin \cite{BSL1}.

\newpage

\lsection{Questions and exercises}

\ms\no Let $\L=\{\lan\}$ be a real sequence. Define its radius of completeness as
$$R(\L)=\sup\{ a\ |\ \EE_\L \textrm{ is complete in }L^2(0,a)\}.$$
Main example:

$$R(\Z)=1$$
(since $\EE_\Z$ is an orthonormal basis of $L^2(0,1)$).
Now let us consider one half of $\Z$, the set $2\Z$ of all even integers.
Show that $R(2\Z)$ is $1/2$. More generally, if we take $1/n$-th of $\Z$, $n\Z$, then
the radius decreases correspondingly: $R(n\Z)=\frac 1n$.

\ms\no \textbf{Question:} Let us take one-half of $\Z$ in a different way and see what happens. Find $R(\N)$, the radius
of completeness of the set of positive integers (use the Paley-Wiener theorem; the answer may surprise you).

\ms\no \textbf{Question:} May $R(\L)$ change if one deletes finitely many points from $\L$? Give an example of an infinite sequence $A\subset \Z$
such that $R(\Z\setminus A)=R(\Z)=1$.

\ms\no \textbf{Question:} Can $A$ from the previous example have positive upper density?

\ms\no \textbf{Question:} Does a complete set have to contain a basis?
A basis in a Banach space is a set $B$ of vectors such that every vector in the space can
be uniquely represented as a series of constant multiples of vectors from $B$, converging in the norm.
(Hint: consider, for instance, the set of monomials in $C([0,1])$. Try to describe the set of functions
representable as uniformly convergent power series.)

\newpage

\lecture{2}{The Beurling-Malliavin theory}

\lsection{The radius of completeness}

\ms\no Let us recall one of the classical completeness problems discussed in the first lecture. Let $\L=\{\l_n\}$
be a sequence of distinct points in the complex plane and let
$$E_\L=\{e^{i2\pi\l_n x}\}$$
be a sequence of complex exponential functions on $\R$ with frequencies from $\L$. We ask under what conditions
on $\L$ will $E_\L$ be complete in $L^2([0,a])$.

\ms\no In this general form the problem does not have a reasonable answer. To formulate the results, including the famous Beurling-Malliavin
theorem mentioned in the first lecture, we will have to refine the formulation of the problem. Recall that  for any complex sequence
$\L$ its
radius of completeness is defined as
$$R(\L)=\sup\{ a\ |\ E_\L \textrm{ is complete in }L^2(0,a)\}.$$
A more realistic goal is to find a formula for $R(\L)$ for an arbitrary $\L$. That goal was accomplished by Beurling and Malliavin and we will
state their result in this lecture.

\ms\no It is well-known in the theory of completeness that the general problem can be easily reduced to the case of real sequences $\L$.
More precisely, if $\L$ is a general complex sequence then $E_\L$ is complete in $L^2([0,a])$ if and only if
$E_{\L'}$ is complete in $L^2([0,a])$, where $\L'$ is the real sequence defined as $\lan'=1/\Re \frac 1\lan$, see for instance
\cite{KoosisL2}.
Also, as will be explained below, one can always assume that $\L$ is a discrete sequence, i.e. has no finite accumulation points.

\lsection{Fourier transform and Paley-Wiener spaces}

\ms\no One of the main tools in the theory of completeness of complex exponentials is the Fourier transform. Recall that
if $f\in L^2(\R)$ then its Fourier transform, $\hat f$ is defined as
$$\hat f(z)=\int_\R e^{-2\pi izx}f(x)dx.$$
The Paley-Wiener theorem says that if $\supp f\subset [-a,a]$ then $\hat f(z)$ is an entire function of exponential type at most $2\pi a$, i.e.
$$|\hat f(z)|\leq \const e^{2\pi a|z|},$$
and $\hat f(x)\in L^2(\R)$. By Parseval's theorem
$$||f||_{L^2([-a,a])}=||\hat f||_{L^2(\R)}.$$
Moreover, every entire function of exponential type at most $2\pi a$ that belongs to $L^2(\R)$ is the Fourier transform
of a function from $L^2([-a,a])$.
The image of $L^2([-a,a])$ under the Fourier transform is the so-called Paley-Wiener space of entire functions, $PW_a$.
More directly, $PW_a$ is a space of entire functions of exponential type at most $2\pi a$ that belong to $L^2(\R)$.

\ms\no A system of vectors in a Hilbert space is incomplete if and only if there exists a non-zero vector orthogonal to all of the vectors of the system.
In particular, a system of complex exponentials $E_\L$ is incomplete in $L^2([0,a])$ if and only if there exists a non-zero $f\in L^2([0,a])$
such that $f\perp e^{i2\pi\l_n  x}$ for all $\lan\in\L$. By the definition of the Fourier transform, the last condition is equivalent
to the condition that $\hat f(\lan)=0$ for all $\lan\in\L$.

\ms\no Thus, via the Fourier transform and the Paley-Wiener theorem, the completeness problem we are discussing becomes a problem of complex analysis. Namely, $E_\L$ is incomplete in
$L^2([0,2a])$ if and only if $\L$ is a zero set of a non-zero function from $PW_a$.

\ms\no Let us point out one immediate consequence of this connection: If $\L$ has a finite accumulation point then $R(\L)=\infty$.
Indeed, if there exists a finite $a>0$ such that $E_\L$ is incomplete in $L^2([0,a])$ then
there exists a non-zero $f\in L^2([0,a]), f\perp E_\L$. Then the entire function $\hat f$ is non-zero and vanishes on $\L$. But
non-zero entire functions cannot vanish on sets with finite accumulation points.

\lsection{Upper density and the Paley-Wiener theorem}

\ms\no A study of zero sets of $PW_a$-functions can give many more results on completeness of complex exponentials. This idea was first
used by Paley and Wiener themselves and later perfected by Levinson in his classical book \cite{LevinsonL2} of 1940, that remains one of the best
books in the area. One of the first fundamental results on completeness obtained this way is the following theorem mentioned in the last lecture.

\begin{theorem}[Paley and Wiener, 1934]
	$$R(\L)\geqslant \bar D(\L)=\limsup_{x\to\infty} \frac{\#(\L\cap (0,x))}{x}.$$
\end{theorem}

\ms\no The idea of the proof is simple. One needs to show that if $F\in PW_a$ then the sequence of its real positive zeros $N$
cannot satisfy $\bar D(N)>a$. This can be done using standard methods of complex analysis, such as Jensen's inequality.
(To complete the details of the proof is a good exercise.)

\lsection{The examples of Kahane and Koosis}

\ms\no This theorem started a long and intensive hunt for the formula for $R(\L)$ in terms of densities. The upper density $\bar D(\L)$, or its various derivations, proved to be insufficient for that goal, as follows from the following historic examples.

\begin{example}[Kahane, 1959] There exists $\L\subset\R$ such that $\bar D(\L)=0$ but $R(\L)=\infty$.
\end{example}

\ms\no In other words, even a very 'thin' sequence of frequencies, in terms of upper density, can generate a sequence of exponentials that will be complete in $L^2$ on any finite interval.

\ms\no In Kahane's example the sequence had large clusters (multiplicities) of points. An immediate question that followed naturally from his construction was whether such clustering is necessary for such a sequence. In particular it was still unclear if a separated sequence can produce a similar example.
A sequence $\L=\{\l_n\}$ is separated if $|\lan -\l_k|>c>0$ for some $c$ and all $n\neq k$. The new question was answered much more quickly
with the following example by Koosis, ruining the remaining hopes for the use of the upper density.

\begin{example}[Koosis, 1960]
	There exists $\L\subset\Z$ such that $\bar D(\L)=0$ but $R(\L)=1$.
\end{example}

\ms\no Notice, that since $\L\subset\Z$, $R(\L)\leq R(\Z)=1$. Hence, Koosis' example gives a subsequence, much thinner than $\Z$, that has maximal possible radius of completeness.

\lsection{Beurling-Malliavin densities}

\ms\no Nevertheless, it turned out that the equation in the Paley-Wiener result will hold if one replaces the upper density with a more delicate density
found in the early sixties by Beurling and Malliavin. We now pass to the definition of the Beurling-Malliavin (effective) density of a real discrete sequence and a dual density that will be used later in the course.

\ms\no
If $\{I_n\}$ is a sequence of disjoint intervals on $\R$, we call it short if
$$\sum\frac{|I_n|^2}{1+\dist^2(0,I_n)}<\infty$$
and long otherwise.

\ms\no Let us point out some simple examples and properties of long (short) sequences of intervals. If $|I_n|<C$ then $\{I_n\}$ is short: $$\sum\asymp \sum \frac 1{n^2}.$$ Also, $I_n=(n^a,(n+1)^a)$ is short for any $a>0$: $$\sum \asymp \sum \frac 1{n^2}.$$
At the same time, a subsequence of dyadic intervals, $I_n=(2^{n_k},2^{n_{k+1}})$, is long for any $n_k\nearrow \infty$, no matter how rare:
$$\sum\asymp 1+1+1+...$$

\ms\no If $\L$ is a sequence of real points define its exterior BM density (effective BM density)
as

$$D^*(\L)=\sup\{ d\ |\ \exists\textrm{ long  }\{I_n\}\textrm{ such that }\#(\L\cap I_n)\geqslant d|I_n|\ \ \forall n\}$$

\ms\no Let us point out the following obvious properties of the new density:

$$D^*(\L)\geq \bar D(\L)$$
and
$$D^*(\Z)=D^*(\N)=1, \ \ D^*(C\Z)=D^*(C\N)=C^{-1}.$$

$$$$

\ms\no A dual definition is used to introduce the interior BM density:

$$D_*(\L)=\inf\{ d\ |\ \exists\textrm{ long  }\{I_n\}\textrm{ such that }\#(\L\cap I_n)\leqslant d|I_n|\ \ \forall n\}.$$

\ms\no We postpone the discussion of the interior density until future lectures.

\ms\no An alternative definition of the two densities, that proves to be more convenient in some of the applications, can be given as follows.

\ms\no For a discrete sequence $\L\subset \R$ we denote by $n_\L (x)$ its
counting function, i.e. the step function  on $\R$, that is constant between any two points   of $\L$, jumps up by $1$ at each point of $\L$ and is equal to $0$ at $0$.
We say that $\L$ is  $a$-\textit{regular}
if
$$\int\frac{|n_{\L}(x)-ax|}{1+x^2}<\infty.$$
In other words, a sequence is $a$-regular if it is close to the  arithmetical progression $a^{-1}\Z$, in the sense that the difference between their counting functions
is summable with respect to the Poisson weight $dx/(1+x^2)$.

\ms\no The exterior density, introduced above, can be equivalently defined as
$$D^*(\L):=\inf \{a\ |\ \exists \ \textrm{$a$-regular supersequence}\ \ \L'\supset \L \}.$$

\ms\no Similarly, for the interior density we have
$$D_*(\L):=\sup \{a\ |\ \exists \ \textrm{$a$-regular subsequence}\ \ \L'\subset \L \}.$$
This definition shows why the terms 'exterior' and 'interior' were used by Beurling and Malliavin in the names of their densities.

\lsection{The Beurling-Malliavin theorem}

\ms\no Now we are ready to formulate one of the deepest theorems of the 20th century Harmonic Analysis.

\begin{theorem}[Beurling and Malliavin, around 1961]
	Let $\L$ be a discrete real sequence. Then
	$$R(\L)=D^*(\L).$$
\end{theorem}

\ms\no (The family $\EE_\L=\{e^{i\l z},\ \l\in\L\}$ is complete in $L^2$ on any interval
of length less than $2\pi D^*(\L)$ and incomplete on any interval of length more than $2\pi D^*(\L)$.)

\ms\no As was mentioned at the beginning of the lecture, the formula extends to the general case $\Lambda\subset \C$ as follows. If $\Lambda$ satisfies the {\it Blaschke} condition
$${\rm (B)},\  \sum_{\lambda\in\Lambda}\left|\Im~\lambda^{-1}\right|<\infty,$$ then
$R(\Lambda)=\pi D_{\rm eff}(\Lambda^*),$ where
$$\Lambda^*=\left\{\lambda^*|\ \lambda^*:=\left[\Re~\lambda^{-1}\right]^{-1},\ \lambda\in
\Lambda\right\},$$
and $R(\Lambda)=\infty$ if $\Lambda\not \in(\textrm{B})$.

\lsection{Exercises}

\bs\no
1) Prove equivalence of the two definitions of BM densities.

\bs\no
2) If $D^*(\L_1)=a$ and $D^*(\L_2)=b$, what can be said about $D^*(\L_1\cap\L_2),\ D^*(\L_1\cup\L_2),\ D^*(\L_1\setminus\L_2)$?
The same question about $D_*$.

\bs\no
3) Using the Beurling-Malliavin theorem, produce real sequences that satisfy the conditions of Kahane's and Koosis' examples.

\bs\no
4) If $0\leq a \leq \infty$, $\e>0$ and $\L$ is a sequence such that $R(\L)=a$, show that  the sequence  of complex exponentials $E_\L$ contains infinitely
many disjoint subsequences such that each of the subsequences is complete in $L^2([0,a-\e])$.

\newpage

\lecture{3}{Inner functions and model spaces}

\ms\no
A version of the Heisenberg Uncertainty Principle formulated in terms of
Harmonic Analysis claims that a non-zero measure (distribution) and its Fourier transform cannot be simultaneously small, see for instance \cite{HJL3}. This broad statement raises a multitude of deep mathematical questions, each corresponding to a particular
sense of "smallness." It includes problems on completeness of exponentials and polynomials that we discussed earlier, inverse spectral problems for differential operators and Krein's canonical systems,
classical problems in the theory of stationary Gaussian Processes, signal processing, etc. Many of such problems remain open to this day.

\ms\no For instance,  the Beurling-Malliavin Problem discussed in the last lecture fits into the general statement of the Uncertainty Principle in the following way. We consider functions with small support, i.e. square summable functions whose support is contained in a finite interval.
We want to show that the Fourier transform of such a function cannot be small in the sense that it cannot have a large zero set (a sequence
of external density larger than the length of the interval). In the opposite direction, if $\hat f$ vanishes on a sequence of large density, then
the support of $f$ cannot be contained in a small interval.

\ms\no Our next goal is to discuss another area within the Uncertainty Principle that deals with Spectral Problems for differential operators and completeness problems for special functions.
To do that we need some preparation in basic complex analysis. The topics we discuss here are covered in a number of textbooks such as
\cite{KoosisHpL3, RudinL3, GarnettL3, cimarossL3}.

\lsection{Herglotz functions}

\ms\no A function $f$ on $\R$ is called Poisson-summable, $f\in L^1_\Pi$, if it is summable with respect to the Poisson measure $\Pi$,
$$ d\Pi=dx/(1+x^2).$$
We say that a complex measure $\mu$ on $\R$ is Poisson-finite if
$$\int\frac{d|\mu|(x)}{1+x^2}<\infty.$$
We denote the set of all Poisson-finite measures on $\R$ by $M_\Pi(\R)$. We will also consider the set of measures $M_\Pi(\hat\R)$
on $\hat\R=\R\cup\{\infty\}$. Each measure from $M_\Pi(\hat\R)$ has the form $\mu=\nu+ c\delta_\infty$, where $\nu\in M_\Pi(\R)$,
$c\in \C$ and $\delta_\infty$ is the unit point mass at infinity.

\ms\no For any $\mu\in M_\Pi(\R)$ one can consider its Poisson integral in the upper half-plane $\C_+$,
$$ P\mu(x+iy)=\frac 1\pi \int\frac{y}{(x-t)^2 + y^2}d\mu(t),$$
that defines a harmonic function in $\C_+$. If the measure $\mu$ is positive, then $P\mu$ defines a positive
harmonic function in $\C_+$ (and a negative harmonic function in $\C_-$).

\ms\no For  $\mu\in M_\Pi(\hat\R)$ the Poisson formula takes the form
$$ P\mu(x+iy)=\frac 1\pi \int_\R\frac{y}{(x-t)^2 + y^2}d\mu(t)+ c y,$$
where the term $cy$ is understood as the Poisson integral of the point mass at infinity.

\ms\no The well-known representation theorem says that any positive harmonic function in $\C_+$ is a Poisson integral of a positive measure
in $M_\Pi(\hat\R)$. (Exercise: find a proof of that statement. Hint: for a positive harmonic function in the unit disk, show that
$L^1$-norms of its restrictions on circles centered at the origin are bounded and take a weak limit of a subsequence as $r\to 1$.)

\ms\no  Moving on from harmonic to analytic functions, for any $\mu\in M_\Pi(\R)$ the Schwarz integral
$$\emph{\emph{S}}\mu(z)=\frac 1{\pi i}\int\left(\frac1{t-z}- \frac t{1+t^2}   \right)d\mu(t)$$
defines an analytic function in $\C_+$ (and in $\C_-$). If $\mu$ is positive, $S\mu$ has positive real part in $\C_+$. Analytic functions
with positive real parts are called Herglotz functions (and the integral $S\mu$ is often called Herglotz integral).

\ms\no For $\mu\in M_\Pi(\hat\R)$,
$$\emph{S}\mu(z)=\frac 1{\pi i}\int_\R\left(\frac1{t-z}- \frac t{1+t^2}   \right)d\mu(t)-icz,$$
where, once again, the last term is understood as the Schwarz integral of the point mass at infinity.
In the analytic case the representation statement is called the Herglotz Representation Theorem, which says that
any Herglotz function $F$ in $\C_+$ can be represented as
$$F=\emph{\emph{S}}\mu +ib,$$
where $\mu>0,\ \mu\in M_\Pi(\hat\R), b\in \R$.

\lsection{Inner functions}

\ms\no  By a theorem of Fatou, any bounded analytic function in $\C_+$ has non-tangential boundary values almost everywhere on $\R$,
see for instance \cite{KoosisHpL3}. If those limits have absolute value $1$ a. e. on $\R$, then such a function is called inner.
I. e., inner functions in $\C_+$ are bounded analytic functions that are equal to $1$ by the absolute value a. e. on the boundary.

\ms\no An important example of an inner function in $\C_+$ is the exponential function
$$S^a(z)=e^{iaz},\ a>0.$$ (Question: why should $a$ be positive?)
Another example is a Blaschke factor
$$b_\l=\frac{z-\l}{z-\bar\l}, \ \ \l\in\C_+.$$
Obviously, any finite product of Blaschke factors will also define an inner function. Finally,
if $\L=\{\lan\}\subset \C_+$ is a sequence satisfying the Blaschke condition
$$\sum \frac{\Im \lan}{1+|\lan|^2}<\infty,$$
then the infinite product
$$B_\L(z)=\prod \e_n b_{\lan},$$
where the constants $\e_n$ are chosen so that $\e_n b_{\lan}(i)>0$, converges normally in $\C_+$ and defines an inner function
$B_\L$ called a Blaschke product.
A theorem by Beurling says that any inner function $I$ in $\C_+$ has the form
$$I(z)=C\exp(-S\mu)B_\L,$$
where $B_\L$ is a Blaschke product, $\mu\in M_\Pi(\hat\R)$ is a \textit{\textit{positive singular}} measure and $C$ is a unimodular constant,
see for instance \cite{cimarossL3, KoosisHpL3, RudinL3, GarnettL3}.

\newpage

\lsection{Meromorphic inner functions}

\ms\no An important subclass of inner functions in $\C_+$ consists of the so-called meromorphic inner functions. These are inner functions that can be
extended into the whole complex plane meromorphically. Such inner functions play an important role in problems of the Uncertainty Principle.

\ms\no  The condition of existence of meromorphic extension into $\C_-$ immediately implies the following representation formula for a meromorphic
inner function $\theta$:

$$\theta(z)=Ce^{iaz}B_\L,$$
where $C$ is a unimodular constant, $a>0$ and $B_\L$ is a Blaschke product corresponding to a \textit{discrete} sequence $\L$ (i.e. $\L$ satisfies the Blaschke condition and has no finite accumulation points).

\lsection{Spaces of analytic functions}

\ms\no Recall that the Hardy space $H^2=H^2(\C_+)$ in the upper half-plane is defined as the space of all analytic functions $f$
in the upper half-plane such that
$$||f||^2_{H^2}=\sup_{y>0}\int_\R |f(x+iy)|^2dx<\infty.$$

\ms\no By Fatou theorem, each function in $H^2$ is uniquely determined by its non-tangential boundary values on $\R$.
If one identifies each function in $H^2$ with its boundary values on $\R$,
the space becomes a closed subspace of $L^2(\R)$ and a Hilbert space. Via this connection, $H^2$ inherits the inner product from $L^2(\R)$:
$$<f,g>_{H^2}=<f,g>_{L^2}=\int_\R f(x) \bar g(x) dx.$$
It turns out that the norm for $H^2$ defined above coincides with the $L^2$-norm, see for instance \cite{KoosisHpL3}.

\ms\no  An important role in the problems of the Uncertainty Principle is played by the following collection of subspaces of $H^2$. If $\theta$ is an inner
function in $\C_+$, denote by $\theta H^2$ the set of functions
$$\theta H^2=\{ \theta f |\ f\in H^2\}.$$
It is clear that $\theta H^2$ is a closed subspace of $H^2$ (that consists of all functions in $H^2$ that are divisible by $\theta$).
Hence, we can consider an orthogonal complement of $\theta H^2$ in $H^2$:
$$K_\theta= H^2 \ominus \theta H^2.$$
This is the definition of the so-called model space $K_\theta$ corresponding to the inner function $\theta$.

\ms\no  As we can see from the definition, such a space $K_\theta$ can be constructed for any inner function $\theta$ in $\C_+$.
Such spaces play a fundamental role in the Functional Model Theory as the only invariant subspaces of the backward shift operator
in $H^2$, see \cite{NikolskiL3, Nikolski1L3}.

\ms\no  The spaces $K_\theta$ may be viewed as generalizations of the classical Paley-Wiener spaces $PW_a$,  see exercises below. They possess many intriguing properties and are still under investigation by analytic function theorists. Problems on sampling, interpolation or uniqueness in $K_\theta$ spaces serve as natural modern extensions of classical completeness problems discussed so far in this course.
We hope to illustrate this with further examples in our future lectures.

\lsection{Spectral theory}

\ms\no Consider the Schr\"odinger equation
\begin{equation}\label{sch}-\ddot u+qu=\lambda u\end{equation} on some interval $(a,b)$ and assume   that the potential $q(t)$ is locally integrable and $a$ is a regular point, i.e. $a$ is finite and $q$ is $ L^1$ at $a$. Let us  fix some selfadjoint  boundary condition at $b$ and consider the Weyl $m$-function
$$m(\lambda)=\frac{\dot u_\lambda(a)}{ u_\lambda(a)},\qquad \lambda\not\in\R,$$
where $u_\lambda(t)$ is any non-trivial solution of \eqref{sch}  satisfying the boundary condition.  We will deal only with the {\it compact resolvent} case, which is equivalent to saying that $m$ extends to a meromorphic function.
It is well known that $m$ defines a Herglotz function in $\C_+$.
Thus we can define   the meromorphic inner function
$$\Theta=\frac{m-i}{m+i},$$
(see exercises) which we call the {\it Weyl inner function} associated with the potential and the fixed boundary condition at $b$.
The transformation
\begin{equation}\label{03}f(t)\;\mapsto\;F(\lambda)=\int_a^bf(t)\frac{u_\lambda(t)}{\dot  u_\lambda(a)+iu_\lambda(a)}~dt\end{equation}
identifies $L^2(a,b)$  with the model space $K_\Theta$ in the same way as the classical Fourier transform (times $S^a$) identifies $L^2(-a,a)$ with $PW_a$. I. e., it is a unitary operator between the two spaces. This allows us to interpret the completeness problem for  families of solutions $\{u_\lambda:~\lambda\in\Lambda\}$  as a problem of uniqueness sets in the model space of $\Theta$.

For some special choices of the potential $q(t)$ the families of solutions $u_{\lan}$ can become families of special functions, such as
Bessel, Jacobi or Airy functions, see for instance \cite{MIFL3}.
Completeness problems of this type, particularly   problems involving  families of special functions, are well-known in the literature, see e.g. \cite{HigL3}. More on this in future lectures.

\lsection{Exercises and notes}

\ms\no 1) Prove the Herglotz Representation Theorem. Show that the theorem fails for non-positive harmonic functions.

\ms\no 2) Prove the Beurling Theorem on the representation formula for inner functions. You may use the fact that for any bounded analytic function
its zeros satisfy the Blaschke condition.

\ms\no 3) Show that any meromorphic inner function has the simplified representation formula given above.

\ms\no 4) Verify that the formula used to define the Weyl inner function,
$$\Theta=\frac{m-i}{m+i},$$
establishes a one-to-one correspondence between inner and Herglotz functions.

\ms\no 5) An alternative definition of the Hardy space is that $H^2$ is equal to the image of $L^2(\R_+)$ under
the Fourier transform. The norm and the inner product in $H^2$ are then defined to make the Fourier transform
a unitary operator $\FF: L^2(\R_+)\to H^2$. Can you verify the equivalence of this definition to our definition?

\ms\no 6) Show that $\FF(L^2([0,a]))$ is a $K_\theta$ space with $\theta=S^{2\pi a}$. This shows that the Paley-Wiener
spaces are particular cases of $K_\theta$:
$$ S^{a\pi}PW_a=K_{S^{2\pi a}}.$$
The classical completeness problems, that are equivalent to uniqueness problems in $PW_a$, as was discussed in the previous lecture, now become particular cases of uniqueness
problems in $K_\theta$. For general $\theta$, most of such problems remain unsolved.

\newpage

\lecture{4}{Weyl inner functions and spectral measures}

\ms\no In this lecture we continue our discussion of connections with spectral theory for differential operators, started in Lecture 3, in more detail.
We will only  discuss the case of Schr\"odinger operators although similar theories exist for general  canonical systems. See \cite{M2L4} and \cite{LSL4} for the basics of the spectral theory. The present discussion is a shortened version of selected sections of our paper with N. Makarov \cite{MIFL4}.

\ms\no Recall that the $H^2$-model space of an inner function $\Theta$,
$$K_\Theta=H^2\ominus\Theta H^2=H^2\cap\Theta \bar H^2,$$
is a Hilbert space with the Hilbert structure inherited from $H^2$. A function $k_\l$ from a Hilbert
space $H$ of analytic functions in a complex domain is called a reproducing kernel corresponding to
the point $\l$ from the domain if for any $f\in H$,
$$ <f,k_\l>_H=f(\l).$$
In the case of $K_\Theta$ the
reproducing kernels are given by the formula
\begin{equation}\label{rek}k_\lambda^\Theta(z)=\frac1{2\pi i}\frac{1-\overline{\Theta(\lambda)}\Theta(z)}{\bar\lambda-z},\qquad \lambda\in \C_+.\end{equation}
The system of all reproducing kernels is complete in $K_\Theta$ (why?). It follows that if $\Theta$ is meromorphic, then all elements of $K_\Theta$ are meromorphic, and one can extend \eqref{rek} to all $\lambda\in\R$.
The monograph \cite{NikolskiL4} provides a comprehensive study of model spaces.

\lsection{Weyl inner functions}

\ms\no As was discussed earlier, meromorphic inner functions appear in the theory of  2nd order selfadjoint differential operators.
Let  $q$ be a real locally summable function on $(a,b)$. We always assume that selfadjoint operators associated with the differential operation $u\mapsto -\ddot u+qu$ have {\it compact resolvent}. This will be the case, for instance, if $q$ is from $L^p(a,b)$, as
well as for most other 'reasonable' potentials.
We suppose that $a$ is  a {\it regular} point  (finite with $q$ summable near $a$), but we allow $b$ to be infinite and/or singular. Let us  fix a selfadjoint boundary condition $\beta$ at $b$; for example, $\beta$ means $u\in L^2$ at $b$ in the case when $b$ is infinite. The  Weyl-Titchmarsh $m$-function of $(q; b, \beta)$ evaluated at $a$,
$$m(\lambda)=m_{b,\beta}^a(\lambda),\qquad \lambda\in \C,$$
is defined by the formula
$$m(\lambda)=\frac{\dot u_\lambda(a)}{u_\lambda(a)},$$
where $u_\lambda(\cdot)$ is a non-trivial  solution of the Schr\"odinger equation satisfying the boundary condition at $b$.  It is well-known
that $m$ is a Herglotz function, and therefore
we can define the corresponding  inner function $\Theta_{b,\beta}^a$  as
$$\Theta_{b,\beta}^a=\frac{m-i}{m+i}.$$
We  call  $\Theta_{b,\beta}^a$ the {\it Weyl} (or  Weyl-Titchmarsh) {\it  inner function} of $q$.

\ms\no Similarly, if
$b\in \R$ is  a {\it regular} point and $\alpha$ is a  selfadjoint boundary condition  at $a\in[-\infty, b)$, we can consider   the $m$-function of $(q; a, \alpha)$ evaluated at $b$,
$$m_{a,\alpha}^b(\lambda)=-
\frac{\dot u_\lambda(b)}{u_\lambda(b)}$$
and define the corresponding  Weyl inner function
$\Theta_{a,\alpha}^b.$ Note that the sign in the last formula has changed with the change of the endpoint.

\bs\no {\bf Example.} The Weyl inner  functions of  the potential $q\equiv 0$ on $[0,1]$ with Dirichlet and,  respectively, Neumann boundary conditions at $a=0$ are
\begin{equation}\label{DN}\Theta_{D}(\lambda)=\frac{\sqrt\lambda\cos\sqrt\lambda+i\sin\sqrt\lambda}
{\sqrt\lambda\cos\sqrt\lambda-i\sin\sqrt\lambda},\qquad\Theta_{N}(\lambda)=\frac{\sqrt\lambda\sin\sqrt\lambda-i\cos\sqrt\lambda}
{\sqrt\lambda\sin\sqrt\lambda+i\cos\sqrt\lambda}.\end{equation}
(The $m$-functions are
$m_{D}(\lambda)=-\sqrt\lambda\cot\sqrt\lambda,$ and
$m_{N}(\lambda)=\sqrt\lambda\tan\sqrt\lambda.$)

\bs\no {\bf Example.} More generally,
for $\nu \ge-1/2$ consider the potential $$q(t)=\frac {\nu^2-\frac14}{t^2}\quad{\rm on}\quad (0,1),$$
and let the boundary condition $\alpha$ at $a=0$  be satisfied by  the solution  $$u_\lambda(t)=\sqrt tJ_\nu(t \sqrt\lambda)$$
of  the Schr\"odinger equation.
For example, if $\nu=-1/2$ then $\alpha=(N)$, and  if $\nu=1/2$ then $\alpha=(D)$, and   we have the limit point case if $\nu\ge 1$. $J_\nu$ is of course the standard  notation for the Bessel function of order $\nu$.
Since $$u_\lambda(1)= J_\nu(\sqrt\lambda),\qquad
\dot u_\lambda(1)=\frac1{2} J_\nu(\sqrt\lambda )+ \sqrt \lambda J'_\nu(\sqrt\lambda ),$$
the corresponding Weyl inner function is
\begin{equation}\label{Bes}\Theta_\nu(\lambda)=\frac{\sqrt\lambda J'_\nu(\sqrt\lambda)+(1/2+i)J_\nu(\sqrt\lambda)}{\sqrt\lambda J'_\nu(\sqrt\lambda)+(1/2-i)J_\nu(\sqrt\lambda)}.\end{equation}
In particular, we have
$\Theta_{-1/2}=\Theta_{N}$ and $\Theta_{1/2}=\Theta_{D}.$

\ms\no We will discuss the Bessel example further in our lectures.

\lsection{Modified Fourier transform}

\ms\no Let $\Theta=\Theta^a_{b,\beta}$ be the Weyl-Titchmarsh inner function of a potential $q$ defined in the previous section. We will construct a unitary operator $L^2(a,b)\to K_\Theta,$
which is a modification of the Weyl-Titchmarsh Fourier transform. We modify the usual construction so that the case of a singular (i.e. non-regular) endpoint $b$ could be included.

\ms\no For every $z\in \C$ we choose a non-trivial  solution $u_z(t)$ of the Schr\"odinger equation satisfying the boundary condition $\beta$. (For real $z$ such a solution exists because of the compact resolvent assumption).
If $z\in \C_+\cup\R$, then  the solution
$$w_z(t)=\frac{u_z(t)}{\dot u_z(a)+i u_z(a)}$$
does not depend on the choice of $u_z$, and $w_z\in L^2(a,b)$.
The  transform $\WW$ is defined as follows:
\begin{equation}\label{W}\WW:~  f(t)\mapsto F(z)=\int_a^bf(t)w_z(t)dt,\quad(z\in \C_+\cup\R).\end{equation}

\ms\no To state the main result we introduce the {\it dual reproducing kernel} of the model space $K_\Theta$. For $\lambda\in \C_+\cup\R$ we define
\begin{equation}\label{drk}k^*_\lambda(z)=\frac1{2\pi i}\frac{\Theta(z)-\Theta(\lambda)}{z-\lambda},\qquad (z\in \C_+\cup\R),\end{equation}so  we have
$$\bar\Theta k_\lambda^\Theta=\overline{k^*_\lambda}\qquad{\rm on}\quad\R,$$
and $k^*_\lambda\in K_\Theta$.
Note that if $\lambda\in \R$, then $k^*_\lambda=\const~ k_\lambda^\Theta$.

\ms\begin{theorem}\cite{MIFL4} The modified Fourier transform $\WW$ is (up to a factor $\sqrt \pi$) a unitary operator
	$L^2(a,b)\to K_\Theta.$ Furthermore, we
	have
	\begin{equation}\label{WW}\WW w_\lambda=\pi k^*_\lambda,\quad \WW \bar w_\lambda=\pi k_\lambda\qquad  (\lambda\in \C_+\cup\R).\end{equation}\end{theorem}

\ms\begin{proof} The formulae \eqref{WW} follow from the Lagrange identity
	$$(z-\lambda)\int_a^b u_\lambda u_z=u_\lambda(a)\dot u_z(a)-\dot u_\lambda(a)u_z(a).$$
	(The Wronskian  at $b$ is zero because  the two solutions satisfy the same boundary conditions.)
	The rest is straightforward:
	$$(\bar w_\lambda,\bar w_\mu)_{L^2}=\int_a^b w_\mu\bar w_\lambda=\WW \bar w_{\lambda}(\mu)=\pi k_\lambda(\mu)=\pi (k_\lambda,k_\mu)_{K_\Theta},$$etc.
\end{proof}

\ms\no As was discussed in the last lecture, a  meromorphic Herglotz function is   a meromorphic function $m$ such that $$\Im m>0\quad{\rm  in}\; \C_+,\qquad m(\bar z)=\overline{m(z)}.$$ One can establish a 1-to-1 correspondence between  meromorphic inner and Herglotz  functions by means of the equations
\begin{equation} \label{iH} m=i\frac{1+\Theta}{1-\Theta},\qquad \Theta=\frac{m-i}{m+i}. \end{equation}
Meromorphic Herglotz functions (and therefore inner functions) can be
described  by parameters $ (b,c,\mu)$ in the Herglotz representation
\begin{equation}\label{Her}m(z)=bz+c+iS\mu,\end{equation}
where $b\ge0$, $c\in \R$, and $\mu$ is a positive discrete measure on $\R$ satisfying
$$\int\frac{d\mu(t)}{1+t^2}<\infty.$$
It is convenient to interpret the number $\pi b$ as a point mass of $\mu$ at infinity. In the case $m=m_\Theta$, see \eqref{iH}, we call this extended  measure $\mu_\Theta$  the {\it spectral} (or {\it Herglotz}) {\it measure} of $\Theta$.  By definition,  the (point) {\it spectrum} of $\Theta$ is the set$$\sigma(\Theta)=\supp~\mu_\Theta=~\{\Theta=1\}\;{\rm or}\;
\{\Theta=1\}\cup\{\infty\},$$
and by  residue calculus we have
\begin{equation}\label{res}\mu_\Theta(t)=\frac{2\pi}{|\Theta'(t)|},\qquad t\in\sigma(\Theta).\end{equation}The following {\it equivalent} conditions are necessary and sufficient for $\mu_\Theta(\infty)\ne0$, see e.g. \cite{Sa1}:
\begin{equation*} \textrm{(i)} \quad \Theta-1\in H^2; \quad\textrm{(ii)} \quad\Theta(\infty)=1,\;\exists \Theta'(\infty);\quad\textrm{(iii)}\quad\sum\Im\lambda<\infty.\end{equation*}
In (ii),  $\Theta(\infty)$ and $\Theta'(\infty)$  mean the angular  limit and angular  derivative at infinity:
$$\Theta(\infty)=\lim_{y\to+\infty}\Theta(iy),\qquad
\Theta'(\infty)=\lim_{y\to+\infty}y^2\Theta'(iy),$$
and in (iii) we also require that the singular factor is trivial.

\ms\no   Note that  Weyl inner functions of Schr\"odinger operators have no point masses at infinity, so
if $\Theta=\Theta_{b,\beta}^a$, then
$$\sigma(\Theta)=\sigma(q, D, \beta),\qquad \sigma(-\Theta)=\sigma(q, N, \beta).$$
Here $\sigma(q, D, \beta)$ means  the spectrum of the Schr\"odinger operator with potential $q$,  Dirichlet boundary condition at $a$, and  boundary condition  $\beta$ at $b$. More generally, for $\alpha\in\R$ let $\alpha$ denote the following  selfadjoint boundary condition at a regular endpoint $a$:
\begin{equation}\label{bc}\cos\frac\alpha2 u(a)+\sin\frac\alpha2 \dot u(a)=0.\end{equation}
Then
$$\sigma(e^{-i\alpha}\Theta)=\sigma(q, \alpha, \beta).$$
The spectral measure of the Schr\"odinger operator $(q,\alpha,\beta)$ is  the Herglotz measure of the inner function $e^{-i\alpha}\Theta$
(this can be viewed as a definition of the spectral measure of a Schr\"odinger operator).

\bs\no (Note: for those readers familiar with Clark theory, see for instance \cite{cimarossL4}, the Herglotz measure of $e^{-i\alpha}\Theta$
is the Clark measure $\sigma_\alpha$ corresponding to $\Theta$.)

\lsection{Exercises}

\ms\no 1) Let $H$ be a Hilbert space of analytic functions in a complex domain $\Omega$ such that the linear functional of point evaluation
$f\mapsto f(\l)$ is bounded with respect to the norm of $H$ for any $\l\in\Omega$. Prove that $H$ has a full system of reproducing kernels
$k_\l,$ i.e. that $k_\l$ exists for any $\l\in\Omega$.

\ms\no 2) If $H$ is as above, show that reproducing kernels are complete in $H$.

\ms\no 3) Using the formula for the reproducing kernel of a $K_\Theta$ space and the connection between $K_\Theta$ and $PW_a$ discussed in the previous lecture, find the formula for a reproducing kernel for $PW_a$. Find the same formula directly by calculating the Fourier transform of a restriction
of $e^{i\l z}$ on $(-a,a)$.

\ms\no 4) Calculate the Weyl inner functions for the free Laplacian ($q=0$), given in the example above, by hand.

\ms\no 5) Let $\Theta$ be a fixed inner function. Prove the following simple properties of Herglotz measures $\sigma_\alpha$ for $e^{-i\alpha}\Theta$, $\alpha\in\R$ (a.k.a. Clark measures).

\ms\no
a) All $\sigma_\alpha$ are singular (with respect to the Lebesgue measure on the line).

\ms\no
b) All $\sigma_\alpha$ are mutually singular, i.e. $\sigma_\alpha\perp\sigma_\beta$ for $\alpha\neq\beta$.

\ms\no
c)  $\sigma_\alpha$ is supported on a set where the non-tangential limits of $\Theta$ are equal to $\alpha$
(note that the set is not generally closed, so it is not the closed support of $\sigma_\alpha$).

\ms\no Think about the corresponding statements for spectral measures for differential operators, that follow from a)-c) and the above discussion. For these and further properties of $\sigma_\alpha$ see for instance \cite{cimarossL4}.

\newpage

\lecture{5}{Mixed spectral problems and defining sets}

\ms\no We continue to discuss applications of complex function theory to spectral problems for differential operators. In this lecture
we will be able to reach recent results and enter an area of current research.

\lsection{Abstract Hochstadt-Liberman problem}

\ms\no We will be considering the following problem concerning general meromorphic inner functions. In the next section we will explain  its relation to  Hochstadt-Liberman's theorem on the spectra of Schr\"odinger operators \cite{HL}.

\bs\no Let $\Phi$ and $\Psi$ be   meromorphic inner  functions and  $\Theta=\Psi\Phi$. As usual, $\sigma(\Theta)$ denotes the (point) spectrum of $\Theta$, $\{\Theta=1\}$, see the last lecture. Recall that $\sigma(\Theta)$   may include  $\infty$.
We say that {\it the data}
$[\Psi, \sigma(\Theta)]$ {\it determine} $\Theta$ 
if any inner function divisible by $\Psi$ whose   spectrum is  $\sigma(\Theta)$ is equal to $\Theta$. I.e. $[\Psi, \sigma(\Theta)]$ determine $\Theta$
if for any inner function $\ti\Phi$, $$\ti\Theta=\Psi\ti\Phi,\quad  \sigma(\ti\Theta)=\sigma(\Theta)\qquad\Rightarrow\qquad \Theta=\ti\Theta.$$
Alternatively, we can say that $\Psi$ and $\sigma(\Phi\Psi)$ determine $\Phi$. Given $\Phi$ and $\Psi$, the problem is to decide if this is the case.

\ms\no
The set of Herglotz measures of inner functions $\ti\Theta$ satisfying $\Psi|\ti\Theta$ ($\Psi$ divides $\Theta$) and $\sigma(\ti\Theta)=\sigma(\Theta)$ is convex, see Section 1.2 in \cite{MIFL5}. We will refer to the dimension of this set as the dimension of the set of solutions.

\bs\no  {\bf Example}.
Suppose $\Theta$ is a {\it finite Blaschke product}. Then $$[\Psi, \sigma(\Theta)]\quad{\rm   determine}\quad \Theta\qquad\Lra\qquad
2\deg\Psi>\deg\Theta.$$  The proof is elementary; it also follows  from the results  below.
As an illustration consider the simplest case $\Theta=b^2$, $\Psi=b$,
where
$$b(z)=\frac{z-i}{z+i}.$$ Then $\sigma(\Theta)=\{0,\infty\}$, and the data $[\Psi, \sigma(\Theta)]$ does not determine $\Theta$. In fact, the set of solutions is one-dimensional; the solutions are given by the formula
$$\ti\Phi(z)=\frac{z-ia}{z+ia},\qquad(a>0).$$

\ms\no
Finding necessary and sufficient conditions for $[\Psi, \sigma(\Theta)]$ to determine $\Theta$ is an important problem of complex function theory.
As we will see shortly, it appears in spectral theory for differential operators, as well as in other areas of analysis. In  \cite{MIFL5} several such conditions are formulated in terms of the Toeplitz kernels with symbol $U=\bar\Phi\Psi$. The theory of Toeplitz operators is a broad and important
part of complex analysis that, due to lack of space, will not be covered in this introductory course. We refer an interested reader to
\cite{MIFL5, MIF2L5} for more information and further references.

\ms\no The rough meaning of the conditions formulated in \cite{MIFL5} is the following: for the data $[\Psi, \sigma(\Theta)]$ to  determine $\Theta$, the known factor $\Psi$ of the inner function has to be "bigger" than  the unknown factor $\Phi$. Let us formulate one of the
results of the 'Toeplitz approach' from \cite{MIFL5}. This relatively simple example extends the original Hochstadt-Liberman theorem, which we will state in the next subsection.

\begin{cor} Suppose  $\Theta=\Psi^2$.  Then the set of solutions is exactly one-dimensional:
	$\ti\Theta$ satisfies $\Psi|\ti \Theta$,
	$\sigma(\ti \Theta)=\sigma(\Theta)$ iff
	\begin{equation}\label{352}\exists r\in (-1,1),\qquad \ti \Theta=\Psi\frac{r+\Psi}{1+r\Psi}.\end{equation}\end{cor}

\lsection{Spectral theory interpretation: Hochstadt-Liberman and Khodakovski theorems}

\ms\no
Consider a Schr\"odinger operator  $L=(q,\alpha,\beta)$ on  $(a,b)$, where $q\in L^1_{{\rm loc}}(a,b)$ and $\alpha$, $\beta$ are selfadjoint boundary conditions at $a$ and $b$ respectively; the endpoints can be infinite and/or singular.
We assume that $L$ has compact resolvent. As usual, $\sigma(L)$ denotes the spectrum of $L$.

\ms\no Suppose  $a<c<b$. We will write $q_-$ for the restriction of $q$ to $(a,c)$ and $q_+$ for the restriction of $q$ to $(c,b)$.  We say that the  data
$(q_-, \alpha,\sigma(L))$ {\it determines} $L$ if for any other Schr\"odinger operator
$\ti L=(\ti q,\ti \alpha,\ti \beta)$,
$$ q_-= q,\;\alpha=\ti\alpha,\;\sigma(\ti L)=\sigma(L)\qquad \Ra\qquad \ti q_+= q_+,\; \ti\beta=\beta.$$

\bs\no Let   $\Theta_-$ denote the Weyl inner function of $(q_-,\alpha)$ computed at $c$ and $\Theta_+$  the Weyl inner function of $(q_+,\beta)$ computed at $c$.

\ms\no \begin{lem} $\sigma(L)=\sigma(\Theta_-\Theta_+)$.\end{lem}

\ms\begin{proof} The equation $\Theta_-(\lambda)\Theta_+(\lambda)=1$ is equivalent to the statement $$m_+(\lambda)+m_-(\lambda)=0\quad{\rm or}\quad m_-(\lambda)=m_+(\lambda)=\infty$$
	for the corresponding  $m$-functions. The latter means that we have the matching
	$$\frac{\dot u_-(c,\lambda)}{u_-(c,\lambda)}=\frac{\dot u_+(c,\lambda)}{u_+(c,\lambda)}$$
	for any two non-trivial solutions $u_-(\cdot,\lambda)$ and  $u_+(\cdot,\lambda)$ of the Schr\"odinger equation with boundary conditions $\alpha$ and $\beta$ respectively, which is possible if and only if $\lambda$ is an eigenvalue of $L$.
\end{proof}

\ms\no \begin{cor} $(q_-,\alpha,\sigma(L))$ determine $L$ if the data $(\Theta_-, \sigma(\Theta_-\Theta_+))$ determine $\Theta_+$.
\end{cor}

\ms\no Here we rely on the fundamental uniqueness theorem of Borg and  Marchenko \cite{Borg}, \cite{M1}: the $m$-function (and therefore the Weyl inner function) determines both  the potential and the boundary condition.

\bs\no {\bf Remark.} We would have an "iff" statement if we considered the problem in some class of canonical systems with a one-to-one correspondence between the systems and inner functions such as the class of  Krein's "strings", see \cite{dBL5}, \cite{DML5}. The effective characterization of inner functions of Schr\"odinger operators is an open problem, so we will use our  general results   to state only sufficient conditions for  Schr\"odinger operators. To obtain necessary condition one has to use more specific techniques of the Schr\"odinger operator theory, see \cite{Borg}, \cite{HoL5}.

\bs\no Let us apply the above corollary to the situation described at the end of the last  subsection.

\bs\no {\bf Example 1.} {\it Let $L$ be a   Schr\"odinger operator on $\R$ with  compact resolvent and limit point  boundary conditions   at $\pm\infty$. Suppose the potential $q(x)$ is an even function:$$ q(-x)=q(x), \qquad (x>0).$$ Then   $q|_{\R_-}$ and  $\sigma(L)$ determine $L$.}

\ms\begin{proof} By Everitt's theorem \cite{Ev} (see \cite{MIFL5} for a more detailed discussion of that result), all the inner functions $(r+\Psi)/(1+r\Psi)$ in \eqref{352} with $r\ne0$ are not Weyl inner functions corresponding to a    Schr\"odinger operator.
\end{proof}

\ms\no This result is a special case of Khodakovski's theorem \cite {Kh}, where    only the equality $q(-x)\le q(x)$  for  $x>0$ is assumed. The full version of Khodakovski's  theorem requires a slightly different approach which we  describe in the next subsection.   Similarly, we derive the following statement (it follows from our previous discussion and from the remark at the end of Section 2.5 in \cite{MIFL5}).

\begin{prop} Let $L$ be as above, and let $\ti L$ be another Schr\"odinger operator on $(-\infty, b)$, $b\ge0$. If $$q=\ti q\quad{\rm  on}\quad  \R_-\qquad {\rm and}\qquad \sigma(\ti L)\subset \sigma(L),$$ then either $\ti L=L$ or $b=0$ and $\ti L$ is  the operator with potential $q_-$ and Dirichlet or Neumann condition at 0.
\end{prop}

\bs\no {\bf Example 2.} {\it Let $L$ be a regular selfadjoint Schr\"odinger operator on $[a,b]$ with non-Dirichlet boundary conditions $\alpha$ and $\beta$  at $a$ and $b$ respectively. If $c=(a+b)/2$, then $(q_-,\alpha, \sigma(L))$ determine $L$.}

\ms\no The statement is also true if one or both boundary conditions are Dirichlet, see next subsection.
This is a stronger version of the Hochstadt-Liberman theorem \cite{HL}, see also \cite {GS1} which states that if {\it both} $L$ and $\ti L$ are regular, and $\ti q_-=q_-$, $\ti\alpha=\alpha$, $\sigma(\ti L)=\sigma(L)$, then $\ti L=L$.  We {\it do not} require $\ti L$ to be regular
(recall that regular $=$ summable potential). Also, we can replace $\sigma(\ti L)=\sigma(L)$ with  $\sigma(\ti L)\subset\sigma(L).$

\lsection{Example: Bessel inner functions}

\ms\no This is an extension of the previous example. We want to  show that the  Hochstadt-Liberman phenomenon occurs not only for regular potentials.

\ms\no We consider the Bessel inner functions $\Theta_\nu$, $\nu \ge-1/2$, see the last lecture.  

\ms\no Applying our methods we get the following result, see \cite{MIFL5}.

\ms
\begin{thm} Let $L$ be the Schr\"odinger operator with potential $q(t)=2t^{-2}$ on $[0,2]$ and with Dirichlet boundary condition at $t=2$.  Then $q|_{(0,1)}$ and the spectrum $\sigma(L)$ determine $L$ in the class of Schr\"odinger operators.
\end{thm}

\lsection{Defining sets of inner functions}

\ms\no Let $\Phi$ be a meromorphic inner function, 
$\Phi=e^{i\phi}$ on $\R$ for a smooth real function $\phi$. Let $\Lambda\subset \R$. We say that $\Lambda$ is a {\it defining set} for $\Phi$ if
for any other meromorphic inner function $\ti\Phi$, $\ti\Phi=e^{i\ti\phi},$
$$\ti\phi=\phi\quad {\rm on}\quad \Lambda\qquad\Ra\qquad \Phi\equiv\ti\Phi.$$

\bs\no  In this definition we tacitly assume $\phi(\pm\infty)=\pm \infty$.
In the "one-sided" case, say if $\phi(-\infty)>-\infty$ and $\phi(+\infty)=+\infty$, one should modify the definition in an obvious way.
(An important and well-known property of arguments of meromorphic inner functions on $\R$ is that they are always monotonically growing. Why?)

\ms\no  One can extend this definition to  divisors. For instance, if all points in $\Lambda\subset \sigma(\Phi)$ are double, then the equality $\ti \Phi=\Phi$ on $\Lambda$ means that  the  spectral measures of the inner functions coincide on $\Lambda$.

\bs\no Let us mention several special cases.

\bs\no (a)~ {\it Two spectra problem.} 

\bs\no {\it Let $\Phi$ be a meromorphic inner function. Then a meromorphic inner function  $\ti \Phi$  satisfies
	$\{\ti\Phi=1\}= \{\Phi=1\}$ and $\{\ti\Phi=-1\}= \{\Phi=-1\}$    iff }
\begin{equation}\label{380}\ti \Phi=\frac{\Phi-c}{1-c\Phi},\qquad c\in(-1,1).\end{equation}

\ms\no This corresponds to the case $$\Lambda=\{\Phi=1\}\cup \{\Phi=-1\}.$$
The meaning of the  statement is that $\Lambda$ is defining for $\Phi$ with deficiency one
(in the case $\phi(\pm\infty)=\pm\infty$, to be accurate).
Various related statements are of course
well-known, see e.g. \cite{Borg}.

\bs\no The easiest way to see this is to use  Krein's shift construction: since
$$\Re\left[\frac 1{\pi i}\log\frac{\ti\Phi+1}{\ti\Phi-1}\right]=\chi_e\qquad{\rm on} \quad\R,$$
where $e=\{\Im\Phi>0\}$,  we have
$$\frac 1{\pi i}\log\frac{\ti\Phi+1}{\ti\Phi-1}=\SS\chi_e+\const.\qquad\Box$$

\bs\no This argument also shows that
{\it given any two intertwining discrete sets $\Lambda_{\pm}$ of real numbers
	there is a meromorphic inner function $\Phi$ such that} $$\{\Phi=\pm1\}=\Lambda_\pm$$
(see Exercises at the end).

\bs\no Let us also mention that the statement \eqref{380} can be derived from the twin inner function theorem, see Section 2.8 in \cite{MIFL5}.

\bs\no (b)~ {\it General mixed data spectral problem.}  The Hochstadt-Liberman problem for inner functions that we discussed above can be viewed as a special case of the defining sets problem. It is easy to see that if (assuming $\arg\Theta(\pm\infty)=\pm \infty$)
$\Theta=\Psi\Phi$ and $\Lambda=\sigma(\Theta)$, then
$$(\Psi,\sigma(\Theta))\quad{\rm determine}\quad \Theta \qquad\Lra\qquad  \Lambda\quad{\rm is ~defining~ for}\quad \Phi.$$

\ms\no This can be generalized in the following way. Let $\Theta=\Psi\Phi$ be a given   meromorphic inner function  and let $\{\lambda_n\}$ be the set of its eigenvalues numbered in the increasing order. Given $M\subset \Z$ we denote
$$\sigma_M(\Theta)=\{\lambda_n:~n\in M\}.$$
The question is whether the factor $\Psi$ and the partial spectrum $\sigma_M(\Theta)$ determine $\Theta$, i.e. whether
$$\ti\Theta=\Psi\ti\Phi,\quad \ti \lambda_n= \lambda_n ~(n\in M)\qquad \Ra \qquad
\ti\Theta\equiv\Theta.$$
Once again, this is equivalent (assuming $\phi(\pm\infty)=\pm\infty$) to saying that $\Lambda=\sigma_M(\Theta)$ is a  defining set for $\Phi$. The spectral theory meaning was explained in  \cite{MIFL5} and the partial spectral problem for Schr\"odinger operators and Jacobi matrices appeared in several publications, e.g. \cite{GS1}, \cite{GS2}.

\bs\no (c)~ {\it A version for spectral measures.} Given   a meromorphic inner function $\Theta$ and a factor $\Psi|\Theta$, and also given a part of the spectrum $\Lambda=\sigma_M(\Theta)$, the
question is whether there is another inner function $\ti\Theta\ne\Theta$ such that $\Psi|\ti\Theta$ and the spectral measures $\mu=\mu_\Theta$ and  $\ti\mu=\mu_{\ti \Theta}$ coincide on $\Lambda$:
$$ \ti \lambda_n= \lambda_n,\quad \mu\{ \lambda_n\}=\ti\mu\{ \lambda_n\}, \qquad  (n\in M).$$
Claim: If $\Theta=\Psi\Phi$, then $\Psi$ and the spectral measure on $\Lambda=\sigma_M(\Theta)$ determine $\Theta$ iff  the divisor $2\chi_\Lambda$ is defining for $\Phi$.

\ss\no Indeed, if $\ti\Theta=\Psi\ti\Phi$, and
$$\arg\Theta(\lambda_n)=\arg\ti\Theta(\lambda_n)=2\pi n,\qquad (n\in M),$$
then $$\arg\ti\Phi=\arg\Phi\qquad{\rm  on}\quad \Lambda.$$ The relation$$\mu\{\lambda\}=\ti\mu\{\lambda\},\qquad \lambda\in\Lambda$$
then implies $\ti \Theta'(\lambda)= \Theta'(\lambda)$, so
$$ \Psi'(\lambda)\ti \Phi(\lambda)+\Psi(\lambda)\ti \Phi'(\lambda)=
\Psi'(\lambda) \Phi(\lambda)+\Psi(\lambda) \Phi'(\lambda),\qquad (\lambda\in\Lambda),$$
and
$$(\arg\ti\Phi)'=(\arg\Phi)'\qquad{\rm  on}\quad \Lambda.\qquad\Box $$

\ms\no Again, the spectral theory interpretation is the same as above: we know some part of a differential operator  and some part of its  spectral measure and we want to know if this information determines the operator uniquely.

\bs\no As usual we can consider the problem in a restricted class of inner functions. Here is the  simplest example.

\bs\no {\bf Example.} Let $\Theta=\Psi\Phi$ be a finite Blaschke product. Then $\Psi$ and $\Lambda\subset \sigma(\Theta)$ determine $\Theta$ iff~
$\#\Lambda>2\deg~\Phi$ in the class of Blaschke products of a fixed degree. Similarly,
$\Psi$ and the spectral measure on $\Lambda$ determine  $\Theta$ iff
$\#\Lambda>\deg~\Phi.$ This extends in an obvious way to the cases where only $\Phi$ or $\Psi$ has a finite degree. These facts follow for instance from the statements in the next section, also cf. \cite{GS2}.

\lsection{Exercises}

\ms\no Note: these exercises are difficult, but if you can do them, you are about ready to start your own research in this area. You may want to consult \cite{MIFL5}
or further references given there if you need help. Even if you cannot finish, try to proceed as far as you can in each exercise.

\ms\no 1) Prove that for any two alternating discrete sequences $\Lambda_+,\ \Lambda_-$ on $\R$ there exists a unique meromorphic inner function $\Phi$ such that $$\{\Phi=\pm1\}=\Lambda_\pm.$$
Consider the case, when one of the sequences contains infinity. What does this statement mean for Schr\"odinger operators and their spectra?

\ms\no 2) Consider the following simplified case of the H-L theorem. Let $L$ be a regular Schr\"odinger operator on an interval $[0,2]$. It is well-known that then its spectrum $\sigma(L)$ is a discrete
sequence on the real line. Suppose that $\ti L$ is another regular Schr\"odinger operator on $[0,2]$ such that 
$\ti q= q$ on $[0,1]$, the boundary conditions for $L$ and $\ti L$ coincide and $\sigma(\ti L)$ contains every other point from $\sigma(L)$, plus at least one more point from $\sigma(L)$. Try to prove that then $L=\ti L$, i.e. $q=\ti q$ on the whole interval. First translate the
problem into the language of Weyl inner functions. Then find a way to apply complex analytic methods to show that $\Theta=\ti\Theta$.

\newpage

\lecture{6}{Defining sets and uniqueness sets}

\ms\no In the second part of this lecture we arrive at one of the main points of the course. We show how spectral problems for differential operators
connect to problems on completeness of complex exponentials and discuss the relations between these two important classical areas of
mathematics.

\ms\no Recall that if $\Phi$ is a meromorphic inner function,
$\Phi=e^{i\phi}$ on $\R$ for a smooth real function $\phi$, we say that $\Lambda$ is a {\it defining set} for $\Phi$ if
for any other meromorphic inner function $\ti\Phi$, $\ti\Phi=e^{i\ti\phi},$
$$\ti\phi=\phi\quad {\rm on}\quad \Lambda\qquad\Ra\qquad \Phi\equiv\ti\Phi.$$
As before, $K_\Phi=K_\Phi^2$ is the model space $H^2\ominus \Phi H^2$ corresponding to $\Phi$. In general,  $K^p_\Phi$
is defined as the closure of finite linear combinations of reproducing kernels in $L^p(\R),\ 0<p\leq\infty$.

\lsection{Relation to uniqueness sets}

\begin{prop} $\Lambda$ is not defining for $\Phi$ if there is a non-constant function $G\in K^\infty_\Phi$ such that
	\begin{equation}\label{39}G=\bar G\quad{\rm on}\quad\Lambda.\end{equation}\end{prop}

\ms
\begin{proof} We can assume   $\|G\|_\infty<1.$ Let $F$ be a bounded analytic function in $\C_+$ such that   $\bar \Phi G=\bar F$ on $\R$, and consider
	$$\ti \Phi=\frac{\Phi+F}{1+G}.$$
	Then $\ti\Phi$ is an inner function because it is in the Smirnov class $\NN^+$ and $$|\Phi+F|=|
	\Phi+\Phi \bar G|=|1+G|\qquad {\rm on}\quad \R.$$
	Also, $\ti\Phi\ne\Phi$ because otherwise we would have $F=G\Phi$, which together with  $F=\Phi \bar G$ implies $G=\bar G$, so $G=\const$. Finally, we have
	$$\ti\Phi=\Phi\frac{1+\bar\Phi F}{1+G}=\Phi\frac{1+\bar G}{1+G}=\Phi\qquad {\rm on}\quad \Lambda,$$and since
	$$\|\arg~\ti\Phi-\arg~\Phi\|_{L^\infty(\R)}<2\pi$$ by construction, we get
	$\arg~\ti\Phi=\arg~\Phi$ on $\Lambda$.
\end{proof}

\ms\no {\bf Remark.}
We say that a set is a uniqueness set for a set of functions if any function that is zero on that set is identically zero.
The condition \eqref{39} is very close to the condition that $\Lambda$ is  not  a uniqueness set for $K^\infty_{\Phi^2}$. The precise relation between the two statements is an interesting question, which we will not   discuss here. We only mention that if $p\in(1,\infty)$, then
$$\exists G\in K^p_{\Phi},\quad G\not\equiv\const,\quad G=\bar G\quad{\rm on}\quad \Lambda,$$iff
$$\exists F\in K^p_{\Phi^2},\qquad F\not\equiv0,\quad F=0\quad{\rm on}\quad \Lambda.$$

\bs\no The  above proposition  gives a necessary condition for a set $\Lambda$ to be defining for $\Phi$. To get sufficient conditions one can use the following simple observation
(prove it).

\ms \begin{lem} If  $\ti \Phi=\Phi$ on  $\Lambda$ and  $F=\ti\Phi-\Phi$, then
	$$F\in K^\infty_{\ti\Phi\Phi},\qquad F=0\quad{\rm on}\quad\Lambda.$$\end{lem}

\ms\no If we also have $\arg~\ti\Phi=\arg~\Phi$ on $\Lambda$ (as in the definition of defining sets), then  we can estimate the argument of $\ti\Phi\Phi$ in terms of the data $(\Phi,\Lambda)$, so we can  apply our results concerning uniqueness sets.

\lsection{Defining sets of regular operators}

\ms\no We now consider the defining sets problem in some restricted classes of inner functions. We will use the spectral theory  language.
For $r\ge1$ let $\Sc(L^r,D)$ denote the class of selfadjoint Schr\"odinger operators on $[0,1]$ with an $L^r$ potential  and  Dirichlet boundary condition at $0$.

\bs\no We say that $\Lambda\subset\R$ is a {\it  defining set for the class} $\Sc(L^r,D)$ if for any two operators in $\Sc(L^r,D)$ with potentials $q$ and $\ti q$, the equality
$\ti\Theta=\Theta$ on $\Lambda$ implies  $\ti q\equiv q$, where $\ti\Theta$ and $\Theta$ are  the corresponding Weyl inner  functions.

\ms\no We have a similar definition for the classes $\Sc(L^r,N)$ of Schr\"odinger operators with Neumann boundary condition at $0$.

\bs\no  Let $\Theta_D$ denote the standard inner function, i.e. the Weyl inner function in the case $q\equiv0$, see previous lectures.
The following statement follows immediately from Lemma 3.9 in \cite{MIFL6}:

\ms\no{\it $\Lambda$ is defining in the class $\Sc(L^1,D)$ if
	$\Lambda$ is a uniqueness  set of } $K^\infty_{\Theta_D^2}$.

\bs\no  This sufficient condition is not optimal because for regular operators, the function $\ti\Phi-\Phi$ (see the statement of Lemma 3.9) has some extra smoothness at infinity as follows from the standard asymptotic formulae (see the end of this section), which are getting more precise if we require more regularity of the potential, in particular if we consider the case $q\in L^r$ with  $r>1$.

\lsection{A theorem of Horv\'ath}

\bs\no In a paper published in Annals in 2005 \cite{HoL6} Horv\'ath gives a  description of defining sets for Schr\"odinger operators in terms of  completeness problem for exponential functions.
The same  result was independently found in our joint work with N. Makarov and presented (by N. Makarov) at a conference in Stockholm in 2003, devoted to the 75th birthday of L.~Carleson (see also our article \cite{MIFL6} in the proceedings of the conference).

\ms\no Below is a selection of  Horv\'ath's results. We  use the following notation: $\sqrt\Lambda=\{z:z^2\in\Lambda\}$, and
$\sqrt\Lambda\cup\{*,*\}$ means $\sqrt\Lambda$ plus any two points.
Recall that by $E_\L$ we denote the system of exponentials $$\{e^{2\pi i\l }|\ \l\in\L\}.$$

\begin{thm}$\ $
	
	\ms\no  (i)  $\Lambda$ is defining in the class $\Sc(L^r,D)$ iff
	$E_{\sqrt\Lambda\cup\{*,*\}}$ is complete in $L^r(-2,2)$;
	
	\ms\no (ii) $\Lambda$ is defining in $\Sc(L^r,N)$ if
	$E_{\sqrt\Lambda}$ is complete in $L^r(-2,2)$.
\end{thm}

\ms\no  (In the  second case, the "only if" part of Horv\'ath's theorem comes with some additional condition.)

\bs\no Let us explain how to prove
the "if" parts of these statements using our methods, see \cite{MIFL6}.
For example,  (i) in the case $r=2$ can be equivalently reformulated as

\ms\begin{prop} $\Lambda$ is  defining in the class $\Sc(L^2,D)$ if $\sqrt{\Lambda}\cup\{*,*\}$ is  a uniqueness set of $PW_2$. \end{prop}

\ms\begin{proof} Let $q, \ti q\in L^2(0,1)$. Without loss of generality we will assume that the corresponding Schr\"odinger operators with  boundary conditions (D) at 0 and (N)  at 1 are positive. Otherwise, we simply add a large positive constant $a$ to both potentials,
	and  using  the transformation  $$F(z)\mapsto F(\sqrt{z^2+a^2})$$ for even  entire functions we observe  that $\sqrt\Lambda$  is a uniqueness set iff $\sqrt{\Lambda+a}$ is.
	
	\ms\no It is well-known that if $m$ is a Herglotz function such that  $$0<  m<+\infty\qquad {\rm on}\quad\R_-,$$ then  $m^*(\lambda)=\lambda m(\lambda^2)$ is again a Herglotz function.
	If $$\Theta=(m-i)/(m+i),$$ then the  inner function corresponding to $m^*$ is $$\Theta^*(z)=\frac{(z+1)\Theta(z^2)+(z-1)}{(z-1)\Theta(z^2)+(z+1)} $$
	We call $\Theta^*$ the square root transform of $\Theta$.
	
	\ms\no Let  $\Theta^*$ and $\ti\Theta^*(z)$  be the square root transforms  of  $\Theta$ and $\ti\Theta$, the Weyl  functions taken with sign minus,
	see Section 1.8 in \cite{MIFL6}. From the standard asymptotic formula for  solutions of a regular Schr\"odinger equation we obtain
	\begin{equation}\label{310}\frac {\Theta^*}{S^2}=\frac{\bar H}H\quad{\rm on}\quad \R,\qquad H^{\pm 1}\in H^\infty,\end{equation}and
	\begin{equation}\label{311}x[\Theta^*(x)-\ti\Theta^*(x)]\in L^2(\R).\end{equation}
	\ss\no
	(For convenience we reproduce the standard argument at the end of the proof .)
	
	\ms\no If $\ti \Theta= \Theta$  on $\Lambda$, then since $\ti\Theta^*(0)=\Theta^*(0)$, we have
	$$\Theta^*=\ti \Theta^*\qquad{\rm on}\quad \{0\}\cup\sqrt\Lambda,$$
	where we regard $\Theta^*$ and $\ti \Theta^*$ as meromorphic functions in the whole plane.
	By \eqref{311},
	$$(z-1)(\Theta^*-\ti \Theta^*)\in K_{\Theta^*\ti \Theta^*},$$ so
	$\sqrt\Lambda\cup\{0,1\}$ is a zero set of some $K_{\Theta^*\ti \Theta^*}$-function, and  therefore  by \eqref{310} a zero set of some function in $K_{S^4}$ or  $PW_2$ . (For  zeros in $\C_-$  we can use the argument with dual reproducing kernels as in  Section 3.1. of \cite{MIFL6}.) \end{proof}

\bs\no{\it Proof of  \eqref{310}--\eqref{311}}.
If $s>0$, then the solution $u_s(t)$ of the IVP
$$ -\ddot u+qu=s^2u,\qquad u(0)=0,\quad\dot u(0)=1,$$
satisfies the integral equation$$u_s(x)=\sin sx+\frac1s\int_0^x\cos s(x-t)~q(t)~u_s(t)~dt.$$
Iterating, we have
$$u_s(1)=\sin s+\frac{F(s)}s+\frac{R(s)}{s^2},$$
where
$$F(s)=\int_0^1\cos s(1-t)~\sin st~ q(t)~dt,$$
and
$$R(s)=\int_0^1\cos s(1-x)~q(x)~dx\int_0^x\cos s(x-t)~q(t)~u_s(t)~dt.$$
We have an  elementary  a priori bound
$$ |u_s(t)|\le C,\qquad (s>0,\; t\in[0,1]),$$
so $$\forall s,\quad|R(s)|\le \const.$$
On the other hand, $F$ is basically the Fourier transform of a function on $(-1,1)$, and
$$q\in L^2\quad\Ra\quad F\in L^2(\R).$$
We also get the corresponding estimates of $\dot u_s(1)$. The resulting estimates of $\Theta$ imply both statements. $\square$

\lsection{Exercises}

\ms\no 1) Prove the lemma at the end of the first subsection.

\ms\no 2) Formulate an equivalent version of Horv\'ath's theorem for $r=2$ replacing conditions of completeness of exponentials
with conditions that the sequence is a uniqueness set in a space of entire functions.

\ms\no 3) Using the Beurling-Malliavin theorem and Horv\'ath's theorem, formulate an if and only if condition for a sequence
of real points to be a defining set of a Schr\"odinger operator on an interval $[0,1-\e]$ for any $0<\e<1$, with an $L^2$ potential and Dirichlet (Neumann) boundary
condition at 0.

\ms\no 4) Horv\'ath's paper \cite{HoL6} contains a list of classical theorems by Ambarzumian, Borg, Levinson, Hochstadt-Liberman, as well
as more recent results by Gesztesy-Simon and by del Rio-Gesztesy-Simon that follow from the last theorem. It is a good exercise to
try to deduce those statements from Horv\'ath's result.

\newpage

\lecture{7}{Bernstein's problem}

\ms\no In this lecture we return to the classical problems of Harmonic Analysis outlined in the first lecture and discuss
Bernstein's problem on weighted polynomial approximation.

\ms\no Let us recall the statement of the problem.

\ms\no In this lecture we allow the weight function $W$ to be semi-continuous from below instead of continuous as in most classical papers and in our first lecture.
Throughout the rest of the lecture we use the following definition.

\ms\no We say that
a  function $W\geqslant 1$ on $\R$ is a \textit{weight} if $W$ is lower semi-continuous and
$x^n=o(W)$ as $|x|\to\infty$ for any $n\in \N$.

\ms\no Our weights are also allowed to take infinite values at finite points on $\R$, which makes it possible to study approximation on subsets of the line within
the same general formulation of the problem.  For instance, the classical Weierstrass theorem
answers the question of density of polynomials in $C_W$ with $W$ equal to 1 on an interval and infinity
elsewhere. Another important case of the problem is approximation on discrete sequences (see, for instance, \cite{BSL7}), which corresponds to the
weights that are infinite outside of a discrete sequence.

\ms\no With a semi-continuous and $\hat\R$-valued $W$ ($\hat\R=\R\cup\{\infty\}$), the quantity $||f||_W$, defined
on the set of all continuous $f$ such that $f/W\to 0$ at $\pm\infty$ as

\begin{equation}||f||_W=\sup_\R\frac{|f|}W\label{norm7}\end{equation}
(see the first lecture), ceases being a norm and becomes a semi-norm: the set
is no longer complete. Those functions supported on $\{W=\infty\}$ will satisfy
$||f||_W=0$.

\ms\no
The semi-norm defined by \eqref{norm7} can be made a norm following a standard
procedure. First the space of continuous functions $g$, such that $g/W\to 0$ at $\pm\infty$, needs to be factorized to obtain a space of equivalence classes: $f\thicksim g$ if and only if $||f-g||_W=0$.
After that the factor-space needs to be completed. We denote by $C_W$ the resulting space.

\ms\no Note that if $W$ is continuous and takes only finite values,
$C_W$ coincides with the space of continuous functions defined in the first lecture. In the general case, we still have the following property.

\ms\no If $W$ is a weight we say that a measure $\mu$ on $\R$ is $W$-finite if
$$\int Wd|\mu|<\infty.$$

\begin{proposition}
	The dual space of $C_W$ consists of $W$-finite measures.
\end{proposition}

\begin{proof}
	Consider a sequence of continuous weights $W_n$ such that $$W_{n+1}(x)\geqslant W_n(x)$$  and $W_n(x)\to W(x)$ for any $x\in\R$.
	Note that any bounded linear functional $\mu$ on $C_W$ induces a linear bounded functional on $C_{W_n}$ for any $n$.
	Because of monotonicity, $C_{W_n}\subset C_{W_{n+1}}$.
	Since any linear bounded functional on $C_{W_n}$ can be identified with a $W_n$-finite measure, again using monotonicity
	of $W_n$, we conclude that $\mu$ can be identified with a $W$-finite measure on the set $\cup C_{W_n}$. Since the last
	set is dense in $C_W$ (or, more precisely, the set of equivalence classes, containing the elements from $\cup C_{W_n}$,
	is dense in $C_W$), $\mu$ can be identified with a $W$-finite measure on the whole $C_W$.

\end{proof}

\ms\no Note that in the general case of semi-continuous $\hat\R$-valued weights, when we say that polynomials are
not dense in $C_W$ that statement  still means  that there exists a continuous $g$ and $\e>0$ such that $ g/W\to 0$ at $\pm\infty$ and
$||g-p||_W>\e$ for every polynomial.
The crucial dual statement, that characterizes non-completeness in the case of continuous weights, still holds for general $W$:
Polynomials are not dense in $C_W$ if and only if there exists a non-zero $W$-finite measure that annihilates polynomials.

\ms\no Here is a well-known fact in the theory. The notation $K\mu$ stands for the Cauchy integral of $\mu$ in $\C_+$:

$$K\mu(z)=\int\frac{d\mu(x)}{x-z}.$$

\ms\no Recall that a measure $\mu$ has finite moments if $x^n\in L^1(|\mu|)$ for all $n=0,1,2, ...\ $.

\begin{lemma}\label{growth}
	A measure $\mu$ with finite moments annihilates polynomials if and only if
	$$K\mu(iy)=o(y^{-n})$$
	for any $n>0$ as $y\to\infty$.
\end{lemma}

\begin{proof} Suppose that $\mu$ annihilates polynomials. Since
	$(t^n-z^n)/(t-z)$ is a polynomial of $t$ for every fixed $z$,
	$$0=\int\frac{t^n-z^n}{t-z}d\mu(t)=[Kt^n\mu](z)-z^nK\mu(z).$$
	Since any Cauchy integral of a finite measure tends to zero along $i\R_+$, so does $Kt^n\mu$. Hence $K\mu(z)=o(z^{-n})$
	as $z\to\infty,\ z\in i\R_+$.
	
	\ms\no Conversely, suppose that $K\mu(iy)=o(y^{-n})$
	for any $n>0$ as $y\to\infty$.
	Without loss of generality, $\mu$ is real (otherwise consider $\mu-\bar\mu$ or $i(\mu+\bar\mu)$).
	Then $$K\mu(-iy)=\overline{K\mu(iy)}=o(y^{-n})$$ as well.
	Since $\mu$ has finite moments we may consider the function
	$$H(z)=\int\frac{t^n-z^n}{t-z}d\mu(t).$$
	It is easy to show that $H$ is entire of exponential type zero.
	Noticing again that
	$$H(z)=[Kt^n\mu](z)-z^nK\mu(z),$$
	we see that $H$ is bounded on $i\R$. Hence $H$ is a constant by the Phragm\'en-Lindel\"of principle. Since $H(iy)$ tends to zero, $H$ is zero. Therefore
	$$z^nK\mu(z)=[Kt^n\mu](z)=\int\frac{t^n}{t-z}d\mu(t).$$
	Putting $z=0$ in the last equation we get that $\mu$ annihilates $t^{n-1}$ for any $n>0$.
\end{proof}

\lsection{Equivalence between weighted uniform and $L^p$-approximation}\label{sectionBakan}

\ms\no In 1924 when Bernstein published his problem the $L^p$-spaces did not play the same dominating role in analysis as they do now. In later years
many approximation problems in weighted situations were replaced with problems on $L^p$-approximation. Nonetheless, the original form
of Bernstein's problem has survived all the major changes in analysis over the last hundred years and is still used today.
One of the reasons for such  longevity is that it implies its more modern $L^p$-reformulations!

\ms\no Close connections between  $L^p$- and weighted uniform approximation
have been known to the experts for a long time. Nevertheless, the formal result that
reduces the problem of polynomial approximation in $L^p$-spaces to Bernstein's problem
was found by A. Bakan only in 2008. This result allows us to concentrate on the latter problem for the rest of the lecture.

\begin{theorem}\cite{BakanL7}\label{tBakan}
	Let  $0< p <\infty $ be a constant and let $\mu$ be a positive finite measure on $\R$ such that $L^p(\mu)$ contains all polynomials. Polynomials are dense in $L^p(\mu)$
	if and only if $\mu$ can be represented as $\mu=W^{-p}\nu$ for some finite positive measure $\nu$ and a weight $W$ such that
	polynomials are dense in $C_W$.
\end{theorem}

\ms\no Let us point out that the weights appearing in the theorem are lower semi-continuous. Hence, to study the $L^p$- and uniform versions as one problem
one needs the general definition of $C_W$ discussed in this lecture, as opposed to its more
traditional version with a continuous $W$. Here is a short proof of Bakan's result.


\begin{proof}
	If polynomials are dense in $C_W$ for some weight $W$ such that $\mu=W^{-p}\nu$
	then for any bounded continuous function $f$ there
	exists a sequence of polynomials $\{s_n\}$
	such that $s_n/W$ converges to $f/W$ uniformly. Then
	$$\int |f-s_n|^p d\mu=\int \frac{|f-s_n|^p}{W^p}W^p d\mu=\int \left|\frac fW-\frac{s_n}W\right|^p d\nu\to 0.$$
	Hence polynomials are dense in $L^p(\mu)$.
	
	\ms\no Suppose that polynomials are dense in $L^p(\mu)$. Let $\{f_n\}_{n\in\N}$ be a  set
	of bounded continuous functions on $\R$, that is dense
	in any $C_W$ (see exercises). Let $\{s_{n,k}\}_{n,k\in\N}$
	be a family of polynomials such
	that
	$$||f_n-s_{n, k}||_{L^p(\mu)}<4^{-(n+k)}.$$
	Define
	$$W=1+\sum_{n,k\in \N}2^{n+k}|f_n-s_{n, k}|.$$
	Notice that then  $W\in L^p(\mu)$, $W$ is lower semi-continuous and $s_{n, k}/W\to f_n/W$ uniformly as $k\to\infty$.
	Without loss of generality, $L^p(\mu)$ is not finite dimensional. Then $\{s_{n,k}\}$ contains polynomials of arbitrarily
	large degrees and $x^n=o(W)$ for any $n$. Thus $W$ is a weight.
	Since $\{f_n\}$ is dense in $C_W$, polynomials are dense in $C_W$. The measure $\nu$ can be chosen as $W^p\mu$.

\end{proof}

\ms\no Most of the results on Bernstein's problem belong to one of the two following groups. The first group, containing classical theorems by Akhiezer, Mergelyan and Pollard as well as more
recent results by Koosis, provides conditions on $W$ in terms of the norms of point evaluation functionals. The second group uses the approach pioneered by
de Branges (see \cite{dBr1} or theorem 66 in \cite{dBL7}) and further developed by Borichev, Sodin and Yuditski. These results are formulated in terms of existence of entire
functions belonging to certain classes (see references given in the first lecture).

\ms\no Both approaches have produced significant progress towards a full solution, although the conditions of density
remained rather implicit. Besides specific examples, the only general explicit results in the literature are a classical theorem by Hall \cite{Hall} and
a theorem on log-convex weights published by Carleson \cite{Carleson}, see  below.

\ms\no Our goal for the rest of this lecture is to discuss an example of an explicit result on Bernstein's problem. We discuss a theorem obtained in \cite{Poly}. It is closely related to a theorem of de Branges \cite{dBL7} that gives an answer in terms of zero sets of entire functions.
Before we state the result
we need the following definitions.

\lsection{Characteristic sequences}\label{char}

\ms\no Recall that a real sequence is \textit{discrete} if it does not have finite accumulation points. To simplify the definitions we will always assume
that a discrete sequence is infinite and does not have multiple points. A discrete sequence is called one-sided if it is bounded from below or from above
and two-sided otherwise.

\ms\no If $\L=\{\lan\}$ is a discrete sequence we will always assume that it is \textit{enumerated in the natural order}, i.e.
$\lan<\l_{n+1}$, non-negative elements are indexed with non-negative integers and negative
elements with negative integers.

\ms\no For instance, if $\L=\{\lan\}_{n\in\Z}$ is a two sided sequence then
$$...\l_{-n-1}<\l_{-n}<...<\l_{-1}<0\leqslant \l_0<\l_1<...\lan<\l_{n+1}<...$$
Thus a one-sided sequence bounded from below (above) will be enumerated with $n\in\Z, n\geqslant-N$ ($n\in\Z, n<N$), where $N$
is the number of negative (non-negative) elements in the sequence.

\ms\no As before, we  say that a sequence $\L=\{\lan\}$ has upper density $d$ if
$$\limsup_{A\to\infty}\frac{\#[\L\cap (-A,A)]}{2A}=d.$$
If $d=0$ we say that the sequence has zero density.

\ms\no A discrete sequence $\L=\{\lan\}$ is called \textit{balanced} if the limit
\begin{equation}
\lim_{N\to\infty}\sum_{|n|<N}\frac {\l_n}{1+\lan^2}
\label{balance}
\end{equation}
exists.

\ms\no Observe that any even sequence (any sequence $\L$ satisfying $-\L=\L$) is balanced.
So is any two-sided sequence sufficiently close to even.
At the same time, a one-sided sequence
has to tend to infinity fast enough to be balanced (the series $\sum \lan^{-1}$ must converge).

\ms\no Let $\L=\{\lan\}$ be a balanced sequence of finite upper density. For each $n, \ \lan\in\L,$ put
$$p_n= \frac 12\left[\log(1+\l_n^2)+\sum_{n\neq k, \ \l_k\in\L}\log\frac{1+\l_k^2}{(\l_k-\lan)^2}\right],$$
where the sum is understood in the sense of principal value, i.e. as
$$\lim_{N\to\infty}\sum_{0<|n-k|<N}\log\frac{1+\l_k^2}{(\l_k-\lan)^2}.$$
We will call the sequence of such numbers $P=\{p_n\}$ the \textit{characteristic sequence} of $\L$.

\ms\no Note that for a sequence of finite upper density the last limit exists for every $n$ if and only if
it exists for some $n$ \textit{if and only if the sequence is balanced}.

\ms\no The paper \cite{Poly} contains the following result on Bernstein's problem.
Recall that per our agreement all sequences  are assumed to be infinite.

\begin{theorem}\label{mainBernstein}
	Polynomials are not dense in $C_W$ if and only if there exists a balanced sequence $\L=\{\lan\}$
	of zero density such that $\L$ and its characteristic sequence $P=\{p_n\}$ satisfy
	\begin{equation}
	\sum W(\lan)\exp (p_n)<\infty.
	\label{condition}
	\end{equation}
\end{theorem}

\ms\no The proof is not difficult but is still too long to include in this short course. We send interested readers to
\cite{Poly}. In the rest of this lecture let us discuss some implications and relations of the above result.

\lsection{Examples and corollaries}\label{exandcor}

\ms\no This section contains further discussion of theorem \ref{mainBernstein} including its relations with some of the known results.

\ms\no A classical theorem by Hall \cite{Hall} says that if
$$\int_{-\infty}^\infty\frac{\log W}{1+x^2}dx<\infty$$
for a weight $W$ then polynomials are not dense in $C_W$.
Indeed, if $F$ is an outer function in $\C_+$ satisfying
$$|F|=\frac 1{(1+x^2)W},$$ then the measure $e^{ix}F(x)dx$ is a $W$-finite
measure that annihilates polynomials by lemma \ref{growth}.

\ms\no A direct inverse to this statement is false. Even if one requires
that $\log W$ is poisson unsummable and $W$ is monotone on $\R_\pm$,
the polynomials may still not be dense in $C_W$, as follows from an example given
in \cite{KoosisL7}.

\ms\no
We say that
$f:E\subset\R_+\to\R$ is log-convex if it is convex as a function of $\log x$,
i.e. if the function $g(t)=f(e^t)$ is convex on $S=\log E=\{\log x|\ x\in E\}$. In particular, a twice differentiable
function $f$ is log-convex on an interval $(a,b)\subset \R_+$ if $f'(x)+xf''(x)\geqslant 0$ for all $x\in (a,b)$.

\ms\no The following classical result, published by L. Carleson in \cite{Carleson}, but seemingly known earlier to several other mathematicians (see for instance \cite{IK}), is
a partial inverse to Hall's theorem.

\begin{theorem}\label{Carleson}
	Let $W$ be an even weight that is log-convex on $\R_+$. Then polynomials are not dense in $C_W$ if and only if $\log W\in L^1(\Pi)$.

\end{theorem}

\begin{proof}
	If $S=\{s_n\}$ is  an even discrete sequence of finite density denote by $v_S$ the function
	$$v_S(x)=\frac12\sum \log\left|\frac{(s_n-x)^2}{1+s_n^2}\right|,$$
	where the sum is understood in terms of normal convergence of partial sums $\sum_{|n|<N}$ in $\C\setminus\L$. Simple computations show that $-v_S$ is log-convex on every interval $(s_n,s_{n+1}),\ s_n\geqslant 0$.
	
	\ms\no To prove the theorem, notice that in one direction it follows from Hall's result. In the opposite direction, suppose that
	polynomials are not dense in $C_W$. Then  there exists a sequence $\L$ like in the statement of  theorem \ref{mainBernstein}. It is not difficult
	to prove that
	$\L$ can be chosen to be even.
	
	\ms\no Fix $n>0$ and denote $\Gamma_n=\L\setminus\{\lan,\l_{-n},\l_{n+1},\l_{-n-1}\}$. Then \eqref{condition} implies
	$$\log W(\l_k)\leqslant v_{\Gamma_n}(\l_k) + \frac 12\log_-\frac{(\lan-\l_{n+1})^2}{1+\l_{n+1}^2}+\const,\textrm{ for }k=n,n+1.$$
	Since both $W$ and $-v_{\Gamma_n}$ are log-convex on $(\lan,\l_{n+1})$ the inequality can be extended to the whole interval
	$(\lan,\l_{n+1})$ for every $n$. Since $v_\L\in L^1(\Pi)$, the quantity
	$$\sum_n\int_{\lan}^{\l_{n+1}}|v_\L-v_{\G_n}|d\Pi$$
	is finite  and $\log W\geqslant 0$, this implies that $\log W\in L^1(\Pi)$.

\end{proof}

\ms\no A direct proof of the log-convex theorem can be found in \cite{KoosisL7}.

\lsection{Asymptotics of characteristic sequences and applications}\label{charass}

\ms\no
Let $u$ be a monotone increasing function on $\R$.
Suppose that the harmonic conjugate function $\tilde u$ is Poisson-summable, i.e. $\tilde u\in L^1(\Pi)$.
(Recall that $d\Pi=dx/(1+x^2)$.)
Let $\L=\{\lan\}$ be a sequence such that
$u(\lan)=n\pi$.

\ms\no It is not difficult to show that then $\L$ is a zero density balanced sequence.
(This condition is actually equivalent to $\tilde u\in L^1(\Pi)$.)
Let $P=\{p_n\}$ be the characteristic sequence
of $\L$.

\ms\no Elementary estimates yield:

\begin{proposition}\label{conjugate}
	Suppose that $u'(x)$ exists and is bounded for large enough $|x|$. Then
	$$p_n= \ti u(\lan) +O(\log |\lan|)$$
	as $|n|\to\infty$.
\end{proposition}



\ms\no Theorem \ref{mainBernstein} gives the following

\bs
\begin{corollary}
	$$$$
	I) If $W$ is a regular weight such that $\log W(\lan)\leqslant \ti u(\lan) + O(\log |\lan|)$ then polynomials are not dense in $C_W$.

	\bs\no II) If $\mu=\sum \alpha_n\delta_{\lan}$ is a finite positive measure such that
	$$\sum \alpha_n^{1-q}\exp{ q p_n}<\infty$$
	for some $1<q<\infty$ then polynomials are not dense in $L^p(\mu), \frac 1p+\frac 1q=1$.
	
	\bs\no III) If
	$$\alpha_n=O(\exp{p_n})$$
	then polynomials are not dense in $L^1(\mu)$.
\end{corollary}

\ms\no For many examples of discrete sequences $\L$ one can easily find a suitable function $u$ and the values of its conjugate
at $\L$.
If, for instance, $\L=\{n^{1/\alpha}\}_{n\geqslant 0},  \ 0<\alpha<1/2$ then one may consider $u$ defined as
$$u(x)=\begin{cases}\pi x^\alpha\textrm{ if }x\in\R_+\\
0\textrm{ if }x\in\R_-
\end{cases}$$
and
find that
$$\ti u(n^{1/\alpha})=-\pi n\tan\left(\alpha\pi-\frac \pi2\right).$$
In the two-sided case $\L=\{\pm n^{1/\alpha}\}_{n\geqslant 0}, \ 0<\alpha<1$,  one may use $u$ defined as
$$u(x)=\begin{cases}\pi x^\alpha\textrm{ if }x\in\R_+\\
-\pi|x|^\alpha\textrm{ if }x\in\R_-
\end{cases}.$$
Then
$$\ti u(\pm n^{1/\alpha})=-\pi n\tan\left(\alpha\frac\pi 2-\frac\pi 2\right).$$
Such simple calculations and estimates,
together with statements from this section, yield the majority of the examples
of discrete measures, whose $L^p$ spaces are not spanned by polynomials, existing in the literature.
See \cite{BSL7} for more examples.

\lsection{Exercises}

\ms\no 1) Let $\mu$ be a finite positive measure concentrated on $\Z$, $\mu=\sum_{n\in\Z} e^{-\sqrt{|n|}}\delta_n$.
Show
that polynomials are not dense in $L^p(\mu)$ for any $0<p\leq\infty$. How much smaller can we make the point masses
for this statement to still hold? (Hint: switch to Bernstein's form and find a way to use Hall's or log-convex theorem. No complete answer to the last question is expected.)

\ms\no 2) In the definition of $C_W$ one requires that all functions from that set satisfied $f/W\to 0$ at $\pm\infty$.
Why? Similarly, why can't we drop the condition that $W$ is lower semi-continuous?

\ms\no 3) Produce a countable set of bounded continuous functions on $\R$ that is dense in $C_W$ for any weight $W$
(we needed such a set in the proof of Bakan's theorem).

\ms\no 4) Consider $\L=\{n^3\}_{n>0}$. Let the weight $W$ be defined as $n^\gamma$ at $n^3$ and as $\infty$ outside of $\L$.
Using the results from this lecture, discuss for what real $\gamma$ the polynomials will be dense (not dense) in $C_W$.
Using Bakan's theorem, reformulate your statements in terms of $L^p$ approximation.

\newpage

\lecture{8}{The Gap Problem}

\ms\no Out of the three classical completeness problems  formulated in the first lecture it remains to discuss the Type Problem.
A solution to the Type Problem was found in \cite{TypeL8} and we plan to present it in our lectures. As it turns out, to approach the Type Problem one needs first to treat another well-known problem of Fourier Analysis, the so-called Gap Problem, which we will consider in this lecture.

\ms\no First, let us recall the statement of the Type Problem. We consider
the family $\EE_\L$ of exponential functions $\exp(2\pi i\l t)$ on $\R$ whose frequencies $\l$ belong to a certain
set $\L\subset \C$:
$$\EE_\L=\{\exp(2\pi i\l t)|\ \l\in \L\}.$$
In particular, we denote by $\EE_a=\EE_{[0,a]}$ the family of exponential functions whose frequencies belong to the interval
from 0 to $a$.
If $\mu$ is a finite positive measure on $\R$ we denote by $T_\mu$ its exponential type that is defined as
\begin{equation} T_\mu=\inf\{\ a>0\ |\ \EE_a  \textrm{ is complete in }  L^2(\mu)\ \}\label{type8}
\end{equation}
if the set of such $a$ is non-empty and as infinity otherwise. The type problem asks to calculate $T_\mu$ in terms
of $\mu$. Various reformulations of this problem appear in many fields of analysis. We discussed some of such connections in the first lecture. For more information see \cite{DymL8, TypeL8, BSL8}.

\lsection{General case $p\neq 2$}

\ms\no The family $\EE_a$ is incomplete in $L^2(\mu)$ if and only if there exists a function $f\in L^2(\mu)$ orthogonal to all
elements of $\EE_a$. Expanding  to other $1\leqslant p\leqslant \infty$ we define

\begin{equation} \GG^p_\mu=\sup\{\ a\ |\  \exists \ f\in L^p(\mu),\int f(x)e^{2\pi i\l x}d\mu(x)=0,  \forall \ \l\in[0,a]\ \}.
\label{typep8}\end{equation}

\ms\no We put $\GG^p_\mu=0$ if the set in \eqref{typep8} is empty. By duality, for $1<p\leqslant\infty$, $\GG^p_\mu$ can still be defined as the infimum of $a$ such that $\EE_a$ is complete in $L^q(\mu),\ \frac 1p +\frac 1q =1$. In particular, $\GG^2_\mu=T_\mu$.
The cases $p\neq 2$ were considered in several papers, see for instance articles by Koosis \cite{Koosis2} or Levin \cite{Levin2} for the case $p=\infty$ or \cite{GAP} for $p=1$.

\ms\no Since $\mu$ is a finite measure we have
\begin{equation}\GG^p_\mu\leqslant\GG^q_\mu\textrm{ for }p\geqslant q.\label{pq}\end{equation} Apart from this obvious observation, the problems of
finding $\GG^p_\mu$ for different $p$ were generally considered non-equivalent until the solution of the Type Problem! One of the consequences of
the main result of \cite{TypeL8} that we will discuss later, is that, in some sense, there are only two significantly different cases, $p=1$ (the gap problem)
and $1<p\leqslant\infty$ (the general type problem).

\lsection{The gap problem}\label{introGAP}

\ms\no Not only is the case $p=1$ important and interesting by itself, but, as was mentioned before, it seems to be a necessary step towards a solution for the Type Problem, $p=2$. Let us start with the following reformulation of the Gap Problem.

\ms\no Let $X$ be a closed subset of the real line. Denote
$$\GG_X=\sup\{\ a\ |\ \exists\   \mu\neq 0, \ \supp\mu\subset X,   \hat\mu=0   \textrm{ on }[0,a]\ \}.$$
Here and in what follows $\hat\mu$ denotes the (inverse) Fourier transform of a finite measure $\mu$ on $\R$:
$$\hat\mu(z)=\int_\R e^{2\pi izt}d\mu(t).$$
As was shown in \cite{GAP}, for any finite measure $\mu$ on $\R$, $\GG^1_\mu$, as defined in the previous section, depends only on its support:
$$\GG^1_\mu=\GG_X,\ X=\supp\mu.$$
This property separates the gap problem from all the cases $p>1$. (See exercises.)

\ms\no For a long time both the gap problem and the type problem were considered
by experts to be "transcendental," i.e. not having a closed form solution.
Following an approach developed in \cite{MIF1} and \cite{MIF2L8},
a solution to the gap problem was suggested in \cite{GAP}, see below.

\lsection{Classic examples}

\ms\no As before, we say that a function $f$ on $\R$ is Poisson-summable if it is summable with respect to the Poisson measure $\Pi$,
$$d\Pi=dx/(1+x^2).$$
We say that a sequence of real numbers $A=\{a_n\}$ is discrete if
it does not have finite accumulation points.
We always assume that a discrete sequence is enumerated
in the natural increasing order:  $a_n\geqslant a_{n-1}$. Since the sequences considered here
have $\pm\infty$ as their density points, the indices run over $\Z$.
In most of our statements and definitions, the sequences do not have multiple points.
We call a discrete sequence $\{a_n\}\subset \R$ separated if $|a_n-a_k|>c$ for some $c>0$ and any $n\neq k$.

\ms\no The following  statement combines results by Krein  (part I in the statement below, case $p=2$) and
by Levinson and McKean  (part II, $p=2$).

\begin{theorem}[Krein \cite{Krein1L8}, Levinson-McKean \cite{DML8}]\label{KLM8}
	Let $\mu$  be a finite measure on $\R$, $\mu=w(x) dx$, where $w(x)\geqslant 0$.
	Then
	
	\ms\no I) If $\log w$ is Poisson-summable then for any $1\leqslant p\leqslant \infty$, $\GG^p_\mu=\infty$.
	
	\ms\no II) If $\log w$ is monotone and Poisson-unsummable on a half-axis $(-\infty,x)$ or $(x,\infty)$ for some $x\in\R$ then for any $1< p\leqslant \infty$, $\GG^p_\mu=0$.
	
\end{theorem}

\ms\no (See Exercises.)

\ms\no A theorem by Duffin and Schaeffer \cite{DS} implies that if $\mu$ is a measure such that
for any $x\in \R$
$$\mu([x-L,x+L])>d$$
for some $L,d>0$ then $\GG^2_\mu\geqslant 1/L$.

\ms\no For discrete measures, in the case $\supp\mu=\Z$,  a deep result by Koosis shows an
analogue of Krein's result: if $\mu=\sum w(n)\delta_n$, where
$$\sum\frac{\log w(n)}{1+n^2}>-\infty,$$
then $\GG^p_\mu=1$ for all $p, \ 1\leqslant p\leqslant\infty$ \cite{Koosis2}. Not much was known about supports other than $\Z$ besides a result from \cite{PolyaL8}, which implies that if
$$\mu=\sum \frac{\delta_{a_n}}{1+a_n^2}$$
for a separated sequence $A=\{a_n\}\subset\R$ then $\GG^p_\mu= D_*(A)$, where $D_*$ is the interior Beurling-Malliavin density of $A$, see lecture 2 for the definition.

\ms\no In addition to these few examples, classical theorems by Levinson-McKean, Beurling and de Branges
show that if a measure has long gaps in its support or decays too fast, then $\GG^p_\mu=0$. We will discuss these theorems in our
next lecture.
Examples of measures of positive type  can  be constructed using the results by Benedicks \cite{BenedicksL8}.
The most significant further development, that allows one to modify existing examples, is the result by Borichev and Sodin
\cite{BSL8},
which says that "exponentially small" changes in weight or support do not change the type of a measure.

\lsection{The gap problem and $d$-uniform sequences}\label{secGAP}

\ms\no It is not difficult to calculate the gap characteristic of an arithmetic progression $\L=a+dn,\ a\in\R,\ d>0$: $\GG_\L=1/d$, see exercises.
It follows that if $X$ contains an arithmetic progression $\L$ then $\GG_X\geq 1/d$. It would be nice if $\GG_X$ for a general $X$
could be calculated as a supremum of such numbers taken over all arithmetic progressions contained in $X$. Unfortunately, this is not the case.
But, as it turns out, this simple idea is the right step towards a solution. We just need to replace arithmetic progressions with a slightly larger
class of sequences, the $d$-uniform sequences defined in this section.

\ms\no Let $\Lambda=\{\lambda_1,...,\lambda_n\}$ be a finite set of distinct points on $\R$. Define

\begin{equation}E(\Lambda)=\sum_{\lambda_k,\lambda_j\in\L, \ k\neq j} \log|\lambda_k-\lambda_j|.\label{electrons}\end{equation}
According to the 2D Coulomb law, the quantity $E(\L)$ can be interpreted as potential energy of the system of "flat electrons" placed at $\L$,
see \cite{GAP}. That observation motivates the term we use for the condition \eqref{energy8} below.
\bs

\ms\no The following example is included to illustrate our next definition.

\bs\no \textbf{Key example:}

\ms\no \textit{Let $I\subset \R$ be an interval and
	let $\L=d^{-1}\Z\cap I$ for some $d>0$.
	Then
	$$\Delta=\#\L=
	d|I|+O(1)$$
	and
	\begin{equation}E=E(\Lambda)=\sum_{1\leqslant m\leqslant \Delta}\log\left[d^{-\Delta+1}(m-1)!(\Delta-m)!\right]=
	\Delta^2\log |I| +  O(|I|^2)\label{eqkey}\end{equation}
	as follows from Stirling's formula. Here  the notation $O(\cdot)$ corresponds to the direction $|I|\to\infty$.
}

\begin{remark}
	The uniform distribution of points on the interval does not maximize
	the energy $E(\Lambda)$ but comes within $O(|I|^2)$ from the maximum, which is negligible
	for our purposes, see the main definition and its discussion below. It is interesting to observe that the
	maximal energy for $k$ points is achieved when the points are placed at the endpoints of $I$ and the zeros of
	the Jacobi $(1,1)$-polynomial of degree $k-2$, see for example \cite{Kerov}.
\end{remark}

\ms\no Let
$$...<a_{-2}<a_{-1}<a_0=0<a_1<a_2<...$$
be a discrete sequence of real points. We say that the intervals $I_n=(a_n,a_{n+1}]$ form a short partition
of $\R$ if $|I_n|\to\infty$  as $ n\to \pm\infty$ and the sequence $\{I_n\}$ is short.

\bs\no \textbf{Main Definition:}

\ms\no Let $\L=\{\lan\}$ be a discrete sequence of real points. We say that $\L$ is
$d$-uniform if
there exists a short partition $\{I_n\}$ such that
\begin{equation} \Delta_n= d|I_n|+o(|I_n|)\ \ \   \text{for all} \ \ \ n\ \ \textrm{(density condition)}\label{density8}\end{equation}
as $n\to\pm\infty$ and
\begin{equation} \sum_n \frac{\Delta_n^2\log|I_n|-E_n}{1+\dist^2(0,I_n)}<\infty\ \ \textrm{(energy condition)}\label{energy8}\end{equation}
where $\Delta_n$ and $E_n$ are defined as
$$\Delta_n=\#(\L\cap I_n)\ \ \text{ and }\ \ E_n=E(\Lambda\cap I_n)=\sum_{\lambda_k,\lambda_l\in I_n,\ \lambda_k\neq\lambda_l}\log|\lambda_k-\lambda_l|.$$

\ms\begin{remark}
	Note that the series in the energy condition is positive:  every term in the sum defining $E_n$ is at most $\log|I_n|$ and there
	are  fewer than $\Delta_n^2$ terms.
	
	\ms\no As follows from the example above, the first term in the numerator of \eqref{energy8} is approximately equal to
	the energy of $\Delta_n$ electrons spread uniformly over $I_n$. The second term is the energy of electrons placed at
	$\L\cap I_n$. Thus the  energy condition
	is a requirement that the placement of the points of $\Lambda$ is close to uniform, in the sense that the work needed to
	spread the points of $\L$ uniformly on each interval is
	summable with respect to the Poisson weight. For a more detailed discussion of this definition see \cite{GAP}
\end{remark}

\ms\no In \cite{GAP}, $d$-uniform sequences were used to solve the gap problem mentioned in the introduction.
Recall that with any closed $X\subset\R$ one can associate its (spectral) gap characteristic $\GG_X$ defined as
the supremum of the size of the spectral gap taken over all finite non-zero measures supported on $X$.
The main result of \cite{GAP} is the following statement:

\begin{theorem}\label{mainGAP}\cite{GAP} Let $X$ be a closed set on $\R$. Then
	$$\GG_X=\sup\{\ d\ |\  X \textrm{ contains a }
	d-\textrm{uniform sequence }\}.$$
\end{theorem}

\ms\no  Recall that, as was proved in \cite{GAP}, $\GG_X=\GG^1_\mu$ for any $\mu$ such that $\supp\mu=X$.

\ms
\begin{remark}\label{rem1}
	
	$$$$
	
	\begin{itemize}

		\item If $\L$ is a $d$-uniform sequence then $D_*(\L)=d$, as follows easily from  the density condition \eqref{density8}.
		
		\item Among other things, the energy condition ensures that the points of $\L$ are not too close to each other. In particular, if
		$\L$ is $d$-uniform for some $d>0$ and
		$\L'=\{\lambda_{n_k}\}$ is a subsequence such that for all $k$,
		$$\lambda_{n_k+1}-\lambda_{n_k}\leqslant e^{-c|\lambda_{n_k}|}$$
		for some $c>0$, then $D_*(\L')=0$.
		
		\item An exponentially small perturbation of a $d$-uniform sequence contains a $d$-uniform subsequence. More precisely,
		if $c>0$ and $\L$ is a $d$-uniform sequence then any sequence $A=\{\alpha_n\}$ such that
		$|\lan-\alpha_n|\leqslant e^{-c|\lan|}$ contains a $d$-uniform subsequence $A'$ consisting of
		all $\alpha_{n_k}$ such that
		$$\lambda_{n_k+1}-\lambda_{n_k}\geqslant e^{-(c-\e)|\lambda_{n_k}|}.$$
		
		\item As discussed in \cite{GAP}, the energy condition always holds for separated sequences. If $\L$ is separated
		then it is $d$-uniform if and only if $D_*(\L)=d$.

	\end{itemize}
\end{remark}

\lsection{Exercises}

\ms\no 1) Show that $$\GG^1_\mu=\GG_X,\ X=\supp\mu.$$ (This is proposition 1 in \cite{GAP}.)

\ms\no 2) Prove the statements in the last remark.

\ms\no 3) The following statement connects the size of the spectral gap of a measure with the asymptotic behavior of its Cauchy integral.
It is similar to the lemma in the last lecture on the measures that annihilate polynomials:

\begin{lemma} Let $\mu$ be a measure with finite total variation.
	Then the Fourier transform of $\mu$ vanishes on $[-a,a]$ if and
	only if $$ \lim_{y \to \pm\infty }e^{xy}
	\int{\frac{d\mu(t)}{t-iy}} =0,$$ for every $x \in [-a,a].$
\end{lemma}

\ms\no Try to prove this statement. (This is lemma 2 in \cite{PolyaL8}.)

\ms\no 4) Show that $\GG_\Z=1$. (Hint: show that $\csc(\pi z)$ is a Schwarz integral of a Poisson-finite measure. Make an adjustment
to obtain a Cauchy integral and use the last exercise.) Obtain the formula for $\GG_\L$, where $\L$ is an arithmetic progression.

\ms\no 5) Prove the classical result by Krein \cite{Krein1L8} which says that if $d\mu=w(x)dx$ and $\log w$ is Poisson-summable then
$\GG^p_\mu=\infty$ for all $p, \ 1\leqslant p\leqslant\infty$. (Hint: consider the outer function
$W=e^{\SS w}$, where $\SS w$ denotes the Schwarz integral defined in lecture 3. Then for any $a>0$ the measure
$e^{2\pi a iz}Wdx$ annihilates the family of exponentials $\EE_a$.)

\newpage

\lecture{9}{Classical gap theorems}

\lsection{Short proofs for classical gap theorems}

\ms\no The proof of the gap theorem presented in the last lecture is too long and technical to include in this course. To give
the reader some taste of the proofs, in this lecture we discuss classical theorems by Krein, Levinson and McKean, Beurling and de Branges
on the same subject.

\ms\no We first formulate an auxiliary statement, theorem \ref{BerLev} below,
and give it a short elementary proof. We then show how to deduce the classical results from theorem \ref{BerLev},
thus giving simple proofs to those theorems as well. Instead of deducing the classical theorems
from each other we prefer to give each a direct closed proof through theorem \ref{BerLev}, which itself could
be considered an equivalent reformulation of de Branges' theorem \ref{Branges} below.

\ms\no In our estimates we write $a(n)\lesssim b(n)$ if $a(n)< C b(n)$ for some positive
constant $C$, not depending on $n$, and large enough $|n|$. Similarly, we write $a(n)\asymp b(n)$ if $c a(n)< b(n)<C a(n)$
for some $C\geq c>0$. Some formulas will have other parameters in place of $n$ or no parameters at all.

\ms\no Recall that a sequence of disjoint intervals $\{I_n\}$ on the real line is long (in the sense of Beurling and Malliavin)
if
\begin{equation}\sum_n\frac{|I_n|^2}{1+\dist^2(0,I_n)}=\infty\label{long}
\end{equation}
where $|I_n|$ stands for the length of $I_n$.
If the sum is finite we call $\{I_n\}$ short.

\ms\no If $I$ is an interval on $\R$ and $C>0$ we denote by $CI$ the interval with the same center as $I$ of length $C|I|$.

\begin{theorem}\label{BerLev}
	Let $\mu$ be a finite measure on $\R$ whose Fourier transform vanishes on an interval. Suppose that there exists a sequence of disjoint intervals
	$\{I_n\}$ such that
	
	\begin{equation} \sum \frac{|I_n|}{1+\dist^2(I_n,0)}\min\left(|I_n|,\log\frac 1{|\mu|(I_n)}\right)=\infty.
	\label{bleq}
	\end{equation}
	
	Then $\mu\equiv 0$.
	
\end{theorem}

\ms\no Roughly speaking, the formula in \eqref{bleq} and the conclusion of the theorem say that if $\mu$ decays fast along
a large sequence of intervals, then it cannot have a spectral gap (unless it is identically zero).
In a sense, this is a hybrid of a theorem by Beurling, which says that a measure with a spectral gap may not
vanish on a large sequence of intervals, and a theorem by Levinson, which says that such a measure may not
decay fast along the whole line (see below).

\ms\no The proof borrows an idea from the proof of Beurling's gap theorem by Benedicks in \cite{BenedicksL9}.

\begin{proof} Without loss of generality  $|I_n|>1$ for all $n$, because the sum in \eqref{bleq} taken over all intervals of length less than $1$ is finite.
	Suppose that $\hat\mu$ vanishes on $[-a,a]$. Then, once again, its Cauchy integral $K\mu$ is divisible by $e^{iaz}$ in $\C_+$,
	in the sense that
	$$K\mu=e^{iaz}K\nu,$$
	where $\nu$ is a finite measure, $\nu=e^{-iaz}\mu$, see for instance  lemma 2 in \cite{PolyaL9} that was already discussed in previous lectures.
	
	\ms\no Denote by $J_n$ the interval on $\R+i$:
	$$J_n=\left\{z\ | \ \Im z =1,\ \Re z \in \frac 12I_n\right\}.$$

	\ms\no Denote by $\mu_n$ the restriction of $\mu$ on $I_n$ and put $\eta_n=\mu-\mu_n$. Notice that
	$K\eta_n(z)$ is holomorphic in $(\C\setminus\R)\cup I_n$. Hence
	$-\log |K\eta_n(z)|$ is superharmonic in $\{|z-\xi|\leqslant |I_n|/4\}$ for any $\xi\in J_n$.
	Since
	$$-\log |K\mu(z)|=-\log |K\nu(z)|-\log |e^{iaz}|\gtrsim a|I_n|$$
	in the half-plane
	$$ \{\Im z> |I_n|/8\},$$
	we obtain
	$$-\log |K\eta_n(\xi)|\geqslant -\frac 1{2\pi}\int_0^{2\pi}\log \left|K\eta_n\left(\xi+\frac{|I_n|}{4}e^{i\phi}\right)\right|d\phi=$$$$
	-\frac 1{2\pi}\int_0^{2\pi}\log \left|K\mu\left(\xi+\frac{|I_n|}{4}e^{i\phi}\right)-K\mu_n\left(\xi+\frac{|I_n|}{4}e^{i\phi}\right)\right|d\phi\gtrsim
	$$$$
	\min\left(a|I_n|,-\log\frac{|\mu|(I_n)}{|I_n|}\right)
	$$
	for any $\xi\in J_n$. On the other hand,
	$$|K\mu(\xi)|=|K\eta_n(\xi)+K\mu_n(\xi)|\leqslant |K\eta_n(\xi)| +|\mu|(I_n)|$$
	and
	$$-\log |K\mu(\xi)|\gtrsim \min(|I_n|,-\log\frac{|\mu|(I_n)}{|I_n|},-\log|\mu_n|(I_n))\gtrsim
	$$$$\min(|I_n|,-\log|\mu|(I_n))$$
	(recall that $|I_n|>1$).
	
	\ms\no Now \eqref{bleq} implies that $\log |K\mu|$ is not Poisson-summable on the line $\{\Im z=1\}$.
	But any Cauchy integral of a non-zero measure must have a Poisson-summable logarithm on any horizontal line in $\C_+$, unless
	it is identically zero in $\C_+$, see for instance \cite{KoosisL9}.
	Similarly,
	it is zero in $\C_-$. If $K\mu$ is zero in both half-planes,  $\mu\equiv 0$.
	
\end{proof}

\ms\no Now assume that the complement of $\supp\mu$ is long. Then the complement can be taken as $\{I_n\}$ in \eqref{bleq}. We obtain

\begin{theorem}[Beurling's Gap Theorem \cite{Beurling-Stanford}]
	If $\mu$ is a finite measure supported on a set with long gaps and
	the Fourier transform of $\mu$ vanishes on an interval, then $\mu\equiv 0$.
	
\end{theorem}

\ms\no If instead of having porous support $\mu$ decays too fast at infinity, one can arrive at the same conclusion:

\begin{theorem}[Levinson, \cite{LevinsonL9}]\label{Lev}
	Let $\mu$ be a finite measure on $\R$ whose Fourier transform
	vanishes on an interval. Denote
	$$M(x)=|\mu|((x,\infty)).$$
	If $\log M$ is not Poisson-summable on $\R_+$ then $\mu\equiv 0$.
\end{theorem}

\begin{proof} Suppose that $\log M$ is not Poisson-summable on $\R_+$.
	Without loss of generality, $M(0)=1$. Let $0=a_0<a_1<a_2<...$
	be the points such that $M(a_n)=2^{-n}$ and denote by $I_n=(a_n,a_{n+1}]$
	the corresponding partition of $\R_+$. If
	$$\sum \frac{n|I_n|}{1+\dist^2(I_n,0)}<\infty$$
	then $\log M$ is Poisson-summable and we have a contradiction.
	
	\ms\no If the last sum is infinite,
	but the sum in \eqref{bleq} is finite, i.e. the partition $I_n$ is short, then any long
	super-partition of $I_n$ will satisfy \eqref{bleq}. If the last sum is infinite and $I_n$ is long,
	then \eqref{bleq} is satisfied.
\end{proof}

\ms\no Levinson's result above was later improved by Beurling \cite{Beurling-Stanford} who showed that an interval can be
replaced with a set of positive Lebesgue measure.

\ms\no Recall the following definition given in previous lectures. If $\mu$ is a finite positive measure on $\R$ we define
$$ \GG^p_\mu=\sup\{\ a\ |\  \exists \ f\in L^p(\mu),\int f(x)e^{2\pi i\l x}d\mu(x)=0,  \forall \ \l\in[0,a]\ \}.
$$
For $p=2$, $\GG^p_\mu$ is equal to $T_\mu$, the exponential type of $\mu$, the infimum of $a$, such that the family of exponentials
with frequencies from $[0,a]$ is complete in $L^2(\mu)$.

\ms\no Our next corollary combines results by Krein  (part I, $p=2$) and
by Levinson and McKean  (part II, $p=2$).

\begin{theorem}[Krein \cite{Krein1L9}, Levinson-McKean \cite{DML9}]\label{KLM9}
	Let $\mu$  be a finite measure on $\R$, $\mu=w(x) dx$ where $w(x)\geqslant 0$.
	Then
	
	I) If $\log w$ is Poisson-summable then for any $1\leqslant p\leqslant \infty$, $\GG^p_\mu=\infty$.
	
	II) If $\log w$ is monotone and Poisson-unsummable on a half-axis $(-\infty,x)$ or $(x,\infty)$ for some $x\in\R$ then for any $1< p\leqslant \infty$, $\GG^p_\mu=0$.
	
\end{theorem}

\begin{proof}
	If $\log w$ is Poisson-summable, denote by $W(z)$ the outer function in $\C_+$ satisfying $|W|=w$ on $\R$.
	Then for any $a>0$ the measure $e^{iax}W(x)dx$ annihilates all exponentials with frequencies from $[0,a)$.
	
	\ms\no It is left to show that if $\log w$ is  Poisson-unsummable and monotone on a half-axis then $\GG^p_\mu=0$ for all $p>1$.
	Without loss of generality, the half-axis is $\R_+$. Since for any $f\in L^p(\mu), p>1$,
	$\log (fw)$ is unsummable as well, we will simply assume that the Fourier transform of $\mu$ itself vanishes
	on an interval and arrive at a contradiction.
	
	\ms\no Choose real points $a_0=0<a_1<...<a_n<...$ in the following way.
	Put $a_0=0$. After $a_n, n\geqslant 0$ is chosen, choose $a_{n+1}$ to
	be the  number such that $a_{n+1}-a_n=-\log\mu((a_n,a_{n+1}))$. Note that such a number always exists
	except in the trivial case when the support of $\mu$ is bounded, see exercises.
	
	\ms\no Notice that if $\{I_n\}$ is long we are done by theorem \ref{BerLev}. It is left to show that since
	$\log w$ is Poisson-unsummable and $w$ is monotone, $I_n$ cannot be short. We leave this part to the reader
	as an elementary exercise.
	
\end{proof}

\ms\no Several of the statements above are also implied by the following theorem of de Branges:

\begin{theorem}[de Branges, theorem 63 \cite{dBr}]\label{Branges}
	Let $K(x)$ be a continuous function on $\R$ such that $K(x)\geqslant 1$, $\log K$ is uniformly continuous
	and Poisson-unsummable.
	Then there is no nonzero finite measure $\mu$ on $\R$ such that
	\begin{equation}\int^{\infty}_{-\infty}Kd|\mu|<\infty
	\label{messum}\end{equation}
	and $\hat\mu$ vanishes on an interval.
	
\end{theorem}

\begin{proof} Without loss of generality $K\geqslant 2$ and $K$ is Poisson-unsummable on $\R_+$.
	Choose points $a_0,a_1,...$ on $\R_+$ in the following way. Put $a_0=0$.
	After $a_{n-1}$ is chosen, choose $a_n$ to be the smallest point greater
	than $a_{n-1}$ such that
	$$\log K(a_n)\not\in \left(\frac{\log K(a_{n-1})}2, 2\log K(a_{n-1})\right).$$
	Note that such $a_n$ always exists because $K$ is unbounded on any ray $[x,\infty)$.
	Denote by $L$ the step function, minorating $\log K$ defined as
	$$L(x)=L_n=\min_{I_n} \log K$$ on each
	$I_n=(a_{n-1},a_n]$. Notice that by the choice of $\{I_n\}$, $\log L\asymp \log K$.
	In particular, $\log L$ is Poisson-unsummable. By \eqref{messum}, $\mu(I_n)\lesssim   1/L_n$.
	Also, because of uniform continuity of $\log K$,  $\log L_n\lesssim |I_n|$.
	Hence the sum in \eqref{bleq} is minorated by
	$$\sum\frac {|I_n|\log L_n}{1+\dist^2(I_n,0)}\gtrsim \int \log L(x)\frac {dx}{1+x^2}=\infty.$$
\end{proof}

\ms\no Theorem \ref{BerLev} has the following partial inverse.

\begin{proposition}
	Let $\mu=w(x)dx$ be an absolutely continuous finite measure with $w> 0$ and $\log |w|$ absolutely continuous.
	Suppose that the sequence of intervals $I_n$ satisfying \eqref{bleq} does not exist. Then $\GG^\infty_\mu=\infty$.
\end{proposition}

\begin{proof} Similarly to the last proof, it is not difficult to show that $\log |w|$ is Poisson-summable. After that for any $C>0$
	consider the measure $u\mu$ with
	$$u=e^{iCx} F/w,$$
	where $F$ is the outer function in the upper half-plane satisfying $|F|=w$.
\end{proof}


\lsection{Exercises}

\ms\no 1) Without using any of the theorems of the last two lectures, prove that if $\mu$ has bounded support (from below or from above) and its Fourier transform
vanishes on an interval, then $\mu\equiv 0$. In a sense, Beurling's and Levinson's theorems above generalize this statement.

\ms\no 2) Show that if $\mu$ has a spectral gap (a gap in the support of its Fourier transform) then it annihilates polynomials. Give an example of a measure that annihilates polynomials
but does not have a spectral gap.

\ms\no 3) Show that the second statement of theorem \ref{KLM9} needs the restriction that $\log w$ is not monotone.

\ms\no 4) Show that the condition that $\log K$ is uniformly continuous cannot be dropped from the statement
of theorem \ref{Branges}.

\ms\no 5) Finish the proof of theorem \ref{KLM9}.

\newpage

\lecture{10}{The Type Problem}

\ms\no In this lecture we finish the discussion of the Type Problem, one of the classical completeness problems
stated in the first lecture. Recall that we denote by $\EE_a=\EE_{[0,a]}$ the family of exponential functions whose frequencies belong to the interval
 from 0 to $a$.
 If $\mu$ is a finite positive measure on $\R$ and $a>0$ is a given number we would like to know if $\EE_a$ is complete in $L^2(\mu)$, i.e. if finite linear combinations of its functions are dense in $L^2(\mu)$.

\ms\no  This problem appears in the work of Kolmogorov, Krein and Wiener in 1930-40's. As discussed in our previous lectures, it is related
 to other classical problems of analysis in several adjacent fields. Our goal in this lecture is to formulate a solution to the Type Problem
 found in \cite{TypeL10}.

\lsection{Bakan's theorem for exponentials}

\ms\no  Just like $L^p$ problems on completeness of polynomials, the Type problem can be reformulated in terms of uniform Bernstein's approximation, using a proper version of Bakan's theorem discussed in lecture 7. Since such a reformulation will be important to us in the future, we recall
 the setup of Bernstein's approximation and Bakan's result.

 \ms\no We say that
a  function $W\geqslant 1$ on $\R$ is a \textit{weight} if $W$ is lower semi-continuous and
$W\to\infty$ as $|x|\to\infty$. The weights are allowed to take infinite values.

\ms\no The semi-norm $||f||_W$ is defined
on the set of all continuous $f$ such that $f/W\to 0$ at $\pm\infty$ as

\begin{equation}||f||_W=\sup_\R\frac{|f|}W.\label{norm10}\end{equation}

\ms\no The following statement is an analog of Bakan's theorem \cite{BakanL10} discussed in lecture 7 for the case of polynomials.

\begin{theorem}
Let  $0< p <\infty $ and $a>0$ be  constants,  and let $\mu$ be a positive finite measure on $\R$. The family of exponentials
$\EE_a$ is complete in  $L^p(\mu)$
if and only if $\mu$ can be represented as $\mu=W^{-p}\nu$ for some finite positive measure $\nu$ and a weight $W$ such that
$\EE_a$ is complete in $C_W$.
\end{theorem}

\ms\no The proof is similar to the polynomial version (see exercises).

\lsection{A solution to the Type Problem}

\ms\no A condition for the completeness of $\EE_a$ in $C_W$ ($L^p(\mu)$) found in \cite{TypeL10} relies on the notion of $d$-uniform sequences
introduced in lecture 8 to treat the Gap Problem. Let us recall the definition of these sequences.

\ms\no Let
$$...<a_{-2}<a_{-1}<a_0=0<a_1<a_2<...$$
 be a discrete sequence of real points. We say that the intervals $I_n=(a_n,a_{n+1}]$ form a short partition
of $\R$ if $|I_n|\to\infty$  as $ n\to \pm\infty$ and the sequence $\{I_n\}$ is short.

\ms\no Let $\L=\{\lan\}$ be a discrete sequence of real points. We say that $\L$ is
$d$-uniform if
there exists a short partition $\{I_n\}$ such that
\begin{equation} \Delta_n= d|I_n|+o(|I_n|)\ \ \   \text{for all} \ \ \ n\ \ \textrm{(density condition)}\label{density10}\end{equation}
as $n\to\pm\infty$ and
\begin{equation} \sum_n \frac{\Delta_n^2\log|I_n|-E_n}{1+\dist^2(0,I_n)}<\infty\ \ \textrm{(energy condition)}\label{energy10}\end{equation}
where $\Delta_n$ and $E_n$ are defined as
$$\Delta_n=\#(\L\cap I_n)\ \ \text{ and }\ \ E_n=E(\Lambda\cap I_n)=\sum_{\lambda_k,\lambda_l\in I_n,\ \lambda_k\neq\lambda_l}\log|\lambda_k-\lambda_l|.$$

\ms\no Via Bakan's theorem an equivalent version of the Type Problem can be formulated as follows:
Given a weight $W$ find
$$T_W=\inf\{a\ |\ \EE_a\textrm{ is complete in }C_W\}.$$
As usual $T_W$ is set to infinity if the set above is empty.

\ms\no Here is an answer to that question
(this is an equivalent version of the main result of \cite{TypeL10}).

\begin{theorem}\label{main1}
Let  $W$ be a weight on $\R$.
$$T_W=\sup\left\{d\ : \ \exists
\textrm{$d$-uniform $\L$, }
\sum\frac{\log W(\lan)}{1+\lan^2}<\infty\right\}.$$

\end{theorem}

\ms\no Returning to the $L^p$-approximation, recall that for the finite positive measure $\mu$ the quantities $\GG_\mu^p$
were defined as

\begin{equation} \GG^p_\mu=\sup\{\ a\ |\  \exists \ f\in L^p(\mu),\int f(x)e^{2\pi i\l x}d\mu(x)=0,  \forall \ \l\in[0,a]\ \}.
 \label{typep10}\end{equation}

\ms\no Recall that the Type of $\mu$ can be equivalently defined as $T_\mu=\GG^2_\mu$. Via Bakan's theorem, the last
result can be reformulated as follows (see exercises). If $\mu $ is a finite positive measure on $\R$, we call a weight $W$
$\mu$-weight if
$$\int Wd\mu<\infty.$$

\begin{theorem}\label{mainType}
Let $\mu$ be a finite positive measure on the line. Let $1<p\leqslant\infty$ and $a>0$ be  constants.

\ms\no Then $\GG^p_\mu\geqslant 2\pi a$
if and only if for any  $\mu$-weight $W$ and any $0<d<a$ there exists
a $d$-uniform sequence $\L=\{\lan\}\subset \supp \mu$ such that
\begin{equation}\sum\frac{\log W(\lan)}{1+\lan^2}<\infty.\label{ur4}\end{equation}

\end{theorem}

\ms\no Once again, the proof of theorem \ref{mainType} is too long to include in this text. Let us discuss some of its
corollaries and applications.

\lsection{Discrete case}\label{secDisc}

\ms\no The conditions of theorem \ref{mainType} are simplified for many specific classes of measures.
In particular, if the measure is discrete, or absolutely continuous with regular enough density, the weight $W$ may be eliminated from
the statement. Here we treat the discrete case that is important in spectral theory of differential operators and other adjacent areas.
Our results in this section may be viewed as extensions of the result by Koosis mentioned earlier in the course.

\ms\no The following statement gives a simplified formula for the type of a measure supported on a discrete sequence, excluding pathological
cases when the counting function of the sequence grows exponentially.

\begin{theorem}\label{mainDiscrete} Let $B=\{b_n\}$ be a discrete sequence of real points. Let  $$\mu=\sum w(n)\delta_{b_n}$$
be a finite positive measure supported on $B$.
Define
$$D = \sup\left\{\ d\ :\ \exists\textrm{ d-uniform }B'\subset B,\ \sum_{b_n\in B'}\frac {\log w(n)}{1+{n}^2}>-\infty\ \right\}.$$
Then for any $1<p\leqslant\infty$,
$$\GG^p_\mu\geqslant 2\pi D.$$
If the counting function of $B$ satisfies $\log (|n_B|+1)\in L^1_\Pi$ then
$$\GG^p_\mu=2\pi D.$$

\end{theorem}

\ms\no The condition $\log (|n_B|+1)\in L^1_\Pi$ in the second part of the statement is
sharp, see
exercises.

\ms\no In the case when the sequence is separated, the condition can be simplified even further.
Note that for $p=1$, $\GG^1_\mu=2\pi D_*(\L)$ for any separated sequence $\L$ and any measure $\mu,\ \supp\mu=\L$,
by the gap theorem from lecture 8. For $p>1$ we have

\begin{theorem}\label{mainSeparated} Let $\L=\{\lambda_n\}$ be a separated sequence and let
$$\mu=\sum w(n)\delta_{\lambda_n}$$ be a finite positive measure supported on $\L$.
Define $$D=\sup D_*(\L'),$$
where the supremum is taken over all subsequences $\L'\subset \L$ satisfying
\begin{equation}\sum_{\lambda_n\in \L'}\frac{\log w(n)}{1+n^2}>-\infty.\label{ur10}
\end{equation}
Then $$\GG^p_\mu=2\pi D$$
for all $1<p\leqslant \infty$.
\end{theorem}

\begin{proof}
Suppose that $\GG^p_\mu>2\pi D$ for some $D>0, p>1$.
Define the $\mu$-weight $W$ as $W(\lambda_n)=(\mu(\{\lambda_n\})(1+\lambda^2_n))^{-1}$.
Then by theorem \ref{mainDiscrete} there
exists a subsequence $\L'\subset \L$ such that $D_*(\L')>D$ and \eqref{ur10} is
satisfied.

\ms\no In the opposite direction the statement follows directly from theorem
\ref{mainDiscrete} and the remark at the end of lecture 8.
\end{proof}

\lsection{Exercises}

\ms\no 1) Prove Bakan's theorem for exponentials (first try to do it without looking at the polynomial proof).

\ms\no 2) Deduce theorem \ref{main1} from theorem \ref{mainType}.

\ms\no 3) Let $S\subset\R$ be a union of unit intervals. Define the measure $\mu$ as
$d\mu(x)=\charf_S(x)\,dx/(1+x^2)$. Show that $\GG^p_\mu$ is either 0 for all $p$ or $\infty$ for all $p$.

\ms\no 4) Using theorem \ref{mainType} or the results of \cite{BSL10} produce an example showing that the condition $\log (|n_B|+1)\in L^1_\Pi$ in the second part of the statement
of the theorem \ref{mainDiscrete} is sharp.

\bs

\newpage

\lecture{11}{The type alternative for Frostman measures}

\ms\no In Lectures 8 and 10 we obtained a solution to the Type Problem: for a finite positive
measure $\mu$ on $\R$, the exponential type $T_\mu$ was computed in terms of $d$-uniform
sequences contained in the support of $\mu$ and weights summable against $\mu$. The answer,
although complete, is not always easy to apply to a concrete measure. In this lecture,
following \cite{FrostA}, we discuss a result of a different nature: a broad regularity
condition on $\mu$, the Frostman condition, under which the answer becomes an alternative --
the type is either zero or infinite. In particular, measures of finite positive type are
necessarily highly irregular.

\lsection{The type alternative}

\ms\no Recall from Lecture 8 that for a finite positive measure $\mu$ on $\R$
$$T_\mu=\inf\{\ a>0\ |\ \EE_a\textrm{ is complete in }L^2(\mu)\ \},$$
where $\EE_a=\{e^{2\pi i\l t}|\ \l\in[0,a]\}$, and $T_\mu=\infty$ if no such $a$ exists. The
three model situations computed in Lectures 8--9 via the theorems of Krein and
Levinson-McKean already show the two extremes:
$$d\mu=\frac{dx}{1+x^2}\ \Rightarrow\ T_\mu=\infty,\qquad
d\mu=e^{-x^2}dx\ \Rightarrow\ T_\mu=0,$$
and any compactly supported $\mu$ has $T_\mu=0$ (exercise 1). Producing a measure with
$0<T_\mu<\infty$ is a much more delicate matter: the standard examples, like the normalized
counting measures of arithmetic progressions from Lecture 8, are discrete. It is natural to
ask which measures with, say, bounded density can have finite positive type. The answer is:
none.

\ms\no We say that a positive measure $\mu$ on $\R$ satisfies the \textit{Frostman condition}
(or is a Frostman measure) if there exist constants $C,\alpha>0$ such that
$$\mu(I)<C|I|^\alpha$$
for every interval $I\subset\R$, $|I|\leq 1$. Conditions of this type appear in potential theory,
where they go back to the classical thesis of O. Frostman \cite{Frost13}, and in the theory of
Hausdorff measures and dimensions. Note that the condition is only a mild regularity
requirement: every measure with bounded density satisfies it with $\alpha=1$, and every
absolutely continuous measure with density in $L^p(\R)$, $1<p\leq\infty$, satisfies it with
$\alpha=1-1/p$ (exercise 2). The masses of small intervals must decay uniformly at some fixed
positive rate, but the rate can be arbitrarily slow.

\ms\no The main result of \cite{FrostA} is the following.

\begin{theorem}[Type alternative, \cite{FrostA}]\label{Frostmain}
Let $\mu$ be a finite positive measure on $\R$ satisfying the Frostman condition. Then
$$T_\mu=0\qquad\textrm{ or }\qquad T_\mu=\infty.$$
\end{theorem}

\ms\no In other words, a Frostman measure cannot have finite positive exponential type. Note
that the statement is independent of the normalization of the frequencies: rescaling the
exponentials multiplies the type by a constant and does not affect the alternative.

\lsection{Sharpness}

\ms\no The regularity assumption of Theorem \ref{Frostmain} is essential: measures
of finite positive type exist, but their densities have to grow rapidly along parts of the
support. An instructive example, obtained by the methods of \cite{BS13}, is the absolutely
continuous measure with the density
$$w=\sum_{n\in\Z}\frac{e^{|n|}}{1+n^2}\,\charf_{[n,n+e^{-|n|})},$$
which is finite (the mass of the $n$-th block is $1/(1+n^2)$) and has type $T_\mu=1$ in our
normalization: the tall thin blocks of the density imitate the point masses of the counting
measure of $\Z$, whose type was computed in Lecture 8. The density of this measure is
unbounded and grows exponentially along the blocks, and the measure of the interval
$[n,n+e^{-|n|})$ is about $n^{-2}$, so no Frostman condition can hold. Exercise 4 makes the
connection with Theorem 17 of Lecture 10 explicit.

\ms\no Thus one can say that finite positive type is a resonance phenomenon: it requires a
precise balance between the support, which must be close to a $d$-uniform sequence as in
Lectures 8 and 10, and the weight, which must load that sequence in an essentially discrete
way. Any uniform regularity of the measure destroys the balance and pushes the type to one of
the extremes.

\lsection{Elements of the proof}

\ms\no The proof of Theorem \ref{Frostmain} in \cite{FrostA} combines the tools developed in
this course. By Bakan's theorem for exponentials (Lecture 10), $\EE_a$ is complete in
$L^2(\mu)$ if and only if it is complete in a weighted space $C_W$ for a suitable $\mu$-weight
$W$; by the results of Lecture 10, the type is then computed through $d$-uniform sequences
$\Lambda\subset\supp\mu$ carrying summable logarithmic weights. The heart of the matter is to
show that if a Frostman measure admits one such sequence of density $d>0$, then it admits such
sequences of every larger density: the Frostman condition allows one to redistribute the
estimates and upgrade the density, so the supremum in the type formula is either $0$ or
$\infty$. The redistribution argument uses the energy condition of Lecture 8 and the stability
of $d$-uniform sequences under small perturbations. For the details we refer to \cite{FrostA}.

\lsection{Exercises}

\ms\no 1) Show that any finite measure with compact support has type $0$. (Hint: if $f\perp\EE_a$
in $L^2(\mu)$, then the Fourier transform of $f\mu$ is an entire function vanishing on an
interval; compare with exercise 1 of Lecture 9.)

\ms\no 2) Show that if $d\mu=w\,dx$ with $w\in L^p(\R)$, $1<p\leq\infty$, then $\mu$ is a
Frostman measure with $\alpha=1-1/p$. (Hint: H\"older's inequality.)

\ms\no 3) Verify the two model computations at the beginning of the lecture using the theorems
of Krein and Levinson-McKean from Lecture 8. Observe that in both cases the conclusion agrees
with the type alternative.

\ms\no 4) Consider the example $w$ of the Sharpness section. Show directly that $\mu=w\,dx$ is
not a Frostman measure, and use Theorem 17 of Lecture 10 (with $B$ close to $\Z$ and
$w(n)\asymp1/(1+n^2)$) to explain why its type is positive and finite. What happens to the
type if $e^{|n|}$ in the definition of $w$ is replaced by $e^{\sqrt{|n|}}$?

\ms\no 5) Let $S\subset\R$ be a union of unit intervals and $d\mu(x)=\charf_S(x)dx/(1+x^2)$, as
in exercise 3 of Lecture 10. Deduce the $0$-$\infty$ alternative for this measure directly
from Theorem \ref{Frostmain}.

\ms\no 6) Show that the Frostman condition in Theorem \ref{Frostmain} cannot be replaced by
the condition $\mu(I)\leq C/\log(1/|I|)$, $|I|<1/2$, by inspecting the example of the
Sharpness section. How slowly can the modulus of continuity of $\mu$ decay for the
counterexamples to disappear? (This is an open-ended question; see \cite{FrostA}.)

\newpage

\lecture{12}{The Beurling-Malliavin theory for Toeplitz kernels}

\ms\no The Beurling-Malliavin theorem of Lecture 2 computes the radius of completeness of a
family of complex exponentials. In Lectures 8--11 we became acquainted with the language in
which modern proofs and generalizations of such results are naturally stated: kernels of
Toeplitz operators. In this lecture, based on \cite{MIF2L12} (see also the companion paper
\cite{MIFL12}), the two threads join. We present the next step of the Beurling-Malliavin
theory: the computation of the analogous ``radius'' for the family of Toeplitz kernels
$N[J\bar S^a]$, where $J$ and $S$ are meromorphic inner functions and the argument of $S$ has
a power behavior on the real line. The classical theorem corresponds to $S(z)=e^{iz}$. The
results of this lecture will be used in the next one, where the Beurling-Malliavin theorem is
placed inside the Toeplitz order, and their spectral theory meaning will reappear in
Lectures 15 and 16.

\lsection{Toeplitz kernels and the completeness radius}

\ms\no Let $U\in L^\infty(\R)$. The Toeplitz operator $T_U$ with the symbol $U$ is the
operator
$$T_U:H^2\to H^2,\qquad T_Uf=P_+(Uf),$$
where $P_+$ is the orthogonal projection from $L^2(\R)$ onto the Hardy space $H^2=H^2(\C_+)$
of Lecture 3. We write
$$N[U]=\ker T_U$$
for the kernel of $T_U$ and call it the Toeplitz kernel of the symbol $U$. Along with the
$H^2$-kernel one considers the kernels in the Smirnov class $\NN^+$ and in the Hardy classes
$H^p$:
$$\begin{gathered}N^{+}[U]=\{F\in \NN^+\cap L^1_{loc}(\R)\ |\ \bar U\bar F\in \NN^+\},\\
N^p[U]=N^+[U]\cap L^p(\R),\quad 0<p\leqslant\infty.\end{gathered}$$
If $\Theta$ is an inner function then $N[\bar\Theta]=K_\Theta$, the model space of Lecture 3.

\ms\no Let now $\L\subset\R$ be a discrete sequence and let $J_\L$ be a meromorphic inner
function whose level set $\{J_\L=1\}$ is exactly $\L$; such a function always exists and can
be produced by an explicit construction (exercise 6). A short argument, which we leave as
exercise 2, connects uniqueness sets of model spaces with Toeplitz kernels: $\L$ is a
uniqueness set of $K_\Theta$, for a meromorphic inner $\Theta$, if and only if
$$N[J_\L\bar\Theta]=0.$$
Recall from Lecture 2 that the completeness problem for the exponential family
$E_\L=\{e^{2\pi i\l x}\}_{\l\in\L}$ in $L^2(0,a)$ is the uniqueness problem for the model
spaces of the exponential inner functions: $E_\L$ is complete in $L^2(0,a)$ if and only if
$\L$ is a uniqueness set of $K_{S^{2\pi a}}$, where here and below
$$S(z)=e^{iz}.$$
Combining the two equivalences we obtain
\begin{equation}\label{BMTradius}
R(\L)=\frac1{2\pi}\,c(J_\L,S),\qquad c(J,S):=\inf\{\ a\ |\ N[J\bar S^{a}]\neq0\ \},
\end{equation}
and the Beurling-Malliavin theorem of Lecture 2 becomes the formula
$$c(J_\L,S)=2\pi D^*(\L).$$
(The paper \cite{MIF2L12} uses the exponentials $e^{i\l x}$ and the spaces $L^2(-a,a)$, which
results in slightly different constants; throughout this lecture we keep the normalization of
Lecture 2 in the statements about $R(\L)$ and quote the theorems about kernels in their
original, normalization-free form.)

\ms\no The problem addressed in \cite{MIF2L12} is the computation of $c(J,S)$ for general
meromorphic inner functions $J$ and $S$. The answer is obtained in the case where the
argument of $S$ has a power law type behavior,
$$(\arg S)'(x)\asymp|x|^\kappa,\qquad x\to\pm\infty,$$
with $\kappa\geqslant0$; here $\arg S$ stands for a branch of the argument continuous in the
closed upper half-plane. Following \cite{MIF2L12} we call this the super-exponential case, to
underline the relation to the classical case $S(x)=e^{iax}$, which corresponds to $\kappa=0$.

\lsection{The basic criterion}

\ms\no Recall that a function $h$ is Poisson-summable, $h\in L^1_\Pi$, if it is summable with
respect to the Poisson measure $d\Pi=dx/(1+x^2)$. For real $h\in L^1_\Pi$ the Hilbert
transform $\ti h$ is defined a.e.\ on $\R$ as the angular boundary value of the imaginary
part of the Schwarz integral
$$\SS h(z)=\frac1{\pi i}\int_\R\left[\frac1{t-z}-\frac t{1+t^2}\right]h(t)\,dt,
\qquad z\in\C_+.$$
The relevance of the Hilbert transform to Toeplitz kernels is explained by the following
criterion, which is the workhorse of the whole theory.

\begin{theorem}[\cite{MIFL12}]\label{BMTcrit}
Let $\gamma:\R\to\R$ be a smooth function. Then $N^+[e^{i\gamma}]\neq0$ if and only if
\begin{equation}\label{BMTbasic}
\gamma=-\alpha+\ti h
\end{equation}
for some smooth increasing function $\alpha$ and some $h\in L^1_\Pi$. Similarly,
$N^p[e^{i\gamma}]\neq0$ if and only if $\gamma$ admits a representation \eqref{BMTbasic} in
which $\alpha$ is the argument of some inner function and $h\in L^1_\Pi$ satisfies
$e^{-h}\in L^{p/2}(\R)$.
\end{theorem}

\ms\no In other words, non-triviality of the kernel means that the argument of the symbol is
an increasing function, up to a correction by a harmonic conjugate of a Poisson-summable
function. All the results below can be viewed as answers, for concrete families of symbols,
to the question of when such a representation is possible.

\lsection{BM intervals and almost decreasing functions}

\ms\no Let $\gamma$ be a continuous function $\R\to\R$ such that
$$\lim_{x\to-\infty}\gamma(x)=+\infty,\qquad\lim_{x\to+\infty}\gamma(x)=-\infty.$$
The family ${\rm BM}(\gamma)$ of BM intervals of $\gamma$ is defined as the collection of the
components of the open set
$$\left\{x\ |\ \gamma(x)\neq\max_{[x,+\infty)}\gamma\right\}.$$
On each BM interval the function lies below its running maximum from the right; outside the
BM intervals it decreases. Thus the family ${\rm BM}(\gamma)$ measures the deviation of $\gamma$
from a decreasing function. For $l\in{\rm BM}(\gamma)$ we write $|l|$ for its length and
$d(l)=\dist(0,l)$.

\ms\no Let $\kappa\geqslant0$. We say that $\gamma$ is $(\kappa)$-almost decreasing if
$\gamma(\mp\infty)=\pm\infty$ and
$$\sum_{l\in{\rm BM}(\gamma),\ d(l)\geqslant1}d(l)^{\kappa-2}|l|^2<\infty.$$
In the classical case $\kappa=0$ the standard terminology is that the family ${\rm BM}(\gamma)$
is short if $\gamma$ is almost decreasing and long otherwise -- the same dichotomy of short
and long families of intervals as in the definition of the densities of Lecture 2.

\lsection{The main theorems}

\ms\no In what follows, $f(x)\gtrsim g(x)$ means that $f(x)\geqslant cg(x)$ for some $c>0$
and all $x$ with $|x|\gg1$.

\begin{theorem}[\cite{MIF2L12}]\label{BMTa}
Let $\kappa\geqslant0$, and let $U=e^{i\gamma}$ and $S=e^{i\sigma}$ be smooth unimodular
functions on $\R$ such that
$$\gamma'(x)\gtrsim-|x|^\kappa,\qquad\sigma'(x)\gtrsim|x|^\kappa,\qquad x\to\infty.$$

\ms\no (i) If $\gamma$ is not $(\kappa)$-almost decreasing, then $N^+[US^\epsilon]=0$ for all
$\epsilon>0$.

\ms\no (ii) If $\gamma$ is $(\kappa)$-almost decreasing, then $N^p[U\bar S^\epsilon]\neq0$
for all $\epsilon>0$ and all $p<\frac13$.
\end{theorem}

\ms\no Given $U$ and $S$ as in the theorem, consider the family of symbols
$$U\bar S^a=e^{i\gamma_a},\qquad\gamma_a=\gamma-a\sigma,\qquad a\in\R.$$
If $\gamma_a$ is $(\kappa)$-almost decreasing and $a_1>a$, then $\gamma_{a_1}$ is
$(\kappa)$-almost decreasing as well (exercise 1), so one can define the transition
parameter
$$c(U,S;\kappa)=\inf\{\ a\ |\ \gamma_a\ \textrm{is}\ (\kappa)\textrm{-almost decreasing}\ \}
\in(-\infty,+\infty].$$

\begin{corollary}[\cite{MIF2L12}]\label{BMTcor}
Let $U=e^{i\gamma}$ and $S=e^{i\sigma}$ satisfy
$$\gamma'(x)\gtrsim-|x|^\kappa,\qquad\sigma'(x)\asymp|x|^\kappa,\qquad x\to\infty,$$
and let $c=c(U,S;\kappa)$. Then for all $p<1/3$,
$$N^p[U\bar S^a]=0\quad(a<c),\qquad N^p[U\bar S^a]\neq0\quad(a>c).$$
\end{corollary}

\ms\no In the special case where the symbol involves an inner function, the dichotomy extends
to all values of $p$, including the Hilbert space case $p=2$.

\begin{theorem}[\cite{MIF2L12}]\label{BMTb}
Let $J$ be a meromorphic inner function, and suppose that a unimodular function $S$
satisfies $(\arg S)'(x)\asymp|x|^\kappa$ as $x\to\infty$. Denote $c=c(J,S;\kappa)$. Then for
all $p\leqslant\infty$,
$$N^p[J\bar S^a]=0\quad(a<c),\qquad N^p[J\bar S^a]\neq0\quad(a>c).$$
\end{theorem}

\ms\no Theorem \ref{BMTb} says precisely that the analytic threshold $c(J,S)$ of
\eqref{BMTradius} coincides with the combinatorial transition parameter $c(J,S;\kappa)$,
which is computable from the argument of the symbol. For $\kappa=0$, $J=J_\L$ and
$S(z)=e^{iz}$, the almost decreasing condition for $\gamma_a=\arg J_\L-ax$ translates into
the interval densities of Lecture 2, and the classical Beurling-Malliavin theorem is
recovered; see exercise 3 for the model computation. In the kernels $N^p[J\bar S^a]$ with
$a>c$ one can say more: they are infinite-dimensional.

\lsection{The upper density estimate}

\ms\no As a sample of the technique we prove the first, easier, step of the theory: an
analogue of the classical estimate of the completeness radius by the upper density
(Lecture 2). The proofs of this section are taken from \cite{MIF2L12} and rest on
Kolmogorov's weak type theorem for the Hilbert transform: for real $h\in L^1_\Pi$,
$$\Pi\{|\ti h|>A\}=o\left(\frac1A\right),\qquad A\to\infty.$$

\begin{lemma}\label{BMTkol}
Let $\kappa\geqslant0$ and $h\in L^1_\Pi$. If $\ti h'(x)\lesssim x^\kappa$ as $x\to+\infty$,
then $\ti h(x)=o(x^{\kappa+1})$ as $x\to+\infty$.
\end{lemma}

\begin{proof}
Suppose $x_*\gg1$ and $\ti h(x_*)\geqslant cx_*^{\kappa+1}$ for some $c>0$. The bound
$\ti h'\lesssim x^\kappa$ implies that, going to the left from $x_*$, the function $\ti h$
cannot decay too fast: for $1\ll x\leqslant x_*$,
$$\ti h(x)\geqslant\ti h(x_*)-a\int_x^{x_*}t^\kappa\,dt\geqslant
cx_*^{\kappa+1}-a(x_*^{\kappa+1}-x^{\kappa+1}).$$
Hence $\ti h\gtrsim x_*^{\kappa+1}\geqslant x_*$ on an interval $(x_{**},x_*)$ of length
comparable to $x_*$. Since $\Pi(x_{**},x_*)\asymp1/x_*$, this contradicts Kolmogorov's
theorem. If $\ti h(x_*)\leqslant-cx_*^{\kappa+1}$, the same argument applied to the right of
$x_*$ gives $\ti h\lesssim-x_*$ on an interval of length comparable to $x_*$, again a
contradiction.
\end{proof}

\begin{proposition}[\cite{MIF2L12}]\label{BMTupper}
Let $U=e^{i\gamma}$ and $S=e^{i\sigma}$ satisfy the conditions of Theorem \ref{BMTa}. If
$N^+[US^\epsilon]\neq0$ for some $\epsilon>0$, then $\gamma(\mp\infty)=\pm\infty$.
\end{proposition}

\begin{proof}
By Theorem \ref{BMTcrit},
$$\gamma+\epsilon\sigma+\alpha=\ti h,\qquad\alpha'\geqslant0,\quad h\in L^1_\Pi.$$
Then $\ti h'(x)\gtrsim\gamma'(x)\gtrsim-|x|^\kappa$, and Lemma \ref{BMTkol}, applied to
$-h$, gives $\ti h(x)=o(|x|^{\kappa+1})$. It follows that
$$\frac{\gamma(x)}x\lesssim\frac{\gamma(x)}x+\frac{\alpha(x)}x=
-\epsilon\frac{\sigma(x)}x+o\left(|x|^\kappa\right)\lesssim-|x|^\kappa,$$
which implies $\gamma(\mp\infty)=\pm\infty$.
\end{proof}

\ms\no The second, much deeper, part of the triviality statement in Theorem \ref{BMTa} (i)
-- divergence of the series over BM intervals implies triviality of the kernel -- corresponds
to what is known as Beurling's lemma in the classical theory. Several proofs of that lemma
are known: through estimates of harmonic measure (Koosis \cite{KooLogL12}), through the
Bellman function (Nazarov), and through PDE techniques (Kargaev); the proof in
\cite{MIF2L12} is different and rests on a weak type estimate for the two-dimensional
Hilbert transform.

\lsection{The multiplier theorem}

\ms\no The non-triviality half of the theory passes through a generalization of the famous
Beurling-Malliavin multiplier theorem. For a unimodular $S$ and $0<p\leqslant\infty$, say
that a real function $w\in L^1_\Pi$ belongs to the multiplier class $\MM_p(S)$ if the outer
function $W=e^{w+i\ti w}$ satisfies the following condition: for every $\epsilon>0$ there
exists $G\in N^+[\bar S^\epsilon]$ such that $WG\in H^p$. In other words, $W$ can be
multiplied into $H^p$ at an arbitrarily small (compared to $S$) cost.

\begin{theorem}[\cite{MIF2L12}]\label{BMTmult}
Suppose $(\arg S)'\gtrsim|x|^\kappa$, and let $w_0\in L^1_\Pi$ be a real function such that
the function $|x|^{-\frac{2+\kappa}2}w_0(x)$ coincides, in a neighborhood of infinity, with
a function of finite Dirichlet integral. Then $w_0\in\MM_p(S)$ for all $p<1$.
\end{theorem}

\ms\no For $\kappa=0$ this is a form of the classical multiplier theorem; see \cite{MNHL12}
for its remarkable history and seven classical proofs, and \cite{HJL12}, \cite{KooLogL12},
\cite{KooLecL12} for detailed expositions of the classical theory. The proof in
\cite{MIF2L12} obtains the multiplier from the solution of an extremal problem in the
Dirichlet space, following an idea from de Branges' book \cite{dBL12}. The overall structure
of the proof of Theorems \ref{BMTa} and \ref{BMTb} is then as follows: the upper and
effective density estimates give part (i); a ``little multiplier theorem'' produces, in the
almost decreasing case, a non-trivial kernel in the Smirnov class; the multiplier theorem,
combined with a theorem of the type of the Beurling-Malliavin First Theorem (the logarithm
of an outer function in a Smirnov kernel has a finite weighted Dirichlet integral),
upgrades non-triviality from $\NN^+$ to $H^p$; finally, approximation of unimodular
functions by inner ones moves the statement to the symbols $J\bar S^a$ of Theorem
\ref{BMTb}.

\lsection{The sub-exponential case}

\ms\no The statements of Theorems \ref{BMTa} and \ref{BMTb} fail for $\kappa<0$: exercise 5
provides unimodular functions with $\gamma(+\infty)\neq-\infty$ whose kernel is
nevertheless non-trivial. The theory extends to $\kappa\in(-1,0]$ in a different fashion:
one replaces the Smirnov and Hardy kernels by their weighted versions,
$$\begin{gathered}N^+_\kappa[U]=N^+[U]\cap\NN^+_\kappa,\\
\NN^+_\kappa=\left\{F\in\NN^+\ \Big|\ \log|F|\in
L^1\left(\frac{dx}{1+|x|^{2+\kappa}}\right)\right\},\end{gathered}$$
and in this scale the dichotomy is again governed by the family ${\rm BM}(\gamma)$: if
${\rm BM}(\gamma)$ is long then $N^+_\kappa[US^\epsilon]=0$ for all $\epsilon>0$, and if
${\rm BM}(\gamma)$ is short then $N^p_\kappa[U\bar S^\epsilon]\neq0$ for all $\epsilon>0$ and
$p<1/3$, see \cite{MIF2L12}. The sub-exponential theory has applications to Volterra
operators and higher order differential operators.

\lsection{Connections with spectral theory}

\ms\no As explained in \cite{MIFL12}, the computation of the radius \eqref{BMTradius} for
general $S$ has direct consequences for the spectral theory of second order differential
operators. Roughly speaking, the case of Sturm-Liouville operators with eigenvalues
$$\l_n\asymp n^\nu$$
belongs to the theory with the parameter
$$\kappa=\frac2\nu-1\geqslant0.$$
The classical case $\kappa=0$, i.e. $S(x)=e^{iax}$, corresponds to regular operators, while
$\kappa>0$ corresponds to singular ones. In addition to completeness problems for families
of solutions of Sturm-Liouville equations, the generalized Beurling-Malliavin theory applies
to problems on spectral determinacy, to the Weyl-Titchmarsh Fourier transforms associated
with such operators, and to the corresponding de Branges spaces of entire functions. The
reader will recognize this dictionary in the canonical systems of Lectures 15 and 16, where
Weyl inner functions with power-type behavior of the phase appear in concrete inverse
spectral problems.

\lsection{Exercises}

\ms\no 1) Show that if $\gamma_a=\gamma-a\sigma$ is $(\kappa)$-almost decreasing,
$\sigma'\gtrsim|x|^\kappa$, and $a_1>a$, then $\gamma_{a_1}$ is $(\kappa)$-almost
decreasing. Conclude that the transition parameter $c(U,S;\kappa)$ is well defined.

\ms\no 2) Let $\L\subset\R$ be discrete, let $J_\L$ be a meromorphic inner function with
$\{J_\L=1\}=\L$, and let $\Theta$ be a meromorphic inner function. Show that there exists a
non-zero function in $K_\Theta$ vanishing on $\L$ if and only if $N[J_\L\bar\Theta]\neq0$.
(Hint: if $F\in K_\Theta$ vanishes on $\L$, divide by $1-J_\L$; in the opposite direction
multiply by it.)

\ms\no 3) Let $\L=\delta\Z$ for some $\delta>0$, so that $\arg J_\L(x)=2\pi x/\delta+O(1)$
for a suitable choice of $J_\L$. Show that $\gamma_a=\arg J_\L-ax$ is $(0)$-almost
decreasing for $a>2\pi/\delta$ and is not for $a<2\pi/\delta$. Conclude from Theorem
\ref{BMTb} and \eqref{BMTradius} that $R(\delta\Z)=1/\delta$, in agreement with Lecture 2.

\ms\no 4) Show that for $\kappa=0$ a function $\gamma$ with $\gamma(\mp\infty)=\pm\infty$ is
almost decreasing if and only if the family ${\rm BM}(\gamma)$ is short in the sense of
Lecture 2. Give an example of $\gamma$ for which ${\rm BM}(\gamma)$ consists of a long sequence
of intervals $l_n$ with $|l_n|/d(l_n)\to0$.

\ms\no 5) Let
$$\sigma(x)=2\,{\rm sign}(x)\,|x|^{1/4},\qquad\gamma(x)=2(1+\sqrt2)\,1_{\R_-}(x)\,|x|^{1/4}.$$
Verify that $\gamma'(x)\gtrsim-|x|^{-3/4}$ and $\sigma'(x)\gtrsim|x|^{-3/4}$, while
$\gamma(+\infty)=0\neq-\infty$. Show that nevertheless $N^\infty[US]\neq0$ for
$U=e^{i\gamma}$, $S=e^{i\sigma}$, by checking that
$$US=\bar f/f,\qquad f(z)=\exp\left\{-(1+i)z^{1/4}\right\}\in H^\infty(\C_+).$$
Thus Theorem \ref{BMTa} has no direct extension to $\kappa<0$.

\ms\no 6) (Krein's shift.) Let $A=\{a_n\}$ and $B=\{b_n\}$ be two intertwining discrete real
sequences, $\dots a_n<b_n<a_{n+1}\dots$, and let $E=\bigcup(a_n,b_n)$. Define $\Theta$ in
$\C_+$ by
$$\frac1{\pi i}\log\frac{\Theta+1}{\Theta-1}=\SS u+ic,\qquad u=1_E-\frac12,\quad c\in\R,$$
where $\SS u$ is the Schwarz integral defined above. Show that $\Theta$ is a meromorphic
inner function with $\{\Theta=1\}=A$ and $\{\Theta=-1\}=B$. Deduce that for any increasing
continuous function $\sigma:\R\to\R$ there is a meromorphic inner function $\Theta=e^{i\theta}$
with $\|\theta-\sigma\|_\infty\leqslant\pi$, and that any discrete $\L\subset\R$ is the level
set $\{J=1\}$ of some meromorphic inner function $J$.

\newpage

\lecture{13}{The Toeplitz order}

\ms\no In Lecture 5 we repeatedly used the heuristic principle that in problems on defining
sets and mixed spectral data ``the known factor of the inner function has to be bigger than the
unknown factor.'' In this lecture, following \cite{TO}, we make the word ``bigger'' precise: we
introduce a partial order on the set of inner functions induced by the action of Toeplitz
operators. The order provides a common language for many of the problems discussed in this
course, including the Beurling-Malliavin theorem of Lecture 2, and organizes the function
theoretic component of the Toeplitz approach to the Uncertainty Principle.

\lsection{Toeplitz operators and their kernels}

\ms\no Let $U\in L^\infty(\R)$. The Toeplitz operator $T_U$ with the symbol $U$ is the operator
$$T_U:H^2\to H^2,\qquad T_Uf=P_+(Uf),$$
where $P_+$ is the orthogonal projection from $L^2(\R)$ onto the Hardy space $H^2=H^2(\C_+)$
defined in Lecture 3. We denote by
$$N[U]=\ker T_U$$
the kernel of $T_U$, and call it the Toeplitz kernel of the symbol $U$. Along with the
$H^2$-kernel one considers the kernels $N^p[U]$ in other Hardy classes $H^p$ and the kernel in
the Smirnov class,
$$N^{+}[U]=\{f\in \NN^+\cap L^1_{loc}(\R)\ |\ \bar U\bar f\in \NN^+\},$$
which will appear below in the statement of the Beurling-Malliavin theorem.

\ms\no Two examples connect these notions with the objects of the previous lectures. First, if
$\Theta$ is inner then
$$N[\bar\Theta]=K_\Theta,$$
the model space of Lecture 3: indeed, $f\in\ker T_{\bar\Theta}$ if and only if $f\perp\Theta
H^2$. Second, the symbols that appear most often in the applications have the form
$U=\bar I J$ with inner $I,J$: as we saw in Lectures 5 and 6, non-triviality of $N[\bar IJ]$
expresses the existence of a function in $K^2_{IJ}$-type spaces vanishing on prescribed sets,
and hence encodes uniqueness and defining set problems.

\ms\no A useful elementary observation, which the reader will prove in the exercises, is the
monotonicity of the kernel under multiplication of the symbol by $\bar S^c$, $c>0$:
$$N[U]\subset N[\bar S^cU].$$
Multiplying the symbol by the conjugate of an inner function can only increase the kernel.

\lsection{The dominance set and the Toeplitz order}

\ms\no The naive idea is to declare $I$ ``larger'' than $J$ when $T_{\bar IJ}$ has a
non-trivial kernel. This relation, however, is neither reflexive nor transitive, and cannot
serve as a definition of an order directly (see Section 3.2 of \cite{TO} for counterexamples).
The correct definition compares inner functions through the collections of functions they
dominate.

\ms\no For an inner function $\theta$ define its \textit{dominance set}
$$\DD(\theta)=\{I\ \textrm{inner}\ |\ N[\bar\theta I]\neq 0\}.$$

\ms\no We say that $I$ and $J$ are \textit{Toeplitz equivalent}, and write $I\stackrel{\tau}\sim J$,
if $\DD(I)=\DD(J)$. We write
$$I\stackrel{T}\leq J\qquad\textrm{ if }\qquad \DD(I)\subset\DD(J),$$
and call this relation the \textit{Toeplitz order}. Since inclusion of sets is reflexive and
transitive, $\stackrel{T}\leq$ is a genuine partial order on the equivalence classes of
$\stackrel{\tau}\sim$.

\begin{proposition}\label{TOprop}
(i) The relation $\stackrel{T}\leq$ is reflexive and transitive, and $\stackrel{\tau}\sim$ is
an equivalence relation.

\ms\no (ii) If $J$ divides $I$, i.e. $I=JG$ for an inner $G$, then $J\stackrel{T}\leq I$.

\ms\no (iii) For the exponential family, $S^a\stackrel{T}\leq S^b$ if and only if $a\leq b$.
Moreover $S^a\in\DD(S^b)$ if and only if $a<b$.
\end{proposition}

\ms\no The proofs are elementary and are left as exercises. Part (ii) says that the Toeplitz
order extends the classical division order on inner functions; the converse implication fails,
and it is precisely the gap between the two orders that makes the Toeplitz order useful: it is
insensitive to ``small'' discrepancies between inner functions that are invisible to spectral
problems. Part (iii) shows that the exponentials $S^a=e^{iaz}$ form a totally ordered chain,
which plays the role of a scale inside the partially ordered set: measuring a given inner
function against the chain $\{S^a\}$ produces a numerical characteristic.

\lsection{The Beurling-Malliavin theorem in the language of the order}

\ms\no For an inner function $I$ define its \textit{radius}
$$r(I)=\sup\{a\geq0\ |\ N[\bar IS^a]\neq0\}=\sup\{a\ |\ S^a\in\DD(I)\}.$$
By the Beurling-Malliavin theory for Toeplitz kernels developed in \cite{MIF212} and
presented in the previous lecture, the radius
can be computed from the argument of $I$ in terms of the densities of Lecture 2; for a
meromorphic inner function it is determined by the distribution of its spectrum
$\sigma(I)$. The following statement from \cite{TO} shows that comparison with the exponential
chain captures exactly this quantity, i.e. that the classical Beurling-Malliavin theorem of
Lecture 2 is a statement about the Toeplitz order.

\begin{theorem}[\cite{TO}]\label{TOBM}
Let $I$ be a meromorphic inner function and let $b>0$. Then
$$I\stackrel{T}\leq S^b\ \Rightarrow\ r(I)\leq b,\qquad\textrm{ and }\qquad r(I)<b\
\Rightarrow\ I\stackrel{T}< S^b.$$
\end{theorem}

\ms\no One half of the theorem is elementary: if $I\stackrel{T}\leq S^b$ and $S^a\in\DD(I)$,
then $S^a\in\DD(S^b)$, so $a<b$ by Proposition \ref{TOprop} (iii), and $r(I)\leq b$ follows
(exercise 4). The other half is a reformulation of the second Beurling-Malliavin theorem,
enhanced by the multiplier theory: knowing only that the single kernel $N[\bar IS^b]$ is
non-trivial for $b>r(I)$, one has to produce a non-trivial kernel $N[\bar\theta I]$ for
\textit{every} $\theta$ dominating $S^b$. We refer to \cite{TO} for the proof.

\lsection{Applications and further directions}

\ms\no The defining set problems of Lectures 5 and 6 become monotonicity statements for the
order. In the Hochstadt-Liberman setting, $\Theta=\Psi\Phi$, the heuristic ``$\Psi$ has to be
bigger than $\Phi$'' now reads $\Phi\stackrel{T}\leq\Psi$; the results of \cite{MIF12} cited in
Lecture 5 give sufficient conditions for the data $[\Psi,\sigma(\Theta)]$ to determine $\Theta$
precisely in this form. Similarly, the Type Problem of Lectures 8 and 10 and the sampling
problems for Paley-Wiener and model spaces admit natural restatements as problems of comparing
concrete inner functions in the Toeplitz order; see Sections 8.3 and 8.4 of \cite{TO}.

\ms\no The structure of the order itself raises many questions: to describe the equivalence
classes of $\stackrel{\tau}\sim$, to characterize the maximal chains, to understand which
families of inner functions admit suprema and infima. Some of these questions are answered in
\cite{TO}, many remain open, and the subject continues the line of research started by the
papers \cite{MIF12, MIF212} discussed earlier in this course.

\lsection{Exercises}

\ms\no 1) Prove part (i) of Proposition \ref{TOprop}. Show also that $I\stackrel{\tau}\sim cI$
for any unimodular constant $c$.

\ms\no 2) Show that $N[U]\subset N[\bar S^cU]$ for any $c>0$ and any symbol $U$. Deduce that
$\DD(S^a)\subset\DD(S^b)$ for $a\leq b$. Then verify that $N[\bar S^bS^a]$ is non-trivial if
and only if $a<b$, by computing this kernel explicitly. (Hint: for $a<b$ the kernel is the
model space $K_{S^{b-a}}$; for $a\geq b$ the symbol is analytic.)

\ms\no 3) Prove part (ii) of Proposition \ref{TOprop}: if $I=JG$ with $G$ inner, and
$N[\bar J\Phi]\neq0$, show that the same function belongs to $N[\bar I\Phi]$ by multiplying
the relation $\bar J\Phi f\in\bar H^2$ by $\bar G$. Where exactly is the innerness of $G$ used?

\ms\no 4) (Coburn's lemma) Let $I\not\equiv J$ be inner. Show that $N[\bar IJ]$ and
$N[\bar JI]$ cannot both be non-trivial. (Hint: if $f$ belongs to the first kernel and $g$ to
the second, consider the product $fg$ and show that it belongs to $H^1\cap\bar H^1=\{0\}$.)
Deduce the easy half of Theorem \ref{TOBM}.

\ms\no 5) Let $I$ be a meromorphic inner function whose spectrum $\sigma(I)$ is a separated
sequence $\Lambda\subset\R$. Using the results of Lectures 2 and 8, express $r(I)$ through the
interior Beurling-Malliavin density $D_*(\Lambda)$, up to the normalization of the exponent.

\newpage

\lecture{14}{The two-spectra theorem with uncertainty}

\ms\no In Lectures 5 and 6 we studied mixed spectral problems, where a part of the potential
and a part of the spectral data are known and the question is whether these data determine the
operator uniquely. In this lecture, based on \cite{TSU}, we return to the oldest result of that
family, Borg's two-spectra theorem, and show how the solutions of the Gap and Type Problems
obtained in Lectures 8--10 allow one to treat a new kind of question: what happens if the
eigenvalues themselves are known only approximately? We will see that the Uncertainty Principle
provides an exact formula for the ``price'' of such an approximation.

\lsection{Borg's two-spectra theorem}

\ms\no Consider the Schr\"odinger equation
$$-\ddot u+qu=z^2u$$
on $[0,\pi]$ with a real potential $q\in L^2(0,\pi)$. Together with a selfadjoint boundary
condition at each endpoint the equation defines a selfadjoint Schr\"odinger operator with
discrete spectrum. We will consider two such operators: $L_{DD}$, corresponding to Dirichlet
conditions at both endpoints, and $L_{ND}$, corresponding to the Neumann condition at $0$ and
the Dirichlet condition at $\pi$. Denote by $\Sigma_{DD}$ and $\Sigma_{ND}$ their spectra. The
two spectra interlace, and for $q\in L^2$ they obey the classical asymptotics: after the square
root transform (see Lecture 6) the symmetrized sequences
$$\sigma_{DD}=\{\lan\}_{n\in\Z},\qquad \sigma_{ND}=\{\eta_n\}_{n\in\Z}$$
are close to $\pi\Z$ and $\pi(\Z+1/2)$ respectively, with $\ell^2$-corrections of order $1/n$.

\ms\no The classical theorem of Borg \cite{Borg11} (complemented by Levinson \cite{Lev11} and
Marchenko \cite{M11}) says that the two spectra determine the operator:
if $q,\ti q\in L^2(0,\pi)$ produce the same pair $(\Sigma_{DD},\Sigma_{ND})$, then $q=\ti q$.
In the language of Lecture 5, the two spectra form a defining set of ``full size,'' and we saw
in Lecture 6 how such statements translate into completeness problems for exponential systems.

\lsection{Spectral data with uncertainty}

\ms\no Suppose now that the eigenvalues are known only up to small errors: instead of the exact
sequences we are given a sequence of positive numbers $\{\e_n\}$ and we know each $\lan$ and
$\eta_n$ only up to an error of size $\e_n^2$. Physically this is a natural setting: spectral
data always come from measurements. The question is how much of the potential must be known in
addition to the approximate spectra to determine the operator, i.e. what part of the
information carried by the two spectra is destroyed by the uncertainty.

\ms\no Following \cite{TSU} we say that the given spectral information lacks a part of size
$a$, $0\leq a\leq 1$, if together with the potential on $(0,a\pi+\e)$, for any $\e>0$, this
information determines the operator uniquely, while the knowledge of the potential on
$(0,a\pi-\e)$ is insufficient.

\ms\no Recall the interior Beurling-Malliavin density $D_*$ defined in Lecture 2. For a
subsequence of integers $\Phi\subset\N$ the quantity $D_*(\Phi)$ measures the maximal density
of a regular part of $\Phi$. The following theorem, the central result of \cite{TSU}, expresses
the size of the lost information through $D_*$. We denote $\log^-x=\max(0,-\log x)$.

\begin{theorem}[\cite{TSU}]\label{TSUmain}
Let $\{\e_n\}_{n\in\N}$ be a sequence of positive numbers and let $a\in[0,1)$. The following
are equivalent:

\ms\no (i) any two Schr\"odinger operators $L,\ti L$ with $L^2$-potentials $q,\ti q$, whose
spectral sequences satisfy
$$|\lan-\ti\l_n|<\e^2_n,\qquad |\eta_n-\ti\eta_n|<\e^2_{n+1}\qquad\textrm{ for all }n,$$
and whose potentials coincide on $(0,d\pi)$ for some $d>a$, must coincide identically;

\ms\no (ii) every subsequence $\Phi\subset\N$ satisfying
$$\sum_{n\in\Phi}\frac{\log^-\e_n}{1+n^2}<\infty$$
has $D_*(\Phi)\leq a$.
\end{theorem}

\ms\no The meaning of the condition in (ii) is the following. The sum is finite exactly when
the errors $\e_n$, $n\in\Phi$, are not too exponentially small, i.e. when along $\Phi$ the
uncertainty is genuinely present. The theorem says that the size of the lost information equals
the maximal interior density of a subsequence along which the uncertainty survives. If, on the
contrary, $\e_n$ decays fast enough along every dense subsequence, the approximate spectra are
as good as the exact ones. Note the appearance of the Poisson-type weight $1/(1+n^2)$, our
constant companion since Lecture 2.

\ms\no The same result can be stated in the geometric form promised at the beginning of the
lecture. Suppose that we are given a sequence of intervals $\{I_n\}$ and the only spectral
information is that the eigenvalues lie in the corresponding intervals,
$$\sigma_{DD}\cup\sigma_{ND}\subset U=\bigcup I_n$$
(one eigenvalue per interval, in the proper enumeration). Define the size of uncertainty of the
family $I=\{I_n\}$ as
\begin{equation}\label{sizeU}
U(I)=\pi\,\sup\Big\{D_*(\Phi)\ \Big|\ \Phi\subset\N,\ \ \sum_{n\in\Phi}\frac{\log^-|I_n|}{1+n^2}<\infty\Big\}.
\end{equation}

\begin{corollary}[Borg's theorem with uncertainty]\label{TSUcor}
The interval data $\{I_n\}$, together with the potential on $(0,\e)$ for every $\e>0$,
determine the operator uniquely if and only if every $\Phi\subset\N$ with
$$\sum_{n\in\Phi}\frac{\log^-|I_n|}{1+n^2}<\infty$$
satisfies $D_*(\Phi)=0$. In general, the part of the potential that has to be added to the
interval data is $(0,U(I))$, with $U(I)$ given by \eqref{sizeU}.
\end{corollary}

\ms\no The proofs in \cite{TSU} rely on the machinery developed in the previous lectures: the
two spectra and the uncertainty intervals are translated, through the Weyl inner functions of
Lecture 4 and the defining sets of Lectures 5--6, into a gap-type problem for measures
concentrated near the symmetrized spectra, to which the Gap Theorem of Lecture 8 and its
stability under exponentially small perturbations (see the remarks at the end of Lecture 8, and
\cite{BS11}) can be applied.

\lsection{Indeterminate pairs in the three-interval problem}

\ms\no The second part of \cite{TSU} concerns the mixed problem where the potential is known
near both endpoints, on $[0,a\pi]\cup[(1-b)\pi,\pi]$, together with one spectrum
$\sigma_{DD}$. When such data fail to determine the operator, the operators sharing the data
come in pairs, which we call indeterminate pairs. A model example is provided by symmetry: if
$q$ is even with respect to the center of the interval, $q(\pi-x)=q(x)$, then $q$ and suitable
non-symmetric perturbations may produce the same partial data (compare with the
Hochstadt-Liberman theorem of Lecture 5, which shows that for $a+b\geq 1$ no such pairs exist).

\ms\no It turns out that all indeterminate pairs admit a complete description in terms of their
spectral measures. Roughly, two operators with the common spectrum $\Lambda=\sigma_{DD}$ and
common potential on $[0,a\pi]\cup[(1-b)\pi,\pi]$ form an indeterminate pair if and only if the
differences of their spectral measures are given by a pair of even functions
$$f\in L^2(\eta)\ominus PW_{2a},\qquad g\in L^2(\eta)\ominus PW_{2b}$$
(where $\eta$ is a discrete measure on $\Lambda$ determined by the spectrum) whose values at
the points of $\Lambda$ obey coupled $\ell^2$-asymptotics; see \cite{TSU} for the exact
statement. One consequence is that the pairs with ``almost symmetric'' potentials form only a
small submanifold of the set of all indeterminate pairs, although every indeterminate pair has
to be close to symmetric in terms of the asymptotics of its spectral measure.

\ms\no This description is yet another instance of the phenomenon we first met in Lecture 6:
non-uniqueness in an inverse spectral problem is measured by the orthogonal complements of
Paley-Wiener spaces in $L^2$ of spectral measures, i.e. by incompleteness of families of
exponentials.

\lsection{Exercises}

\ms\no 1) Show that if $\e_n=e^{-n^2}$ then condition (ii) of Theorem \ref{TSUmain} holds with
$a=0$, and hence the approximate spectra together with an arbitrarily small germ of the
potential at $0$ determine the operator. (Hint: if the sum converges for $\Phi$, then
$\sum_{n\in\Phi}1/n<\infty$; show that then $D_*(\Phi)=0$.)

\ms\no 2) Let $\e_n=e^{-n^\alpha}$, $\alpha>0$. Show that
$$U=\begin{cases}\pi&\textrm{ if }\alpha<1,\\ 0&\textrm{ if }\alpha\geq1,\end{cases}$$
i.e. the uncertainty is maximal for all $\alpha<1$ and vanishes for all $\alpha\geq 1$. Compare
this dichotomy with the result of Borichev and Sodin on exponentially small perturbations
discussed in Lecture 8.

\ms\no 3) Show that if $|I_n|\geq c>0$ along a subsequence $\Phi$ with $D_*(\Phi)=d$, then
$U(I)\geq\pi d$. In particular, if all the intervals have lengths bounded from below, the
interval data alone determine nothing: the whole potential is needed.

\ms\no 4) Using the remarks at the end of Lecture 8 on $d$-uniform sequences, explain why the
threshold in exercise 2 occurs exactly at the exponential rate of decay, $\e_n\asymp e^{-cn}$.

\ms\no 5) Deduce from Theorem \ref{TSUmain} the following version of the classical Borg
theorem: the exact spectra $\sigma_{DD}$, $\sigma_{ND}$ together with the germ of $q$ at $0$
determine $L$. (The classical theorem does not need the germ; the loss is the price of the
generality of the method. See \cite{TSU} for a discussion.)

\newpage

\lecture{15}{Canonical systems and the inverse spectral problem}

\ms\no The natural habitat for the spectral problems of this course is the Krein-de Branges
theory of canonical Hamiltonian systems, which translates spectral problems for a broad
class of differential operators into problems of complex and harmonic analysis. In this lecture and the next one, based on the
papers \cite{E1, E2L14}, we make this step. We introduce canonical systems and their de Branges
spaces and discuss the inverse spectral problem: the reconstruction of the system from its
spectral measure. The classical solutions of Gelfand-Levitan, Krein and Marchenko will
reappear here in the form of invertibility of truncated Toeplitz operators on Paley-Wiener
spaces.

\lsection{Canonical systems and de Branges spaces}

\ms\no A (half-line, regular at $0$) canonical Hamiltonian system is the differential equation
\begin{equation}\label{cansys}
\Omega\dot X=z\HH X\quad\textrm{ on }[0,\infty),\qquad
\Omega=\begin{pmatrix}0&1\\-1&0\end{pmatrix},
\end{equation}
where $z\in\C$ is the spectral parameter and the Hamiltonian $\HH(t)$ is a $2\times2$ real
symmetric matrix function with $\HH\geq0$ and entries in $L^1_{loc}[0,\infty)$. Schr\"odinger
operators, Dirac systems, Jacobi matrices and Krein strings can all be rewritten as canonical
systems (exercise 1 treats the Schr\"odinger case), which makes \eqref{cansys} the universal
model of one-dimensional selfadjoint spectral theory.

\ms\no Denote by $M(t,z)=\begin{pmatrix}A&B\\C&D\end{pmatrix}(t,z)$ the transfer matrix
(matrizant) of the system, i.e. the matrix solution of \eqref{cansys} with $M(0,z)=I$. For
each fixed $t$ the entries are entire functions of $z$, and the function
$$E(t,z)=A(t,z)-iC(t,z)$$
belongs to the Hermite-Biehler class: $|E(t,z)|>|E(t,\bar z)|$ for $z\in\C_+$. With any
Hermite-Biehler function $E$ one associates the de Branges space
$$\BB(E)=\left\{G\ \textrm{entire}\ \Big|\ \frac GE,\ \frac{G^\#}E\in H^2\right\},\qquad
G^\#(z)=\overline{G(\bar z)},$$
with the norm $||G||^2_{\BB(E)}=\int_\R|G(t)|^2\frac{dt}{|E(t)|^2}$. The spaces
$\BB_t=\BB(E(t,\cdot))$ form an increasing chain of Hilbert spaces of entire functions,
$\BB_s\subset\BB_t$ for $s<t$, with isometric inclusions at regular points. Each $\BB_t$ has
reproducing kernels
$$K_t(\l,z)=\frac1\pi\,\frac{A(t,z)C(t,\bar\l)-C(t,z)A(t,\bar\l)}{\bar\l-z},$$
generalizing the reproducing kernels of the spaces $K_\Theta$ and $PW_a$ computed in Lecture 4.
In the free case $\HH\equiv I$ the matrizant is the rotation by the angle $tz$, $E=e^{-itz}$,
and $\BB_t$ is the Paley-Wiener space of exponential type $t$ (exercise 2); a general chain
$\{\BB_t\}$ should be viewed as a curved version of the Paley-Wiener scale.

\ms\no The spectral measure of the system is a positive Poisson-finite measure $\mu$ on $\R$
such that the spaces $\BB_t$ sit isometrically inside $L^2(\mu)$:
$$||G||_{\BB_t}=||G||_{L^2(\mu)}\qquad\textrm{ for all }G\in\BB_t,\ t>0.$$
It can be obtained from the Herglotz representation of the Weyl function of the system,
constructed from $M(t,z)$ in the same way as in Lecture 4 the $m$-function was constructed
from solutions of the Schr\"odinger equation. The fundamental theorem of the area, the
Krein-de Branges theorem, says that the correspondence is essentially bijective:

\begin{theorem}[Krein, de Branges, see \cite{dB14, Rom14}]\label{KdB}
Every Poisson-finite positive measure $\mu$ on $\R$ (with, possibly, a point mass at
infinity) is the spectral measure of a canonical system \eqref{cansys}, and the Hamiltonian
$\HH$ is determined by $\mu$ uniquely up to a natural reparametrization of the variable $t$.
\end{theorem}

\ms\no The inverse spectral problem (ISP) is the problem of actually finding $\HH$ from
$\mu$. In this generality no explicit solution is known, or likely to exist; the goal of
\cite{E1} is to develop formulas and algorithms solving the ISP for broad classes of spectral
measures, extending the classical procedures of Gelfand-Levitan, Krein and Marchenko.

\lsection{Truncated Toeplitz operators and sampling measures}

\ms\no Let $\mu$ be a positive measure on $\R$, and let $PW_a$ be the Paley-Wiener space of
Lecture 2. Consider the sesquilinear form $\int f\bar g\,d\mu$ on $PW_a$. When the form is
bounded, it defines a bounded selfadjoint \textit{truncated Toeplitz operator}
$L_{\mu,a}:PW_a\to PW_a$,
$$<L_{\mu,a}f,g>_{PW_a}=\int f\bar g\,d\mu,\qquad f,g\in PW_a.$$
We say that $\mu$ is a \textit{sampling measure} for $PW_a$ if
$$c\,||f||^2_{L^2(\R)}\leq\int|f|^2d\mu\leq C\,||f||^2_{L^2(\R)},\qquad f\in PW_a,$$
for some $C\geq c>0$; equivalently, if $L_{\mu,a}$ is bounded and invertible. Sampling
measures are a classical object of Harmonic Analysis: the two-sided estimate says that the
data $\{f\ \textrm{on}\ \supp\mu\}$ is a stable substitute for $f$ itself, and the theorem of
Duffin and Schaeffer cited in Lecture 8 was one of the first results in this direction. In
\cite{E1} it is shown that the measures which are sampling for \textit{all} Paley-Wiener
spaces simultaneously admit an elementary description, in the spirit of uniform boundedness of
$\mu$ on unit intervals from above together with a uniform lower bound on sufficiently long
intervals; we refer to \cite{E1} for the precise statement, and to \cite{OS14} for the
description of sampling measures of a fixed $PW_a$.

\ms\no The relevance of these notions to the ISP is explained by the following principle
proved in \cite{E1}: the canonical systems whose spectral measures are sampling for all
Paley-Wiener spaces are exactly the systems whose chains of de Branges spaces are comparable
with the Paley-Wiener scale (with $\det\HH\neq0$ a.e. after normalization). For such systems
the ISP can be solved by inverting truncated Toeplitz operators. The scheme is as follows.
Since $\BB_t$ sits isometrically in $L^2(\mu)$ and is comparable with $PW_t$, the reproducing
kernel $K_t(\l,\cdot)$ of $\BB_t$ can be found by solving the truncated Toeplitz equation
$$L_{\mu,t}\,K_t(\l,\cdot)=K^{PW_t}(\l,\cdot)$$
in $PW_t$, where $K^{PW_t}$ is the sinc-kernel of the Paley-Wiener space. Once the kernels are
known, the Hamiltonian is recovered by differentiation in $t$; for instance, for the diagonal
entry one has the formula
\begin{equation}\label{hrec}
h_{11}(t)=\pi\,\frac{d}{dt}\,k_t(0),\qquad k_t(z):=K_t(0,z),
\end{equation}
and a companion formula recovers $h_{22}$ from the dual (conjugate) kernels, in analogy with
the dual reproducing kernels of Lecture 4; the off-diagonal entry is expressed through the
generalized Hilbert transform of the kernels, which we will discuss in the next lecture. In
the free case $\mu=$ Lebesgue measure, $L_{\mu,t}$ is the identity, $k_t$ is the sinc kernel
and \eqref{hrec} returns $h_{11}\equiv1$ (exercise 4).

\lsection{The Gelfand-Levitan class}

\ms\no The classical inverse spectral theory corresponds to measures close to the Lebesgue
measure. Following \cite{E1}, we say that $\mu$ belongs to the Gelfand-Levitan class if
$$\hat\mu=\delta_0+\phi,\qquad \phi\in L^1_{loc}(\R),$$
i.e. if the (distributional) Fourier transform of $\mu$ differs from that of the Lebesgue
measure by a locally summable function. For such measures the truncated Toeplitz equation
becomes an integral equation with the kernel $\phi(x-y)$ on $(-t,t)$, and the scheme of the
previous section turns into the classical Gelfand-Levitan procedure \cite{GL14}; the
invertibility of $L_{\mu,t}$, which in the classical theory is a theorem, here becomes the
defining property of the class of measures for which the algorithm runs. One of the results
of \cite{E1} is that the Gelfand-Levitan class is a proper subclass of the sampling measures
discussed above, so the truncated Toeplitz approach strictly extends the classical one. In the
periodic case, when $\mu$ is a periodic measure, the equations reduce to finite-dimensional
linear algebra and connect the ISP with the theory of orthogonal polynomials on the unit
circle; see Section 6 of \cite{E1}.

\lsection{Exercises}

\ms\no 1) Let $q$ be real, $q\in L^1_{loc}$, and let $u_0$ be a real non-vanishing solution of
$-\ddot u+qu=0$ on an interval. Show that the substitution reducing the Schr\"odinger
equation $-\ddot u+qu=zu$ to a canonical system \eqref{cansys} can be performed with the
Hamiltonian
$$\HH(t)=\begin{pmatrix}u_0^2(t)&0\\0&u_0^{-2}(t)\end{pmatrix}$$
(the spectral parameter enters as $z$, not $z^2$; the change of variable absorbs the second
derivative). What goes wrong at the zeros of $u_0$?

\ms\no 2) Compute the matrizant for $\HH\equiv I$:
$$M(t,z)=\begin{pmatrix}\cos tz&-\sin tz\\ \sin tz&\cos tz\end{pmatrix}.$$
Show that $E(t,z)=e^{-itz}$, that $\BB_t$ coincides with a Paley-Wiener space, and identify
the spectral measure up to normalization. (Hint: Parseval's theorem, Lecture 2.)

\ms\no 3) Verify that $\mu$ is a sampling measure for $PW_a$ if and only if $L_{\mu,a}$ is
bounded and invertible, and that $<L_{\mu,a}f,f>=\int|f|^2d\mu$ for $f\in PW_a$.

\ms\no 4) Let $\mu$ be the Lebesgue measure. Show that $L_{\mu,t}$ is the identity operator on
$PW_t$ (Plancherel), that $k_t(z)=\frac{\sin tz}{\pi z}$, and check that formula \eqref{hrec}
gives $h_{11}\equiv1$, in agreement with $\HH\equiv I$ and exercise 2.

\ms\no 5) Show that a finite measure cannot be sampling for any $PW_a$, and that a sampling
measure for all $PW_a$, $a>0$, cannot have gaps of unboundedly growing length in its support.
Compare with the long intervals of Lectures 2 and 8.

\ms\no 6) In the Gelfand-Levitan class, write the truncated Toeplitz equation as the integral
equation
$$f(x)+\int_{-t}^t\phi(x-y)f(y)\,dy=g(x),\qquad x\in(-t,t),$$
on the Fourier side. Where is the invertibility of $L_{\mu,t}$ used?

\newpage

\lecture{16}{The inverse spectral problem: examples}

\ms\no In the last lecture we described the general scheme of \cite{E1b}: for canonical
systems whose spectral measures are sampling for the Paley-Wiener scale, the inverse spectral
problem reduces to the inversion of truncated Toeplitz operators. In this final lecture,
following \cite{E2}, we run the scheme on concrete families of measures. The examples
illustrate the connections of the inverse problem with classical objects of analysis: the
Hilbert transform, the Riemann-Hilbert problem, and Bessel functions, which make here their
second appearance in the course after the Bessel inner functions of Lectures 4 and 5.

\lsection{Homogeneous spectral measures}

\ms\no For a Borel measure $\mu$ on $\R$ and $t>0$ define the rescaled measure
$\mu_t(B)=\frac1t\mu(tB)$. We call $\mu$ \textit{homogeneous} if $\mu_t=\mu$ for all $t>0$. It
is an elementary exercise to check that the homogeneous Poisson-finite measures are exactly
the absolutely continuous measures whose density is constant on each of the half-axes:
$$d\mu=(c_1+c_2\,\sign x)\,dx,\qquad c_1>|c_2|\geq0.$$
For $c_2=0$ we recover the Lebesgue measure, i.e. the free system of the previous lecture; for
$c_2\neq0$ the measure is a genuinely non-symmetric perturbation of the free one, and the
corresponding Hamiltonian is a first non-trivial test of the method.

\ms\no Homogeneity of the measure forces self-similarity of the whole chain of de Branges
spaces. If $k_t(z)=K_t(0,z)$ denotes the reproducing kernel of $\BB_t$ at the origin, then for
a homogeneous spectral measure
$$k_t(z)=t\,k_1(tz),$$
and, by the recovery formula \eqref{hrec} of Lecture 15, the entry $h_{11}$ is constant in
$t$. A similar argument applies to the other entries and shows that the Hamiltonian of a
homogeneous measure is constant up to logarithmic terms. The precise result of \cite{E2} is
that, in the natural normalization $\det\HH=1$, the Hamiltonian has the form
$$\HH(t)=\begin{pmatrix}C_1&C-C_2\log t\\ C-C_2\log t&\dfrac{1-(C-C_2\log t)^2}{C_1}\end{pmatrix},$$
where the constants $C_1>0$, $C_2$, $C$ are determined explicitly by $c_1,c_2$, and $C_2=0$
if and only if $c_2=0$. The appearance of $\log t$ in the off-diagonal entries is a genuinely
new effect compared to the free case: a constant non-even density produces an unbounded,
slowly turning Hamiltonian.

\lsection{The Hilbert transform and the Riemann-Hilbert problem}

\ms\no Two classical tools enter the computation. The first is the \textit{generalized Hilbert
transform}: for an entire $f\in L^2(\mu)$ put
$$(H^\mu f)(z)=\frac1\pi\int\left[\frac{f(s)-f(z)}{s-z}+\frac{s\,f(z)}{1+s^2}\right]d\mu(s),$$
a regularized version of the singular integral which reduces to the usual Hilbert transform
for $\mu=$ Lebesgue measure (compare with the regularized Cauchy and Schwarz integrals of
Lecture 3). The transform $H^\mu$ applied to the reproducing kernels produces the conjugate
kernels and recovers the off-diagonal entries of the Hamiltonian, completing the recovery
formulas of Lecture 15.

\ms\no The second tool appears when one actually inverts the truncated Toeplitz operator
$L_{\mu,t}$ for $d\mu=(c_1+c_2\sign x)dx$. On the Fourier side the equation
$L_{\mu,t}f=g$ becomes a Riemann-Hilbert problem on the interval $[-t,t]$: to find a
piecewise analytic function $\Phi$ with the jump relation
$$\Phi^+=F\,\Phi^-+f\qquad\textrm{ on }[-t,t],$$
with an explicit coefficient $F$ built from $c_1,c_2$. The problem is solved by the classical
formula of Muskhelishvili \cite{Musk15}: the homogeneous problem has the solution
$$X_t(z)=\exp\left\{\frac1{2\pi i}\int_{-t}^t\frac{\log F(s)}{s-z}\,ds\right\},$$
and the full solution is obtained from $X_t$ by one more Cauchy integral. The constants
$C_1,C_2,C$ in the formula for the Hamiltonian above are computed from $X_t$. We refer to
\cite{E2} for the computation and to \cite{Deift15} for a modern survey of Riemann-Hilbert
methods.

\lsection{Quasi-homogeneous measures and Bessel functions}

\ms\no The natural generalization of homogeneity is homogeneity of a non-zero order: we call
$\mu$ \textit{quasi-homogeneous of order} $\nu$ if
$$\mu_t=t^{-(1+2\nu)}\mu\qquad\textrm{ for all }t>0.$$
The Poisson-finite quasi-homogeneous measures are the measures with the densities
$$w(x)=\begin{cases}\ c_+\,x^{1+2\nu},&x>0,\\ \ c_-\,|x|^{1+2\nu},&x<0,\end{cases}
\qquad -1<\nu<0$$
(exercise 2 explains the range of $\nu$). For such measures the chain of de Branges spaces is
again self-similar, with the scaling of the kernels
$$k_t(z)=t^{2+2\nu}\,k_1(tz),$$
and the Hamiltonian can be computed in closed form. The most instructive case is the even one,
$c_+=c_-$: here the answer is the diagonal power Hamiltonian.

\ms\no Consider the canonical system with
$$\HH(t)=\begin{pmatrix}t^{m}&0\\0&t^{-m}\end{pmatrix},\qquad m>-1.$$
It is shown in \cite{E2} that the entries of its matrizant are expressed through Bessel
functions: with $\nu=\frac{1+m}2$ and $F_\nu$ the entire function defined by
$J_\nu(\l)=\l^\nu F_\nu(\l)$,
$$A(t,z)=q_\nu\,F_{\nu-1}(zt),\qquad C(t,z)=g_\nu\,t^{2\nu}\,z\,F_\nu(zt)$$
for explicit constants $q_\nu,g_\nu$, the reproducing kernels are
$$k_t(z)=\frac{g_\nu^2F_{\nu-1}(0)}\pi\,t^{2\nu}F_\nu(zt),$$
and the spectral measure is the quasi-homogeneous measure
$$d\mu=\const|x|^{m}dx.$$
Thus the diagonal power Hamiltonians are the canonical-system incarnation of the Bessel
family: the reader should compare this with the Bessel potentials $q(t)=(\nu^2-\frac14)/t^2$
and the Bessel inner functions $\Theta_\nu$ of Lecture 4, to which these systems are related
by the transformation of exercise 1 of Lecture 15. For $m=0$ everything collapses to the free
system, $F_{-1/2}$ and $F_{1/2}$ being the cosine and the sine.

\ms\no These examples show the inverse spectral problem from a concrete, computational side:
each explicitly solvable spectral measure is a point of contact between the abstract
Krein-de Branges theory and classical special functions. Further examples, including
periodic and finite-dimensional ones, can be found in \cite{E1b, E2}, and the search for new
solvable classes continues; the reader who has followed the course this far is well equipped
to join it.

\lsection{Exercises}

\ms\no 1) Prove the classification of homogeneous Poisson-finite measures stated at the
beginning of the lecture. (Hint: homogeneity determines $\mu$ on all intervals from its
values on $(0,1)$ and $(-1,0)$; show first that $\mu$ has no point masses.)

\ms\no 2) Show that the density $|x|^{1+2\nu}$ defines a Poisson-finite measure if and only if
$-1<\nu<0$, and check that for such $\nu$ the corresponding measure is quasi-homogeneous of
order $\nu$. For which $\nu$ is the measure a Frostman measure, and what does the type
alternative of Lecture 11 say about it?

\ms\no 3) Deduce from the scaling relation $k_t(z)=tk_1(tz)$ and formula \eqref{hrec} that
$h_{11}$ is constant for any homogeneous spectral measure, and compute the constant for the
Lebesgue measure. Derive the analogous statement in the quasi-homogeneous case.

\ms\no 4) Verify the Bessel example for $m=0$: check that $F_{-1/2}(\l)=\sqrt{\frac2\pi}\cos\l$
and $\l F_{1/2}(\l)=\sqrt{\frac2\pi}\sin\l$ (up to normalization), and that the formulas of
the last section reduce to the free matrizant of exercise 2 of Lecture 15.

\ms\no 5) Using exercise 1 of Lecture 15, transform the canonical system with
$\HH={\rm diag}(t^m,t^{-m})$ into a Schr\"odinger equation and identify the resulting potential
with the Bessel potential of Lecture 4. (This is a challenge exercise; consult Section 6 of
\cite{Rem15} if needed.)

\ms\no 6) Solve the Riemann-Hilbert problem with a constant coefficient $F\equiv c>0$ on
$[-t,t]$ explicitly, and observe how the solution $X_t$ develops power-type singularities at
the endpoints $\pm t$. What do these singularities correspond to in the behavior of the
kernels $k_t$?

\newpage

\lecture{17}{Cases of pointwise convergence of the non-linear Fourier transform}

\ms\no The Krein-de Branges theory of the last two lectures has a scattering counterpart,
the non-linear Fourier analysis of Dirac systems, a non-linear analogue of the setting of
the classical theorems of Carleson \cite{Car16} and Hunt on the convergence of Fourier
series and integrals. In this final lecture, based on \cite{Cases16, NM16}, we give
examples of simpler convergence and maximal statements for the non-linear Fourier transform
of an $L^2$-potential, which follow from the resonance approach developed in
\cite{P16}. We will see the objects of this course -- Hermite-Biehler functions,
meromorphic inner functions, de Branges chains and Poisson-summable logarithms -- working
together in a problem that connects harmonic analysis with the scattering theory of
differential operators.

\lsection{Dirac systems and scattering data}

\ms\no Consider the Dirac system on $\R_+$,
\begin{equation}\label{DirS}
\Omega\dot X=zX-QX,\qquad
\Omega=\begin{pmatrix}0&1\\-1&0\end{pmatrix},\quad
Q(t)=\begin{pmatrix}0&f(t)\\f(t)&0\end{pmatrix},
\end{equation}
with a real-valued potential $f\in L^2(\R_+)$, the spectral parameter $z\in\C$ and
$X(t,z)=(u,v)^T$. This is a relative of the canonical systems of Lecture 15, with the spectral
parameter entering in the free term; the reader may consult \cite{Den16, Rem16} for the
translation between the two languages. Let $E(t,z)$ and $\ti E(t,z)$ denote the functions
$u-iv$ of the solutions with the Neumann, $X(0,z)=(1,0)^T$, and the Dirichlet,
$X(0,z)=(0,1)^T$, initial conditions. For each $t$ these are Hermite-Biehler functions of
exponential type $t$; in the free case $f\equiv0$, $E=e^{-itz}$ and $\ti E=-ie^{-itz}$. The
corresponding chains of de Branges spaces carry spectral measures $\mu,\ti\mu$ with a.c.\
densities $w,\ti w$, and for $f\in L^2(\R_+)$ the measures satisfy the Szeg\"o condition
$$\log w,\ \log\ti w\in L^1(\Pi)$$
(see \cite{Den16}) -- the Poisson-summable logarithms of Krein's theorem from Lecture 8 thus
reappear as the exact spectral signature of square-summable potentials.

\ms\no The scattering data of the system are the functions
$$a(t,z)=\frac{e^{itz}}2\big(E+i\ti E\big),\qquad b(t,z)=\frac{e^{itz}}2\big(E-i\ti E\big),$$
which satisfy $|a|^2-|b|^2=1$ on $\R$, with $a$ outer in $\C_+$ and $a(t,0)>0$. The
\textit{non-linear Fourier transform} (NLFT) of the potential is the ratio
$$\nlhat f_T(z)=\frac{b(T,z)}{a(T,z)},\qquad \nlhat f(z)=\frac{b(\infty,z)}{a(\infty,z)}.$$
The name is explained by linearization: for small potentials the equation satisfied by
$(a,b)$ gives $b(T,z)\approx\int_0^Tf(t)e^{2izt}dt$ (exercise 2), so $\nlhat f$ is a
non-linear deformation of the Fourier transform of $f$. The role of Plancherel's theorem is
played by the non-linear Parseval identity
\begin{equation}\label{nlParseval}
\big|\big|\log|a(t,\cdot)|\big|\big|_{L^1(\R)}=\frac\pi2\,||f||^2_{L^2(0,t)},
\end{equation}
where the left-hand side is a genuine norm since $\log|a|\geq0$ on $\R$. For the history of
the subject and its discrete counterpart, the Fourier analysis of orthogonal polynomials on
the unit circle, we refer to the lecture notes of Tao and Thiele \cite{TT16} and to
\cite{Den16}.

\ms\no It is known that $a(t,\cdot)$ and $b(t,\cdot)$ converge as $t\to\infty$ normally in
$\C_+$ and in measure on compacts of $\R$, so that $\nlhat f_T\to\nlhat f$ in measure; see
\cite{P16}. For potentials in $L^p(\R_+)$, $1\leq p<2$, almost everywhere convergence of
$\nlhat f_T$ follows from the work of Christ and Kiselev \cite{CK16}; a Carleson-type
theorem for a Cantor group model of the transform was proved by Muscalu, Tao and Thiele
\cite{MTT16}. The statements presented below concern real potentials in $L^2(\R_+)$ and
follow from the resonance approach of \cite{P16}.

\lsection{Resonances and universality}

\ms\no The bridge to the earlier lectures is the family of meromorphic inner functions
$$\theta(t,z)=\frac{E^\#(t,z)}{E(t,z)},\qquad E^\#(z)=\overline{E(\bar z)},$$
familiar from Lecture 3. The zeros of $E(t,\cdot)$, which all lie in $\C_-$, are the
\textit{resonances} of the system \eqref{DirS} restricted to $(0,t)$. Differentiating in $t$
one finds that $\theta$ satisfies the Riccati equation
\begin{equation}\label{Ricc}
\frac{\pa}{\pa t}\theta=2iz\theta+f(t)\big(1-\theta^2\big),
\end{equation}
which describes the dynamics of the resonances: a resonance moves only under the action of
the potential, at the speed $|f(t)|/|\theta_z|$.

\ms\no The analytic core of the resonance approach of \cite{P16} is a collection of
universality-type approximations:
for a.e.\ $s\in\R$ there exists $C(t)\to\infty$ such that on the boxes
$Q(s,C/t)=\{|\Re(z-s)|\leq C/t,\ |\Im z|\leq C/t\}$ the function $E(t,\cdot)$ is, up to an
$o(1)$ error, an elementary function. Two regimes occur. If no resonance is present in the
box, then
$$E(t,z)\approx\frac{\alpha(s,t)}{\sqrt{w(s)}}\,e^{-itz}$$
for a unimodular $\alpha(s,t)$: locally the system looks free, with the density $w(s)$ as the
only trace of the potential. If a resonance $z(t)=x(t)-iy(t)$ is present in the box, then
$$E(t,z)\approx\frac{\alpha(s,t)\,\gamma(ty(t))}{\sqrt{w(s)}}\,\sin\big[t(z-z(t))\big],\qquad
\gamma(p)=\frac{\sqrt2}{\sqrt{\sinh(2p)}},$$
a sine with one complex zero; the quantity $ty(t)$, the \textit{depth} of the resonance,
measures how strongly it distorts the local picture. Analogous approximations hold for the
Dirichlet family $\ti E$, with cosines, and the parameters of the two families are rigidly
coupled through the determinant relation $E\ti E^\#-\ti EE^\#\equiv2i$. Here $\gamma$ is a
smooth version of the quantities we met in Lecture 4 in the residue formula for spectral
measures; we refer to \cite{Cases16}, where the statements of \cite{P16} used below are
reproduced with proofs.

\lsection{The price of a resonance}

\ms\no The key mechanism behind the convergence results is that resonances are expensive. If
the system has a resonance of bounded depth near $s$ at time $t$, then on the time window
$(t,2t)$ the transition scattering coefficient $a_{t\to2t}$ must be noticeably large near
$s$; by the non-linear Parseval identity \eqref{nlParseval}, applied to the restricted
potential, this consumes a definite portion of $||f||^2_{L^2((t,2t))}$. Summing over a cover
one obtains a pricing estimate: for suitable compacts $S'$ of full density in $\R$,
$$\begin{gathered}\big|\{s\in S'\ |\ \textrm{a resonance of depth}\leq C\textrm{ is present near }s
\textrm{ at time }t\}\big|\ \leq\\
\leq\ c(C,S')\,||f||^2_{L^2((t,2t))}.\end{gathered}$$
Integrating in $t$ gives an occupation bound: for a.e.\ $s$,
$$\sum_{n\geq0}2^{-n}\big|\{t\in[2^n,2^{n+1}]\ |\ \textrm{a resonance is present near }s
\textrm{ at time }t\}\big|<\infty,$$
i.e.\ near almost every point resonances can occupy only a summably small proportion of each
dyadic block of times. This Borel-Cantelli mechanism converts the $L^2$ smallness of the
potential tails into pointwise information -- a strategy the reader may compare with the use
of the Poisson-summable weights in Lectures 8 and 9.

\lsection{Three cases of convergence}

\ms\no The first result concerns convergence along sequences adapted to the local $L^2$ norm
of the potential. For an increasing sequence $\{t_k\}$, $t_k\uparrow\infty$, let
$N(t)=\#\{k\ |\ t/2\leq t_k<t\}$ be its octave counting function.

\begin{theorem}[\cite{Cases16}]\label{NLFTa}
Let $f\in L^2(\R_+)$ be real and let $\{t_k\}$ satisfy
$$\int_0^\infty f(t)^2N(t)\,dt=\sum_k\int_{t_k}^{2t_k}f^2<\infty.$$
Then $\nlhat f_{t_k}(s)\to\nlhat f(s)$ as $k\to\infty$ for a.e.\ $s\in\R$.
\end{theorem}

\begin{corollary}\label{NLFTlac}
For every real $f\in L^2(\R_+)$ and every lacunary sequence, $t_{k+1}/t_k\geq q>1$,
$$\nlhat f_{t_k}(s)\to\nlhat f(s)\qquad\textrm{ for a.e. }s\in\R.$$
\end{corollary}

\ms\no Indeed, for a lacunary sequence $N\leq\log_q2+1$, and the condition of the theorem
holds automatically (exercise 4). Since the dyadic masses $\int_{2^n}^{2^{n+1}}f^2$ tend to
zero, admissible sequences can in fact be chosen with relative gaps tending to $0$.

\ms\no The second and third results give full convergence, over $T\to\infty$, under
conditions on the $L^1$ masses of the potential.

\begin{theorem}[\cite{Cases16}]\label{NLFTb}
Let $f\in L^2(\R_+)$ be real and suppose that its dyadic $L^1$ masses are bounded:
$$\sup_{n\geq0}\int_{2^n}^{2^{n+1}}|f|<\infty.$$
Then $\nlhat f_T(s)\to\nlhat f(s)$ as $T\to\infty$ for a.e.\ $s\in\R$.
\end{theorem}

\begin{theorem}[\cite{Cases16}]\label{NLFTmw}
Let $f\in L^2(\R_+)$ be real and suppose that
$$\int_0^\infty f(t)^2\left(\int_{t/2}^t|f(\tau)|\,d\tau\right)dt<\infty.$$
Then $\nlhat f_T(s)\to\nlhat f(s)$ as $T\to\infty$ for a.e.\ $s\in\R$.
\end{theorem}

\ms\no The proofs combine Theorem \ref{NLFTa} with an elementary mobility bound for the
scattering data: if $\check\rho_{t\to t'}$ denotes the modulus of the NLFT of the potential
restricted to $(t,t')$, then
\begin{equation}\label{mob}
\big|\nlhat f_{t'}(s)-\nlhat f_t(s)\big|\leq\frac{2\check\rho_{t\to t'}}{1-\check\rho_{t\to t'}},
\qquad
\check\rho_{t\to t'}\leq\tanh\left(\int_t^{t'}|f|\right),
\end{equation}
which follows from the M\"obius action of the transfer matrix on the scattering data and a
Gronwall-type estimate (exercise 3). One constructs a grid $\{t_j\}$ subordinate to the
$|f|$-mass, applies Theorem \ref{NLFTa} along the grid and \eqref{mob} between consecutive
grid points; the mass conditions of Theorems \ref{NLFTb} and \ref{NLFTmw} guarantee that the
grid can be chosen with both small steps in $|f|$-mass and summable $L^2$ pricing.

\lsection{A maximal estimate}

\ms\no In the linear theory, the theorems of Carleson and Hunt are inseparable from
weak-type estimates for the maximal operator of the partial sums: convergence is derived
from bounds on the associated maximal function. It is natural to accompany the convergence
statements of the previous section with a maximal estimate for the scattering data.
Following \cite{NM16}, we consider the quantity controlled by the non-linear Parseval
identity,
$$u(t,s)=\log|a(t,s)|,$$
and measure its deviation from the limiting level
$$A(s)=\log\left(\frac12\sqrt{\frac1{w(s)}+\frac1{\ti w(s)}+2}\right).$$
The meaning of $A(s)$ is explained by the \textit{elevation identity}
$$|a(t,s)|^2=\frac{|E(t,s)|^2+|\ti E(t,s)|^2+2}4,\qquad s\in\R,$$
an exact pointwise consequence of the determinant relation $E\ti E^\#-\ti EE^\#\equiv2i$
(exercise 8): $e^{2A(s)}$ is the value of $|a(t,s)|^2$ produced by the limiting moduli
$|E(t,s)|^2\to w^{-1}(s)$, $|\ti E(t,s)|^2\to\ti w^{-1}(s)$ at resonance-free times.

\ms\no The estimate is proved on \textit{regular sets} $\RR(T,\eta_0,\l)$, defined in
\cite{NM16} for $\eta_0\in(0,\frac12]$ and $\l>0$ as the sets of points
$s\in\{\min(w,\ti w)\geq\eta_0\}$ at which three classical maximal functions of the
spectral data -- the non-tangential maximal functions of $\log w$, of the outer function
$G$, $|G|^2=w$, and of the measure $\mu$, truncated at heights comparable to $1/t$,
together with their Dirichlet counterparts -- are suitably small for all $t\geq T/32$. The
regular sets increase in $T$ and, for every fixed $(\eta_0,\l)$, exhaust almost all of the
set $\{\min(w,\ti w)\geq\eta_0\}$ as $T\to\infty$.

\begin{theorem}[\cite{NM16}]\label{NLFTmax}
Let $f\in L^2(\R_+)$ be real and satisfy the dyadic mass restriction of Theorem
\ref{NLFTb},
$$M=\sup_{n\geq0}\int_{2^n}^{2^{n+1}}|f|<\infty,$$
and put $M'=\max\left(M,\int_0^1|f|\right)$. There exist absolute constants $K$ and $N$
such that for all $\eta_0\in(0,\tfrac12]$, $\l>0$ and $T\geq1$,
$$\begin{gathered}
\left|\left\{s\in\RR(T,\eta_0,\l)\ \Big|\ \sup_{t\geq T}\big|\log|a(t,s)|-A(s)\big|>\l\right\}\right|\ \leq\\
\leq\ K\,(M'+1)\left(\eta_0\min(\l,1)\right)^{-N}\,\frac{||f||^2_{L^2((T/32,\infty))}}\l\, .
\end{gathered}$$
\end{theorem}

\ms\no The right-hand side tends to zero as $T\to\infty$. Combined with the exhaustion
property of the regular sets, the theorem gives the almost everywhere convergence of the
moduli of the scattering data:

\begin{corollary}[\cite{NM16}]\label{NLFTmaxcor}
Under the dyadic mass restriction, for a.e.\ $s\in\R$,
$$\log|a(t,s)|\ \to\ A(s)\qquad\textrm{ as }t\to\infty,$$
and, through the identity $|\nlhat f_t|^2=1-|a(t,\cdot)|^{-2}$,
$$\big|\nlhat f_t(s)\big|\ \to\ \sqrt{1-e^{-2A(s)}}=\big|\nlhat f(s)\big|\, .$$
\end{corollary}

\ms\no The proof of Theorem \ref{NLFTmax} is a quantitative version of the resonance
approach described earlier in this lecture: a deviation of $u(t,s)$ from $A(s)$ of size
$\l$, in either direction, forces a resonance of depth exponentially small in $\l$ near
$s$; each such resonance is paid for by the $L^2$ mass of the potential on the forward
octave, through the two-endpoint estimate and the non-linear Parseval identity; and a
first-entry argument sums the prices over a grid of times subordinate to the $|f|$-mass.
The dyadic mass restriction enters through the Lipschitz property
$$\big|\log|a(t_2,s)|-\log|a(t_1,s)|\big|\ \leq\ \int_{t_1}^{t_2}|f|,$$
which reduces the maximal function over all times to its lacunary version, as in exercise
6. Let us also note that the restriction of the maximal estimate to the modulus of $a$ is
essential: as shown by Denisov \cite{Den216}, there exist potentials for which
$\arg a(t,s)$ diverges as $t\to\infty$ at every $s$, while the moduli of the scattering
data converge; the phase $\arg a(t,\cdot)$, the harmonic conjugate of $\log|a(t,\cdot)|$,
belongs to the conjugate-function layer of the problem.

\lsection{Convergence in density}

\ms\no For a general $f\in L^2(\R_+)$ the occupation bound alone gives an unconditional
statement: convergence holds outside a small exceptional set of times.

\begin{theorem}[\cite{Cases16}]\label{NLFTd}
Let $f\in L^2(\R_+)$ be real. For a.e.\ $s\in\R$ there is a measurable set
$\Xi(s)\subset\R_+$ with
$$\sum_{n\geq0}2^{-n}\big|\Xi(s)\cap[2^n,2^{n+1}]\big|<\infty,$$
such that, as $t\to\infty$ outside $\Xi(s)$,
$$|E(t,s)|\to\frac1{\sqrt{w(s)}},\qquad|\ti E(t,s)|\to\frac1{\sqrt{\ti w(s)}},$$
$$\begin{gathered}|a(t,s)|^2\to\frac14\left(\frac1{w(s)}+\frac1{\ti w(s)}+2\right),\\
|b(t,s)|^2\to\frac14\left(\frac1{w(s)}+\frac1{\ti w(s)}-2\right),\end{gathered}$$
and $\nlhat f_t(s)$ approaches a two-point set $\{r_+(s),r_-(s)\}$, computed explicitly from
$w(s)$ and $\ti w(s)$, which contains the limit $\nlhat f(s)$; in particular
$|r_+|=|r_-|$ and
$$\big|\nlhat f_t(s)\big|\to\big|\nlhat f(s)\big|\qquad\textrm{ as }t\to\infty\textrm{ outside }\Xi(s).$$
\end{theorem}

\ms\no Thus for every square-summable potential all the moduli of the scattering data
converge at a.e.\ $s$ along times of full density, and $\nlhat f_t(s)$ itself approaches
the two-point set $\{r_+(s),r_-(s)\}$, whose elements differ only in the sign of a square
root. Under the hypotheses of Theorems \ref{NLFTb}-\ref{NLFTmw} the dichotomy resolves
and no exceptional set is needed.

\lsection{Exercises}

\ms\no 1) Solve the system \eqref{DirS} for $f\equiv0$ and verify all the formulas of the
free case: $E=e^{-itz}$, $\ti E=-ie^{-itz}$, $a\equiv1$, $b\equiv0$, $\nlhat f\equiv0$, and
$w=\ti w=$ const.

\ms\no 2) Derive from \eqref{DirS} the system satisfied by $(a,b)$:
$$\frac{\pa a}{\pa t}=e^{-2izt}f(t)\,b,\qquad\frac{\pa b}{\pa t}=e^{2izt}f(t)\,a,$$
and show that the first Picard iteration gives $b(T,z)\approx\int_0^Tf(t)e^{2izt}dt$.
Verify that in this linear approximation, using $\log|a|\approx\frac12|b|^2$, the non-linear
Parseval identity \eqref{nlParseval} becomes the classical Plancherel theorem (up to
normalization).

\ms\no 3) Prove the second inequality in \eqref{mob}: from the system of exercise 2 for the
restricted potential, show that $|\check a|+|\check b|\leq e^F$ and
$|\check a|-|\check b|\geq e^{-F}$ with $F=\int_t^{t'}|f|$, and conclude
$\check\rho_{t\to t'}\leq\tanh F$.

\ms\no 4) Show that for a lacunary sequence, $t_{k+1}/t_k\geq q>1$, the octave counting
function satisfies $N\leq\log_q2+1$, and deduce Corollary \ref{NLFTlac} from Theorem
\ref{NLFTa}.

\ms\no 5) For potentials with regular behavior at infinity, $|f(t)|\leq C(1+t)^{-\alpha}$
on $\R_+$ with $\alpha>2/3$, show that
$$f(t)^2\int_{t/2}^t|f|\lesssim(1+t)^{1-3\alpha},$$
and deduce from Theorem \ref{NLFTmw} that $\nlhat f_T(s)\to\nlhat f(s)$ as $T\to\infty$
for a.e.\ $s\in\R$.

\ms\no 6) Suppose that the dyadic $L^1$ masses tend to zero: $M_n=\int_{2^n}^{2^{n+1}}|f|\to0$.
Combining Corollary \ref{NLFTlac} along $t_k=2^k$ with the bound \eqref{mob} inside each
dyadic block, show that $\nlhat f_T(s)\to\nlhat f(s)$ for a.e.\ $s$. (This is the easy
subclass of Theorem \ref{NLFTb}; the content of Theorems \ref{NLFTb} and \ref{NLFTmw} lies
in the regime where $M_n$ does not tend to zero.)

\ms\no 7) Using the Riccati equation \eqref{Ricc}, verify that a curve $z(t)$ of zeros of
$\theta(t,\cdot)$ satisfies $z'(t)=-f(t)/\theta_z(t,z(t))$, and explain why resonances do
not move on the intervals where $f=0$.

\ms\no 8) Prove the elevation identity: on $\R$, where $E^\#=\bar E$ and
$\ti E^\#=\bar{\ti E}$, the determinant relation $E\ti E^\#-\ti EE^\#\equiv2i$ becomes
$\Im\big(E\bar{\ti E}\big)\equiv1$, and therefore
$$4|a(t,s)|^2=\big|E+i\ti E\big|^2=|E|^2+|\ti E|^2+2\,\Im\big(E\bar{\ti E}\big)
=|E|^2+|\ti E|^2+2 .$$

\ms\no 9) Deduce Corollary \ref{NLFTmaxcor} from Theorem \ref{NLFTmax}, letting
$T\to\infty$ and using the exhaustion property of the regular sets. Verify that the
limiting value of $|\nlhat f_t|^2$ agrees with the common modulus of the two-point set
$\{r_\pm\}$ of Theorem \ref{NLFTd}.

\lsection{Finishing remarks}

\ms\no The main goal of this course was to introduce the reader to a classical, yet still
rapidly developing, area of analysis formed by the applications of complex function theory
to harmonic analysis, spectral and scattering theory, centered around the branch of
mathematics called the Uncertainty Principle in Harmonic Analysis. As we saw, this area is
full of long standing problems, deep theorems and exciting applications. Despite a century
of active research by many outstanding analysts, most fundamental problems in this field
are wide open. What will the answer be if one replaces the interval with a union of
intervals in the statement of the Beurling-Malliavin Problem? What if one replaces the
interval with a more general set in the Type or the Hochstadt-Liberman Problems? These
natural questions stand open, some for many decades, despite considerable efforts by the
analytic community.

\ms\no The Uncertainty Principle is definitely much wider than the completeness, spectral
and scattering problems discussed here. A natural next step after completeness are
questions on excess and deficiency of exponential systems, polynomials or special
functions, problems on frames, Riesz bases, sampling and interpolation. The reader
interested in further material in these directions may consult the classic survey of
Redheffer \cite{Red16}, a more modern survey by Khabibullin \cite{Kh16} (in Russian), and
the books by Nikolski \cite{Ni16} and Seip \cite{Seip16}.

\ms\no The last three lectures could offer only a first acquaintance with the Krein-de
Branges theory of Hilbert spaces of entire functions, which translates spectral problems
for differential operators into problems of complex and harmonic analysis, and with its
scattering counterpart, the non-linear Fourier analysis of Dirac systems. The basics of the
theory and further references can be found in \cite{DM16}, \cite{Rem16} and \cite{Rom16};
the complex analytic part is covered in de Branges' book \cite{dB16}; a modern account will
(hopefully) soon appear in \cite{MPS16}. The Krein-de Branges theory experiences a new peak
of popularity in recent years, partially due to the newfound connections to spectral
problems, the non-linear Fourier transform, number theory and the Riemann Hypothesis. The
reader who has worked through these notes is well prepared to take part in this
development.

\end{document}